\documentclass[11pt,oneside]{amsart}
\usepackage[latin1]{inputenc}
\usepackage{amsmath,amsfonts,amsgen,amstext,amsbsy,amsopn,amsthm, amssymb}
\usepackage{multirow, verbatim, cancel, mathabx}
\usepackage{bbold}
\usepackage{enumitem}
\usepackage[T1]{fontenc}
\usepackage[colorlinks=true,hyperindex=true]{hyperref}
\usepackage{tikz}
\usetikzlibrary{calc}
\usepackage{verbatim}
\usetikzlibrary{arrows,chains,matrix,positioning,scopes}
\usepackage{mathrsfs}
\usepackage{color,soul}
\usepackage{xcolor}
\usepackage{tabulary}
\usepackage{pgfplots}
\usepackage{mathrsfs}
\usepackage{bbm}
\usepackage{dsfont}
\usepackage{bm}
\usepackage{mathtools}
\usepackage{float}

\DeclarePairedDelimiter\floor{\lfloor}{\rfloor}
\usepackage{arydshln}
\usepackage{color,soul}
\usepackage{subcaption}

\usepackage{float}
\usepackage{flafter}

\usepackage{array}
\newcolumntype{R}[1]{>{\raggedleft\arraybackslash}p{#1}}

\usepackage{tabu} 
\usepackage{booktabs}

\usepackage[mathscr]{euscript}

\DeclareSymbolFont{rsfs}{U}{rsfs}{m}{n}
\DeclareSymbolFontAlphabet{\mathscrsfs}{rsfs}

\makeatletter

\usepackage{tkz-fct} 
\usepackage{xcolor}

\usepackage{tkz-fct} 
\usepackage{pgfplots}
\usepackage{setspace}

\usetikzlibrary{matrix}
\usepackage{tikz}

\numberwithin{equation}{section}

\newcounter{smallarabics}

\newcounter{smallroman}
\newenvironment{romanenumerate}
{\begin{list}{{\normalfont\textrm{(\roman{smallroman})}}}
  {\usecounter{smallroman}\setlength{\itemindent}{0cm}
   \setlength{\leftmargin}{5ex}\setlength{\labelwidth}{4ex}
   \setlength{\topsep}{0.75\parsep}\setlength{\partopsep}{0ex}
   \setlength{\itemsep}{0ex}}}
{\end{list}}

\newcommand{\ben}{\begin{romanenumerate}}  
\newcommand{\een}{\end{romanenumerate}}  

\newtheorem{theoreme}{theorem }[section]
\newtheorem{theorem}[theoreme]{Theorem}
\newtheorem{proposition}[theoreme]{Proposition}
\newtheorem{Lemma}[theoreme]{Lemma}

\newtheorem{corollary}[theoreme]{Corollary}
\newtheorem{remark}{Remark}[section]

\newcommand{\Pp}{P^{\perp}}

\newcommand{\tH}{\theta(\mathfrak{H})}

\newcolumntype{L}{>{\centering\arraybackslash}m{3cm}}

\newcommand\nn\nonumber
\renewcommand\leq\varleq
\renewcommand\geq\vargeq

\newcommand{\1}{\mathbb 1}

 \newcommand{\R}{\mathbb{R}}
 \newcommand{\N}{\mathbb{N}}
\newcommand{\Z}{\mathbb{Z}} \newcommand{\C}{\mathbb{C}}

\renewcommand{\i}{\mathrm{i}}

\renewcommand{\epsilon}{\varepsilon}

\newcommand\blue[1]{\textcolor{blue}{#1}}

\usepackage{tikz-cd}

\usepackage{xcolor}
\usepackage{dcolumn}
\usepackage{tabu}

\newcolumntype{A}{D{.}{.}{2.3}}

\usepackage{tikz,tkz-tab}
\usepackage{pgfplots}
\pgfplotsset{compat=1.11}
\usetikzlibrary{positioning}
\makeatletter
      \def\@setcopyright{}
      \def\serieslogo@{}
      \makeatother
\begin{document}

\author{Sylvain Gol\'enia and Marc-Adrien Mandich}
   \address{Univ. Bordeaux, CNRS, Bordeaux INP, IMB, UMR 5251,  F-33400, Talence, France}
   \email{sylvain.golenia@math.u-bordeaux.fr}
      \address{Independent researcher, Jersey City, 07305, NJ, USA}
	\email{marcadrien.mandich@gmail.com}
   

   \title[LAP for discrete Schr\"odinger operator]{Bands of pure a.c.\ spectrum for lattice Schr{\"o}dinger operators with a more general long range condition. Part I}

   \begin{abstract}
   Commutator methods are applied to get limiting absorption principles for the discrete standard and Molchanov-Vainberg Schr\"odinger operators, $\Delta+V$ and $D+V$ on $\ell^2(\Z^d)$, with emphasis on $d=1,2,3$. Considered are electric potentials $V$ satisfying a long range condition of the type: $V-\tau_j ^{\kappa}V$ decays appropriately at infinity for some $\kappa \in \N$ and all $1 \leq j \leq d$, where $\tau_j ^{\kappa} V$ is the potential shifted by $\kappa$ units on the $j^{\text{th}}$ coordinate. More comprehensive results are obtained for small values of $\kappa$, e.g.\ $\kappa =1,2,3,4$. We work in a simplified framework in which the main takeaway appears to be the existence of bands where a limiting absorption principle holds, and hence pure absolutely continuous (a.c.)\ spectrum exists. Other decay conditions at infinity for $V$ arise from an isomorphism between $\Delta$ and $D$ in dimension 2.  Oscillating potentials are examples in application.
   \end{abstract}

%

 

\maketitle
\hypersetup{linkbordercolor=black}
\hypersetup{linkcolor=blue}
\hypersetup{citecolor=blue}
\hypersetup{urlcolor=blue}
\tableofcontents

\section{Introduction}

The spectral and scattering theory of quantum systems on the lattice $\Z^d$ has been done primarily using the standard definition of the Laplacian
\begin{equation}
\label{def:std}
\Delta  \equiv \Delta_{\mathrm{std}}:=  \sum _{j=1}^d \Delta_j, \quad \text{where } \Delta_j:=  \frac{1}{2} (S_j + S^*_j), \quad \text{defined on} \ \mathscr{H}:= \ell^2(\Z^d).
\end{equation}
Here $S_j = S_j^1$ and $S^*_j = S^{-1}_j$ are the shifts to the right and left respectively on the $j^{th}$ coordinate. So $(S_j u)(n) = u(n_1,\ldots,n_j-1,\ldots,n_d)$ for $u \in \mathscr{H}$, $n=(n_1,\ldots,n_d) \in \Z^d$. Set $|n|^2:= n_1^2 + ...+n_d^2$. In \cite{MV}, Molchanov and Vainberg proposed an alternative Laplacian on $\mathscr{H}$ defined by
\begin{equation}
\label{def:MV22}
D \equiv D_{\mathrm{MV}}:= \prod_{j=1}^d \Delta_j = 2^{-d} \sum_{\nu \in \Sigma} S_1^{\nu_1} S_2^{\nu_2} \cdot \cdot \cdot S_d^{\nu_d}, \quad \text{where } \Sigma:=  \{-1,1\}^d, \nu = (\nu_1, \ldots, \nu_d).
\end{equation}
In the literature, this difference operator has been coined the \emph{Molchanov-Vainberg} Laplacian. A Fourier transformation shows that the spectra of $\Delta_j$, $\Delta$ and $D$ are purely absolutely continuous (a.c.) and $\sigma(\Delta_j) \equiv [-1,1]$, $\sigma(\Delta) = [-d,d]$ and $\sigma(D) =[-1,1]$. Moreover $d^{-1}\Delta$ and $D$ are equal when $d=1$ and isomorphic when $d=2$ (see below), but we believe they are fundamentally different for $d\geq 3$. In spite of this difference, a common point is convergence to the continuous Laplacian. Namely, in \cite{NT} it was proved that $\Delta$ converges to the continuous Laplacian on $\R^d$ in the norm resolvent sense, when the mesh size of the lattice goes to zero. It is not very hard to prove that this is also true for $D$, for all dimensions, see Appendix \ref{convergence_appendix}.

An interesting property of $D$ is the decay of its Green's function. It satisfies $G(0,n,E):= \langle \delta_0, (D - E - \i 0) ^{-1} \delta_n \rangle = O ( | n | ^{- \frac{d-1}{2}} )$ as $|n| \to \infty$ for all $E \in (-1,1)\setminus \{0 \}$, whereas such a rate of decay for the Green's function of $\Delta$ does not hold in dimensions $d \geqslant 3$ for energies in the middle part of its spectrum. This statement was conjectured in \cite{MV}, see also \cite{SV}, and proven in \cite{P}. The literature on the Molchanov-Vainberg Laplacian is thin. Here we study its spectral properties by means of commutator methods and obtain limiting absorption principles (LAP) in a spirit similar to that of \cite{BSa}, \cite{GM2} or \cite{IK}, for the discrete Schr\"odinger operators based on the standard definition of the Laplacian. But unlike these articles, here we consider a more general long\blue{-}range condition on the perturbation $V$, which we describe shortly. 

Let $V$ model a discrete electric potential and act pointwise, i.e.\ $(Vu)(n) = V(n) u(n)$, for $u \in \mathscr{H}$. Let $\kappa = (\kappa_j)_{j=1}^d \in (\N^*)^d$ (non-zero positive integers) be given. The potential shifted by $\kappa_j$ units on the $j^{th}$ coordinate is again a multiplication operator defined by
\begin{equation*}
(\tau_j ^{\kappa_j} V) u(n):= V(n_1,\ldots,n_j-\kappa_j,\ldots n_d)u(n), \quad (\tau_j^{-\kappa_j} V) u(n):= V(n_1,\ldots,n_j+\kappa_j,\ldots, n_d) u(n).
\end{equation*} 
We are interested in potentials $V$ satisfying a non-radial condition at infinity of the form 
\begin{equation}
\label{generalLR condition}
n_j (V - \tau_j ^{\kappa_j} V)(n) = O(g(n)), \quad \forall  1 \leq j \leq d,
\end{equation}
where $g(n)$ is a (radial) function which goes to zero at infinity. This type of condition arises rather naturally in a wider framework of applied Mourre theory/commutator methods on a square lattice which we develop here and in \cite{GM3}. Let us also point out that condition \eqref{generalLR condition} is close to a summability condition $\sum_{ n \in \Z^d} |(V - \tau_j ^{\kappa_j} V)(n)| < \infty$, especially if $V$ is radial. If all the $\kappa_j$'s are equal, which will often be our assumption, we write $\kappa = \kappa_j$ in short. So for example the notation $\kappa=2$ means $\kappa_j \equiv 2$. When the underlying Laplacian is $\Delta$, we refer to the case $\kappa = 1$ and $g(n) = |n|^{-\epsilon}$, $\epsilon > 0$, as the \textit{base case condition}. When the underlying Laplacian is $D$, we define the \textit{base case condition} to be $\kappa = 2$ and $g(n) = |n|^{-\epsilon}$, $\epsilon > 0$. For $\Delta$, the base case condition is treated in \cite{Ma1}: the result is that one has a LAP locally for $\Delta + V$ on $[-d,d] \setminus \{-d+2l: l=0,...,d\}$, and thereby absence of singular continuous spectrum on this set. One can then look to improve the base case conditions by either \begin{enumerate}
\item weakening the decay of $g(n)$ at $|n| \to \infty$, or
\item increasing the $\kappa_j$'s. 
\end{enumerate}
Our methods indicate that these 2 points can be analyzed independently. Article \cite{GM2} tackles the former point, by weakening $g(n) = |n|^{-\epsilon}$, $\epsilon >0$, to functions that decay logarithmically, e.g.\ $g(n) = \log ^{-q} ( |n|)$, $q >2$. In this article and \cite{GM3} we turn our attention to point (2). In addition to \eqref{generalLR condition} we always assume $V$ is real-valued and goes to zero at infinity. Thus $\sigma_{\mathrm{ess}}(D+V) = \sigma(D)$ and $\sigma_{\mathrm{ess}}(\Delta+V) = \sigma(\Delta)$, where $\sigma_{\text{ess}}(\cdot)$ is the essential spectrum. The question we attempt to answer, at least to some degree of generality, is: If we fix $V$ satisfying \eqref{generalLR condition} for some $\kappa$ and appropriate $g(n)$ (say $g(n) = |n|^{-\epsilon}$ for simplicity), what is the support of the pure a.c.\ spectra of $\Delta+V$ and $D+V$ ? What  LAPs can be derived and at what energies?

We will formulate LAPs according to 3 distinct commutator theories: 1) Mourre's original papers \cite{Mo1}, \cite{Mo2},  2) the Besov space approach of \cite{ABG}, \cite{BSa}, and 3) G\'erard's energy estimate approach \cite{G}, \cite{GJ2}. Novel commutator ideas have recently been advanced in \cite{AIIS} ; we do not probe them here. Let $\mathfrak{H}$ be a Hamiltonian on $\mathscr{H} = \ell^2(\Z^d)$. For a closed interval $I \subset \R$ let $I_{\pm}:= \{ z \in \C_{\pm}: \mathrm{Re} (z) \in I\}$, $\C_\pm:=\{z\in \C, \pm\mathrm{Im}(z)>0\}$. The LAPs are statements about the extension of the holomorphic maps 
\begin{equation*}
\label{LaP_generic}
I_{\pm} \ni z \mapsto (\mathfrak{H}-z)^{-1} \in \mathscr{B}(\mathcal{K}, \mathcal{K}^*),
\end{equation*}
to $I$ for some appropriate Banach space $\mathcal{K} \subset \mathscr{H}$. Here $\mathcal{K}^*$ is the antidual of $\mathcal{K}$, when $\mathscr{H}=\mathscr{H}^*$ by the Riesz isomorphism ; $\mathscr{B}(\mathcal{K}, \mathcal{K}^*)$ are the bounded operators from $\mathcal{K}$ to $\mathcal{K}^*$.

Let us present the framework and results. We start by fixing $\kappa = (\kappa_j)$ according to \eqref{generalLR condition}. The coordinate-wise position operators are 

\begin{equation*}
\label{notation_position}
(N_j u)(n):= n_j u(n), \quad \mathrm{Dom}[N_j] = \Big \{ u \in \mathscr{H} = \ell^2(\Z^d): \sum_{n \in \Z^d} |n_j u(n)|^2 < \infty \Big \}.
\end{equation*}
Commutator methods rely on a self-adjoint \emph{conjugate operator}. To handle condition \eqref{generalLR condition} it makes sense to consider a linear combination of conjugate operators, namely $\sum_{l \in \N^*} \rho_{l \kappa} A_{l\kappa}$, where $\rho_{l \kappa} \in \R$ and $A_{l \kappa} := \sum_{j=1}^d A_j (l \kappa)$, and where the $A_j (l \kappa)$ are the closure in $\mathscr{H}$ of
\small
\begin{equation}
\label{generatorDilations_0111}
A_{j}(l \kappa):= \frac{1}{2\i} \bigg [  \frac{l \kappa_j}{2}(S_j ^{l \kappa_j} + S_j^{-l \kappa_j}) + (S_j ^{l \kappa_j} - S_j^{-l \kappa_j}) N_j \bigg ]  = \frac{1}{4\i} \bigg[  (S_j ^{l \kappa_j} - S_j^{-l \kappa_j})N_j + N_j(S_j^{l \kappa_j}-S_j^{-l \kappa_j}) \bigg],
\end{equation} 
\normalsize
on the domain of compactly supported sequences. Each $A_{l \kappa}$ is self-adjoint by a straightforward adaptation of the case $\kappa = 1$ and $l=1$, see e.g.\ \cite{GGo}. $A_{l \kappa}$ is more compactly expressed in Fourier space. Let $\mathcal{F}: \mathscr{H} \to L^2([-\pi,\pi]^d,d\xi)$ be the Fourier transform 
\begin{equation}
\normalsize
(\mathcal{F} u)(\xi):=  (2\pi)^{-d/2} \sum \limits_{n \in \Z^d} u(n) e^{\i n \cdot \xi}, \quad  \xi=(\xi_1,\ldots,\xi_d).
\label{FourierTT}
\end{equation}
Then $A_{l \kappa}$ is unitarily equivalent to the self-adjoint realization of the operator
\begin{equation}
\label{A_fourier}
\mathcal{F} A_{l \kappa} \mathcal{F}^{-1}: = \frac{1}{2\i} \sum_{j=1} ^d \sin(l \kappa_j \xi_j) \frac{\partial}{ \partial \xi_j} + \frac{\partial }{\partial \xi_j} \sin(l \kappa_j \xi_j) .
\end{equation}
To our knowledge the commutator methods applied to the discrete Schr\"odinger operators have thus far only considered the above $A_{l\kappa}$ with $\kappa = 1$, and $l=1$, see e.g.\ \cite{BSa} and \cite{GM2}. Here we extend the usual generator of dilations (with some inspiration based on Nakamura's paper \cite{N}). In this article (part I) we work in the simplified framework where the linear combination $\sum_{l \in \N^*} \rho_{l \kappa} A_{l \kappa}$ contains just the first term, i.e. $\rho_{l \kappa} = 1$ if $l =1$ and $\rho_{l \kappa} =0$ otherwise. Thus, henceforth we stick to the simpler notation $A_{\kappa}:= \sum_{j=1}^d A_j$ where the $A_j$'s are 
\begin{equation}
\label{generatorDilations_011}
A_{j}:= \frac{1}{4\i} \bigg[  (S_j ^{\kappa_j} - S_j^{- \kappa_j})N_j + N_j(S_j^{\kappa_j}-S_j^{- \kappa_j}) \bigg].
\end{equation} 
In part II \cite{GM3}, which relies more heavily on numerical calculations with the computer, finite non-trivial combinations $\sum_{l \in \N^*} \rho_{l \kappa} A_{l \kappa}$ are considered. Results in Part II add to those in Part I.

Notation is set in order to treat the 2 Laplacians $\Delta$ and $D$ in parallel. It will be interesting to compare them afterwards. Let $(\mathfrak{D}, \mathfrak{H})$ be a generic pair of (free, perturbed) Hamiltonians. Let $\mathbb{A}$ be a generic conjugate operator. We treat the 2 cases $(\mathfrak{D}, \mathfrak{H}) = (D,D+V)$ or $(\Delta, \Delta+V)$, and initially consider $\mathbb{A} = \pm A_{\kappa}$ but later other options arise. At a basic level, the ability to obtain a LAP on an interval $I \subset \sigma(\mathfrak{H})$ depends directly on the ability to prove a \textit{strict Mourre estimate} for $\mathfrak{D}$ with respect to some conjugate operator $\mathbb{A}$ on this interval, that is to say, $\exists \gamma >0$ such that 
\begin{equation}
\label{mourreEstimate123}
E_I(\mathfrak{D}) [\mathfrak{D}, \i \mathbb{A}] _{\circ} E_{I} (\mathfrak{D}) \geq \gamma E_{I} (\mathfrak{D}),
\end{equation}
where $E_I(\mathfrak{D})$ is the spectral projection of $\mathfrak{D}$ onto $I$, and $[\mathfrak{D}, \i \mathbb{A}] _{\circ}$ means the extension of the commutator between $\mathfrak{D}$ and $\mathbb{A}$ to a bounded operator in $\mathscr{H}$. Consider:
\begin{equation}
\label{mu_set_def}
\boldsymbol{\mu}_{\mathbb{A}} ^{\pm} (\mathfrak{D}):= \{ E \in \sigma(\mathfrak{D}): \exists \ \mathrm{an \ open \ interval \ } I, I \ni E, \gamma >0, \mathrm{ \ s.t. \ } \eqref{mourreEstimate123} \mathrm{ \ holds  \ with \ } \pm \mathbb{A} \}.
 \end{equation} 
From the Mourre theory, $\boldsymbol{\mu}_{\mathbb{A}}^{\pm} (\mathfrak{D})$ are open subsets of $\R$. Now define $\boldsymbol{\mu}_{\mathbb{A}} (\mathfrak{D}):= \boldsymbol{\mu}_{\mathbb{A}} ^{+} (\mathfrak{D}) \cup \boldsymbol{\mu}_{\mathbb{A}} ^{-} (\mathfrak{D})$. 

It will be proved that choosing $\mathbb{A}$ equal to $+A_{\kappa}$ or $-A_{\kappa}$ yields a strict Mourre estimate on some intervals $I \subset \sigma(\mathfrak{D})$, for $\mathfrak{D} = D$ or $\Delta$. The $\pm$ sign depends on a number of things, namely $I$, $\kappa$, the dimension $d$, and whether $\mathfrak{D} = D$ or $\Delta$. A main goal of this article is to identify the sets $\boldsymbol{\mu}_{A_{\kappa}}(D)$ and $\boldsymbol{\mu}_{A_{\kappa}}(\Delta)$. Denote $U_{n}(\cdot)$ the Chebyshev polynomials of the second kind of order $n$. In Proposition \ref{reg_Prop_DD}, it is shown that
\begin{equation}
[\Delta, \i A_{\kappa} ]_{\circ} = \sum_{j=1} ^d (1-\Delta_j^2) U_{\kappa_j - 1} (\Delta_j), \quad \mathrm{and} \quad [D, \i A_{\kappa}]_{\circ} = \sum_{j=1} ^d \frac{D}{\Delta_j} (1-\Delta_j^2) U_{\kappa_j-1} (\Delta_j).
\label{MV:k323533}
\end{equation}
Since $D, \Delta$ and $\Delta_j$ are self-adjoint commuting operators, the problem of determining the sets $\boldsymbol{\mu}_{A_{\kappa}}(D)$ and $\boldsymbol{\mu}_{A_{\kappa}}(\Delta)$ translates into an extremization problem of polynomials in $d$ variables of degrees respectively $\kappa^*+1$ and $\kappa^*+d$ over a $d-1$ dimensional surface on a bounded domain, where $\kappa^*:= \max \kappa_j$. The complexity increases with $\kappa^*$ and $d$. The case of the multi-dimensional Standard Laplacian for $\kappa_j \equiv 1$ is well studied, see e.g.\ \cite{BSa} or \cite{GM2}. One has $\boldsymbol{\mu}_{A_{\kappa=1}}(\Delta) = [-d,d] \setminus \{-d+2l: l=0,...,d\}$. In dimension $1$, our new results are comprehensive: we prove $\boldsymbol{\mu}_{A_{\kappa}}(\Delta) = \boldsymbol{\mu}_{A_{\kappa}}(D) = [-1,1] \setminus \{\pm \cos(\pi j / \kappa), j=0,\ldots,\floor*{\kappa/2}  \}$, $\forall \kappa \in \N^*$. For $d\geq 2$, our rigorous results are mostly partial. Now we assume $\kappa \equiv \kappa_j$. Tables \ref{table:202} and \ref{table:101} collect such results for respectively $\Delta$, $\kappa = 2,3,4,$ and $D$, $\kappa =2,4,6$. For these values of $\kappa$, and so for the results in Tables \ref{table:202} and \ref{table:101}, the sets $\boldsymbol{\mu}_{A_{\kappa}}(\mathfrak{D})$ are symmetric about zero, i.e.\  $\boldsymbol{\mu}_{A_{\kappa}}(\mathfrak{D}) = -  \boldsymbol{\mu}_{A_{\kappa}}(\mathfrak{D})$ (this propriety is due to the bipartite property of $\Z^d$).

\begin{table}[H]
    \begin{tabular}{c|c|c|c} 
      $\kappa$ & $d=1$ & $d = 2$ & $d=3$    \\ [0.4em]
      \hline
      $2$  & $(0,1)$ & $ (1,2)$ & $ (2,3)$   \\ [0.4em]
     $3$  & $[0, 1) \setminus \{ \frac{1}{2} \}$  & $ \left(\frac{1}{2} \sqrt{\frac{1}{2}(5-\sqrt{7})},1 \right) \cup  \left( \frac{3}{2}, 2 \right)$ & $\left(\frac{5}{2},3 \right)$   \\ [0.4em]
      $4$  & $(0, 1) \setminus \{ \frac{1}{\sqrt{2}} \}$  & $\left( \sqrt{\frac{3}{2}-\frac{1}{\sqrt{5}}}, \sqrt{2} \right) \cup \left(1+\frac{\sqrt{2}}{2}, 2 \right)$ &  $\left(2+\frac{1}{\sqrt{2}},3\right)$  \\ [0.4em]
            \hline \hline
      $2$  & $ \{ 0,1 \}$ & $ [0,1] \cup \{2\}$ & $[0,2] \cup \{3\}$  \\ [0.4em]
     $3$  &  $\{ \frac{1}{2}, 1\}$ & $ \left[0, \frac{1}{2} \sqrt{\frac{1}{2}(5-\sqrt{7})} \right]  \cup \left[ 1, \frac{3}{2} \right] \cup \{2\}$ & $\left[ 0, \frac{5}{2} \right] \cup \{3\}$   \\ [0.4em]
      $4$  & $\left\{0, \frac{1}{\sqrt{2}},1 \right\}$  & $\left[0, \sqrt{\frac{3}{2}-\frac{1}{\sqrt{5}}} \right] \cup   \left[\sqrt{2} , 1+\frac{\sqrt{2}}{2} \right] \cup \{2\}$ &  $[0,2] \cup \left[ \frac{3}{\sqrt{2}}, 2+\frac{1}{\sqrt{2}}\right] \cup \{3\}$
    \end{tabular}
  \caption{Top: open sets for which we prove they are included in $\boldsymbol{\mu}_{A_{\kappa}}(\Delta) \cap [0,d]$. Bottom: closed sets for which we prove they are included in $[0,d] \setminus \boldsymbol{\mu}_{A_{\kappa}}(\Delta)$.}
      \label{table:202}
\end{table}
\begin{table}[H]
    \begin{tabular}{c|c|c|c} 
      $\kappa$ & $d=1$ & $d = 2$ & $d=3$  \\ [0.4em]
      \hline
      $2$  & $ (0,1)$ & $ (0,1)$ & $ (0,1)$  \\ [0.4em]
      $4$  & $\left(0,1\right) \setminus \{ \frac{1}{\sqrt{2}}\} $ & $\left(0, \frac{1}{2}\right)  \cup \left(\frac{1}{\sqrt{2}},1\right) $ & $\left(0, \frac{\sqrt{3-\sqrt{3}}}{4}\right)  \cup \left(\frac{1}{\sqrt{2}}, 1\right) $    \\[0.4em]
      $6$ & $\left(0,1\right) \setminus \{\frac{1}{2}, \frac{\sqrt{3}}{2}\}$ & $(0,0.21) \cup (0.53,0.71) \cup   \left(\frac{\sqrt{3}}{2},1 \right) $ & $(0,0.08) \cup (0.57,0.58) \cup \left(\frac{\sqrt{3}}{2},1 \right)$ \\ [0.4em]
          \hline \hline
      $2$  & $\{0,1\}$ & $\{0,1\}$ & $\{0,1\}$ \\ [0.4em]
      $4$  & $\{0,\frac{1}{\sqrt{2}}, 1\}$ & $\left[\frac{1}{2}, \frac{1}{\sqrt{2}} \right] \cup \{0,1\}$ & $\left[\frac{1}{2\sqrt{2}}, \frac{1}{\sqrt{2}} \right] \cup \{0,1\}$  \\[0.4em]
      $6$ & $\{0, \frac{1}{2}, \frac{\sqrt{3}}{2}, 1 \}$ &  $\left[\frac{1}{4}, \frac{1}{2} \right] \cup \left[\frac{3}{4}, \frac{\sqrt{3}}{2} \right] \cup \{0,1\}$ & $\left[\frac{1}{8}, \frac{1}{2} \right] \cup \left[\frac{3\sqrt{3}}{8}, \frac{\sqrt{3}}{2} \right] \cup \{0,1\}$  
    \end{tabular}
  \caption{Top: open sets for which we prove they are included in $\boldsymbol{\mu}_{A_{\kappa}}(D) \cap [0,1]$. Bottom: closed sets for which we prove they are included in $[0,1] \setminus \boldsymbol{\mu}_{A_{\kappa}}(D)$.  }
     \label{table:101}
\end{table}

Denote the point and singular continuous spectra of $\mathfrak{H}$ by $\sigma_\mathrm{p}(\mathfrak{H})$ and $\sigma_\mathrm{sc}(\mathfrak{H})$ respectively. The following new result holds for $(\mathfrak{D}, \mathfrak{H}) = (\Delta, \Delta+V)$ or $(D,D+V)$, and $\mathbb{A} = A_{\kappa}$, $\kappa \in (\N^*)^d$:

\begin{theorem}
\label{lapy305} 
Suppose there is $\epsilon >0$ such that $V(n) = O \left( |n|^{-\epsilon} \right)$ as $|n| \to \infty$ and 
\begin{equation}
n_j (V - \tau_j^{\kappa_j} V)(n)  = O \left( |n|^{-\epsilon} \right ),\quad as \ |n| \to \infty, \quad j=1,\ldots,d.
\label{A20}
\end{equation}
Let $E \in \boldsymbol{\mu}_{\mathbb{A}}(\mathfrak{D}) \setminus \sigma_\mathrm{p}(\mathfrak{H})$, where $\boldsymbol{\mu}_{\mathbb{A}}(\mathfrak{D})$ is as in Tables \ref{table:202} and \ref{table:101} for example. Then there is a closed interval $I$, such that $E$ belongs to the interior of $I$ and
\begin{enumerate}
\item $\sigma_\mathrm{p}(\mathfrak{H}) \cap I$ is finite and consists of eigenvalues of finite multiplicity, 
\item $\forall s>1/2$ the map $I_{\pm} \ni z \mapsto (\mathfrak{H}-z)^{-1} \in \mathscr{B}\left( \mathcal{K}, \mathcal{K}^* \right)$ extends to a uniformly bounded map on each open set of $I\setminus \sigma_{p}(\mathfrak{H})$, with $\mathcal{K} = L^2_{s}(\mathbf{N}) = L^2_{s,0,0,0}(\mathbf{N})$, defined by \eqref{L2sppM}, and $\mathbf{N} := (1 + N_1 ^2 + ... +N_d^2)^{1/2}$. In particular $\sigma_{\mathrm{sc}} (\mathfrak{H}) \cap I = \emptyset$. 
\end{enumerate}
\end{theorem}

A sharper version of Theorem \ref{lapy305} is stated in Theorem \ref{lapy3055} -- and the latter formulation also applies to classes of oscillating potentials, such as Wigner von-Neumann potentials, which decay only like $O(|n|^{-1})$, see examples \ref{exampleOSC} and \ref{exWvN}. The formulation of the LAP according to Mourre's original papers is Theorem \ref{THM_MOURRE}; the formulation in the Besov spaces is Theorem \ref{BesovTHM}.

The purpose of the second half of this article is to strengthen the connection between the two Laplacians. In Section \ref{Iso2D} an isomorphism $\pi$ is established between $\Delta$ and $D$ in dimension $2$. In our view the isomorphism is not surprising afterall as it is due to the trigonometric identity $2\cos(\xi_1) \cos(\xi_2) = \cos(\xi_1+\xi_2) + \cos(\xi_1-\xi_2)$. The isomorphism however enriches the spectral theory for the two Laplacians. For $\kappa = (\kappa_1, \kappa_2) \in (\N^*)^2$, $\alpha \in \R$, let $\alpha \kappa = (\alpha \kappa_1, \alpha \kappa_2)$. One finds conjugate operators $\pi A_{2\kappa} \pi^{-1} =$ \eqref{PIinversePI} for $\Delta$ and $\pi^{-1} A_{\kappa} \pi=$ \eqref{A_2d} for $D$. The new result is:

\begin{theorem}
\label{iso2Dthm} Fix $d=2$, $\kappa = (\kappa_1,\kappa_2) \in (\N^*)^2$. Then 
$$\boldsymbol{\mu}_{\pi A_{2\kappa} \pi^{-1}}(\Delta) = \boldsymbol{\mu}_{A_{2\kappa}}(2D) = 2 \times \boldsymbol{\mu}_{A_{2\kappa}}(D), \quad \text{and} \quad \boldsymbol{\mu}_{\pi^{-1} A_{\kappa} \pi}(D) = \boldsymbol{\mu}_{A_{\kappa}}(\frac{1}{2} \Delta) = \frac{1}{2} \times \boldsymbol{\mu}_{A_{\kappa}}(\Delta).$$ In addition to $V$ satisfying $V(n) = O ( |n| ^{-\epsilon} )$ as $|n| \to \infty$ for some $\epsilon >0$, consider the assumptions
\begin{align}
\label{CriterionA444}  
& (n_1-n_2) (V - \tau_1^{\kappa_1} \tau_2^{-\kappa_1} V)(n) \quad and \quad (n_1+n_2) (V - \tau_1^{\kappa_2} \tau_2^{\kappa_2} V)(n) = O \left( |n| ^{-\epsilon} \right), \\
\label{CriterionA555}  
& (n_1+n_2) (V - \tau_1^{\kappa_1} \tau_2^{\kappa_1} V)(n) \quad and \quad (n_2-n_1) (V - \tau_1^{\kappa_2} \tau_2^{-\kappa_2} V)(n) = O \left( |n| ^{-\epsilon} \right),
\end{align} 
for some $\epsilon >0$. Then Theorem \ref{lapy305} holds for $(\mathfrak{D}, \mathfrak{H}) = (\Delta, \Delta + V)$, $\mathbb{A} = \pi A_{2\kappa} \pi^{-1}$ and assuming \eqref{CriterionA444} instead of \eqref{A20} on the one hand, and $(\mathfrak{D}, \mathfrak{H}) = (D, D+V)$, $\mathbb{A} = \pi^{-1} A_{\kappa} \pi$ and assuming \eqref{CriterionA555} instead of \eqref{A20} on the other hand.
\end{theorem}

As seen in Table \ref{table:456798999}, the set $\boldsymbol{\mu}_{\pi A_{2\kappa} \pi^{-1}}(\Delta)$ (corresponding to $\mathbb{A} = \pi A_{2\kappa} \pi^{-1}$) covers parts of the spectrum of the 2-dimensional $\Delta$ that are not covered by neither $\boldsymbol{\mu}_{A_{\kappa}}(\Delta)$ nor $\boldsymbol{\mu}_{A_{2\kappa}}(\Delta)$. 
The spectral contributions appear to be somewhat complementary. Thus the isomorphism allows to augment the results of Theorem \ref{lapy305}, but different conditions on the potential are required. We find it interesting and convenient that $\boldsymbol{\mu}_{A_{\kappa}}(\Delta)$ appears to cover energies close to $0$ only if $\kappa = (\kappa_1,\kappa_2) = (1,1)$, see Tables \ref{tab:table1011} and \ref{tab:table10112162} ; on the other hand $\boldsymbol{\mu}_{\pi A_{2\kappa} \pi^{-1}}(\Delta)$ appears to cover energies close to $0$ for various combinations of $(\kappa_1,\kappa_2)$, see Tables \ref{tab:table1012} and \ref{tab:table101244}.

\begin{table}[H]
  \begin{tabular}{c|c|c|c|c}
    \multirow{2}{*}{} &       
      \multicolumn{2}{c|}{$d=2$} &
      \multicolumn{2}{c}{$d=3$} \\
      \hline
  $\kappa$  & $ A_{\kappa}$ & $ \pi A_{\kappa} \pi^{-1}$ & $ A_{\kappa}$ & $ \pi A_{\kappa, \rho} \pi^{-1}$   \\
    \hline
    $2$  & $(1,2)$ & $(0,2)$  & $(2,3)$  &  $(1,2)$ \\ [0.4em]
      $4$  & $(1.026,1.414) \cup (1.707,2)$ & $(0,1) \cup (1.414,2)$ & $(2.072,2.121) \cup (2.707,3)$ & $(2.414,2.707)$   \\ [0.4em]
  \end{tabular}
    \caption{$\boldsymbol{\mu}_{\mathbb{A}}(\Delta) \cap [0,d]$ for $\Delta$ in dimensions $d=2,3$. $\mathbb{A}$ as in the table. $\kappa_1 = \kappa_2$ for $d=2$ and $\kappa_1=\kappa_2=\kappa_3$ for $d=3$. $\rho = -0.5$. Numbers are rounded.}
     \label{table:456798999}
\end{table}

In dimension $3$, $\Delta$ and $D$ cannot be isomorphic due to a different number of neighbors ($6$ for $\Delta$ and $8$ for $D$). Nevertheless we specify an isomorphism between $\Delta$ and a hybrid of $D$ and $\Delta$ in $3d$. This isomorphism is basically the aforementioned $2$-dimensional isomorphism for the $x-y$ coordinates and the identity on the $z$ coordinate. It is analogous to a transformation from Cartesian to cylindrical coordinates. It gives rise to another conjugate operator for $\Delta$, which we denote $\pi A_{\kappa, \rho} \pi^{-1}$. This produces results that augment those obtained with just $A_{\kappa}$ (see Tables \ref{table:456798999}, \ref{table:4567} and \ref{table:4567new}). Another curious phenomenon observed is that for this conjugate operator, numerical results suggest that varying the parameter $\rho \in \R$ against the vertical coordinate produces different bands of a.c.\ spectrum. This is a phenomenon we do not grasp.

\begin{table}[H]
    \begin{tabular}{c|c|c|c|c} 
      \footnotesize $(\mathfrak{D}, \mathfrak{H})$ &  \footnotesize  $A_{\kappa}$ = \eqref{generatorDilations_011} & \footnotesize  $\pi A_{2\kappa} \pi^{-1}$ = \eqref{PIinversePI} & \footnotesize  $\pi^{-1} A_{\kappa} \pi$ = \eqref{A_2d}  &  \footnotesize  $\pi A_{\tilde{\kappa},\rho} \pi^{-1}$ = \eqref{CONJ_3d}  \\ [0.4em]
      \hline
      \small $(\Delta, H_{\mathrm{std}})$ & $d \in \N^*$ & $d=2$ & - & $d=3$   \\[0.4em]
      \small $(D, H_{\mathrm{MV}})$ & $d \in \N^*$ & - & $d=2$ & - 
    \end{tabular}
  \caption{List of conjugate operators in this article. $d=$ dimension. $\kappa = (\kappa_j) \in (\N^*)^d$.}
      \label{table:2020987}
\end{table}
We conclude with a series of remarks. Table \ref{table:2020987} summarizes all the conjugate operators $\mathbb{A}$ considered in this article and the Laplacians and dimensions to which they apply. 
Sections \ref{stdLaplacianMourre} and \ref{SME_MV} further elaborate on the sets $\boldsymbol{\mu}_{A_{\kappa}}(\Delta)$ and $\boldsymbol{\mu}_{A_{\kappa}}(D)$. Because of the complexity of the polynomial extremization problem alluded to before, numerical results with the help of the computer have been produced, see Tables \ref{tab:table1011}, \ref{tab:table10112162}, \ref{tab:table1013127conj}, \ref{tab:table1012}, \ref{tab:table101244}, \ref{tab:table10144}, along with conjectures on exact values for the band endpoints. Those results are denoted $\boldsymbol{\mu} ^{\mathrm{NUM}}_{A_{\kappa}}(\mathfrak{D})$ and $\boldsymbol{\mu} ^{\mathrm{CONJ}}_{A_{\kappa}}(\mathfrak{D})$ respectively. In the light of the numerical results we highlight three conjectures:
\begin{enumerate}
\item $\boldsymbol{\mu}_{A_{\kappa}}(\mathfrak{D})$ is a finite union of open intervals (bands) whose size go to zero as $\kappa^*$ goes to $\infty$.
\item $\boldsymbol{\mu}_{A_{\kappa}}(D)$ contains a band adjacent to energy $0$ in dimensions 2 and 3 for all even $\kappa$ ; whereas $\boldsymbol{\mu}_{A_{\kappa}}(\Delta)$ contains a band adjacent to energy $0$ in dimensions 2 and 3 only if $\kappa = 1$.
\item For $d=2,3$, $\kappa \equiv \kappa_j$, $\boldsymbol{\mu}_{A_{\kappa}}(\mathfrak{D}) \subset \boldsymbol{\mu}_{A_{\kappa / \delta}}(\mathfrak{D})$ whenever $\delta \in \N^*$ is a divisor of $\kappa$. 
\end{enumerate}
The second conjecture potentially illustrates a significant contrast between the two Laplacians, and also gives further importance to the isomorphism in dimension $2$. The third conjecture is quite natural from the viewpoint of the potential, in the sense that if $V$ satisfies \eqref{generalLR condition} for $\kappa$, then it satisfies \eqref{generalLR condition} for $\kappa/ \delta$, assuming $\delta | \kappa$ and $V$ goes to zero at infinity.

Let $V_{p}$ be a periodic potential of period $p = (p_1,\ldots,p_d)$. In spite of having $[V_{p}, A_{p}]_{\circ} = 0$ we cannot include periodic potentials as such because compactness remains a crucial assumption in our methods. For Mourre theory adapted for periodic potentials we refer to \cite{GN} and \cite{Sa2}.

There is the broad question of whether there are other conjugate operators consistent with our working assumption \eqref{generalLR condition} which contribute additional bands of a.c.\ spectrum. In part II of our investigations \cite{GM3} finite linear combinations $\sum_{l \in \N^*} \rho_{l \kappa} A_{l \kappa}$ are constructed and the evidence suggests that more bands of a.c.\ spectrum may be uncovered in this way. This partly addresses the question.

There is the intriguing question about what is the nature of the spectrum of $\mathfrak{H}$ between the bands of a.c.\ spectrum, i.e.\ for energies in $\sigma(\mathfrak{H}) \setminus \boldsymbol{\mu}_{\mathbb{A}}(\mathfrak{D})$. For $d=1$, Liu \cite{Li1}, based on Kiselev's work \cite{Ki}, proved that $\sigma_{\mathrm{sc}}(\Delta+V) = \emptyset$ whenever $V(n)= O(|n|^{-1})$. This topic is an open problem for $d \geq 2$. Perhaps the results of this article and those in part II \cite{GM3} can serve as an indication for this research. We also refer to articles by Stolz \cite{St1} and \cite{St2} where a.c.\ spectrum in dimension 1 is proved under different but akin conditions on $V$. 

We also wonder if the band endpoints could be potential locations of resonances or embedded eigenvalues. To illustrate what we mean, consider the oscillating $1$-dimensional perturbation created by Remling \cite{R}, see example \ref{exampleOSC}. It is such that $(V-\tau^2 V)(n) = O(|n|^{-2})$, but $(V-\tau V)(n) = O(|n|^{-1})$. It is relevant to pick $\mathbb{A} = A_{\kappa=2}$ to obtain a LAP. The embedded eigenvalue of this Hamiltonian, $E=0$, corresponds to the threshold between the two spectral bands $(-1,0) \cup (0,1) = \boldsymbol{\mu}_{A_{\kappa=2}}(\Delta)$. There are many open questions on the topic of embedded threshold eigenvalues. We refer to \cite{NoTa}, \cite{Li2}, \cite{IJ} and \cite{IJ2} for recent research in this area. 

Finally, we point out that the commutators $[\Delta, A_{\kappa}]_{\circ}$ and $[D, A_{\kappa}]_{\circ}$ are expressed in terms of Chebyshev polynomials and that these also enter directly in the Green's function for the Laplacian, at least in dimension $1$, see e.g.\ \cite{Y}, but also \cite{IJ} and \cite{IJ2}. It would be interesting to understand if there is a deeper connection. 

The plan of the article is as follows. Section \ref{full_Regularity} reviews operator regularity and practical criteria for the potential. In Sections \ref{stdLaplacianMourre} and \ref{SME_MV} the sets $\boldsymbol{\mu}_{A_{\kappa}}(\Delta)$ and $\boldsymbol{\mu}_{A_{\kappa}}(D)$ are determined to the best of our capability, with numerical additions in Appendices \ref{appendix_std} and \ref{appendix_MV} respectively. In Sections \ref{Iso2D} and \ref{Iso3D} we discuss isomorphisms in dimensions $2$ and $3$ respectively. In Sections \ref{proofMourre}, \ref{proofBesov} and \ref{proofEnergy} the main Theorems (LAPs) are stated with respect to the three commutator theories. Section \ref{Sec_example} has examples. Appendix \ref{convergence_appendix} details the convergence between the 2 different discrete Laplacians. Appendix \ref{appendix_algo} details the procedure used numerically estimate the sets $\boldsymbol{\mu}_{A_{\kappa}}(\Delta)$ and $\boldsymbol{\mu}_{A_{\kappa}}(D)$.

\noindent \textbf{Acknowledgements:} It is a pleasure to thank Vojkan Jak\u si\'c for drawing our attention to the Molchanov-Vainberg Laplacian, and Shu Nakamura for very helpful explanations.
\\
\noindent \textbf{Notation:} $\N$ denotes the non-negative integers, it contains $0$. Given a countable set $X$, $\ell_0(X):=\{f: X\to \C, \textrm{ with finite support}\}$. $\mathscr{B}(\mathcal{H})$ are the bounded operators on a Hilbert space $\mathcal{H}$. 

\section{Regularity and Mourre estimate for the full Hamiltonians}
\label{full_Regularity}

\subsection{Regularity: Abstract definitions.} Consider three self-adjoint operators $T,S$ and $\mathbb{A}$ acting in some complex Hilbert space $\mathcal{H}$. Suppose also $T,S \in \mathscr{B}(\mathcal{H})$. $T$ is of class $C^k(\mathbb{A})$, $k \in \N$, in notation $T \in C^k(\mathbb{A})$, if the map  
\begin{equation}
\label{defCk}
\R \ni t \mapsto e^{\i t\mathbb{A}}Te^{-\i t\mathbb{A}} \in \mathscr{B}(\mathcal{H})
\end{equation} 
has the usual $C^k(\R)$ regularity with $\mathscr{B}(\mathcal{H})$ endowed with the strong operator topology. Write $T \in C^{\infty}(\mathbb{A})$, if $T \in C^k(\mathbb{A})$ for all $k \in \N$. The form $[T,\mathbb{A}]$ is defined on $\text{Dom}[\mathbb{A}] \times\text{Dom}[\mathbb{A}]$ by $\langle \psi , [T,\mathbb{A}] \phi \rangle:= \langle T\psi,\mathbb{A} \phi \rangle - \langle \mathbb{A}\psi , T\phi \rangle$. By \cite[Lemma 6.2.9]{ABG} $T \in C^1(\mathbb{A})$ if and only if the form $[T,\mathbb{A}]$ extends to a bounded form on $\mathcal{H}\times \mathcal{H}$, in which case we denote the extended bounded form by $[T,\mathbb{A}]_{\circ}$. $T$ is of class $C^{k}_u(\mathbb{A})$, $k \in \N$, in notation $T \in C^k_u(\mathbb{A})$, if the map \eqref{defCk} has the usual $C^k(\R)$ regularity with $\mathscr{B}(\mathcal{H})$ endowed with the norm operator topology. $T \in C^{1,1}(\mathbb{A})$ if 
\begin{equation*}
\int_0 ^1 \| [T, e^{\i t\mathbb{A}} ]_{\circ} , e^{\i t\mathbb{A}}]_{\circ} \| t^{-2} dt < \infty.
\end{equation*}
One has $C^2(\mathbb{A}) \subset C^{1,1}(\mathbb{A}) \subset C^1_u(\mathbb{A}) \subset C^1(\mathbb{A})$, see \cite{ABG}. For each of these classes, the subclass comprising of the bounded operators on $\mathcal{H}$ constitutes a $\C-$vector space.

\subsection{Mourre estimate: Abstract definitions and properties.}
\label{full_Mourre}
 Suppose $T \in C^1(\mathbb{A})$, with $T, \mathbb{A}$ self-adjoint. The \textit{Mourre estimate} for $T$ holds w.r.t.\ $\mathbb{A}$ on an interval $I$, if $\exists \gamma >0$ such that 
\begin{equation}
\label{mourreEstimate12345}
E_I(T) [T, \i \mathbb{A}] _{\circ} E_{I} (T) \geq \gamma E_{I} (T) +K,
\end{equation}
where $E_I(T)$ is the spectral projection of $T$ onto $I$ and $K$ is some compact operator. Consider:
\begin{equation}
\label{mu_set_def_tilde}
\tilde{\boldsymbol{\mu}}_{\mathbb{A}} ^{\pm} (T):= \{ E \in \sigma(T): \exists \ \mathrm{an \ open \ interval \ } I, I \ni E, \gamma >0, \mathrm{ \ s.t. \ } \eqref{mourreEstimate12345} \mathrm{ \ holds  \ with \ }\pm \mathbb{A} \}.
 \end{equation}
Let $\boldsymbol{\mu}_{\mathbb{A}} ^{\pm} (T)$ be the subset of $\tilde{\boldsymbol{\mu}}_{\mathbb{A}} ^{\pm} (T)$ consisting of those energies for which \eqref{mourreEstimate12345} holds with $K=0$. The general theory provides useful results. First \cite[Proposition 7.2.10]{ABG} ensures $\boldsymbol{\mu}_{\mathbb{A}} ^{\pm} (T) \cap \sigma_{\mathrm{p}} (T) = \emptyset$. Second \cite[Theorem 7.2.13]{ABG} $\tilde{\boldsymbol{\mu}}_{\mathbb{A}} ^{\pm} (T) \setminus \boldsymbol{\mu}_{\mathbb{A}} ^{\pm} (T)$ consists of eigenvalues of $T$ of finite multiplicity. Third \cite[Theorem 7.2.9]{ABG} (see \cite[Lemma 3.3]{GM} for a proof) one has that $\tilde{\boldsymbol{\mu}}_{\mathbb{A}} ^{\pm} (T) = \tilde{\boldsymbol{\mu}}_{\mathbb{A}} ^{\pm} (S)$ whenever $T-S$ is compact. Now let $\boldsymbol{\mu}_{\mathbb{A}} (T):= \boldsymbol{\mu}_{\mathbb{A}} ^{+} (T)  \cup \boldsymbol{\mu}_{\mathbb{A}} ^{-} (T)$ and $\tilde{\boldsymbol{\mu}}_{\mathbb{A}}  (T):= \tilde{\boldsymbol{\mu}}_{\mathbb{A}} ^{+} (T)  \cup \tilde{\boldsymbol{\mu}}_{\mathbb{A}} ^{-} (T) $. In the context of our Schr\"odinger operators we have
$$\boldsymbol{\mu}_{\mathbb{A}} (\mathfrak{D}) = \tilde{\boldsymbol{\mu}}_{\mathbb{A}} (\mathfrak{D}) = \tilde{\boldsymbol{\mu}}_{\mathbb{A}} (\mathfrak{H}) = \boldsymbol{\mu}_{\mathbb{A}} (\mathfrak{H}) \cup \{ \mathrm{eigenvalues \ of } \ \mathfrak{H} \mathrm{ \ of \ finite \ multiplicity} \},$$
for $(\mathfrak{D}, \mathfrak{H}) = (\Delta, \Delta+V)$ or $(D,D+V)$, and any $\mathbb{A}$ as in Table \ref{table:2020987}, $\kappa = (\kappa_j)$ (we always assume $V(n) = o(1)$ as $|n|\to \infty$, hence compact). Alternatively, we note that, since we assume $V(n) = o(1)$ and $V \in C^1(\mathbb{A})$ in Theorems \ref{THM_MOURRE}, \ref{BesovTHM} and \ref{lapy3055}, then we have $V \in C^1_u(\mathbb{A})$ if and only if $[V, \mathbb{A}]_{\circ}$ is compact, see \cite[Proposition 2.1]{GM}.

\subsection{Regularity for the free operators $D$ and $\Delta$: Application.} 
Notation given in the Introduction is assumed, notably \eqref{def:std}, \eqref{def:MV22}, \eqref{FourierTT}, and \eqref{generatorDilations_011}. A simple computation on compactly supported sequences leads to the conclusion that
$$[N_j, S_j ^{\kappa_j} ]_{\circ} = \kappa_j S_j^{\kappa_j}, \quad [N_j, S_j^{-\kappa_j} ]_{\circ} = -\kappa_j S_j^{-\kappa_j}, \quad \forall \kappa = (\kappa_i)_{i=1} ^{d} \in (\N^*)^d, \quad \forall 1 \leq j \leq d.$$
Let $U_{n}$ be the \emph{Chebyshev polynomials of the second kind} of order $n$. They are defined by
\begin{align}\label{e:deftcheby}
 U_0:=1, \quad U_1:= 2X, \quad \text{ and } \quad U_{n+1}:= 2X U_n- U_{n-1}, \quad \forall n\geq 1.
 \end{align}
Note that $U_n$ is a polynomial of degree $n$ and that $U_n(1)=n+1$ for all $n\geq 0$. One also has
\begin{equation}
\label{cheby}
\sin(n \xi) = 2^{n-1} \prod_{j=0} ^{n-1} \sin \left( \frac{\pi j}{n} + \xi \right) = \sin(\xi) U_{n-1}(\cos(\xi)), \quad  \forall n\geq 1.
\end{equation}
It implies that 
\begin{align}\label{e:zerotcheby}
U_n\left(\cos\left(k\pi / (n+1 )\right)\right)=0, \quad \forall k\in\{1, \ldots, n\}.
\end{align}
Finally, if $n$ is even, note that
\begin{align}\label{e:polyeven}
\frac{U_{n-1}(x)}{x} = 2^{n-1} \prod_{k=1} ^{n/2-1} (x^2-\cos^2(k\pi /n)).
\end{align}
Given $j\in \{1, \ldots, d\}$ and $\kappa_j\geq 1$, thanks to \eqref{cheby}, one computes on $\ell_0(\Z^d)$ :
$$[\Delta_j , \i A_{\kappa}] = \mathcal{F}^{-1} \left[  \sin(\xi_j) \sin(\kappa_j \xi_j) \right] \mathcal{F}  =  \mathcal{F}^{-1} \left [  \sin ^2 (\xi_j) U_{\kappa_j-1} (\cos(\xi_j))  \right] \mathcal{F} = (1-\Delta_j^2) U_{\kappa_j-1} (\Delta_j).$$
Extending by density to all sequences in $\ell^2(\Z^d)$ allows to conclude that $\Delta_j \in C^1(A_{\kappa})$. Using linearity and obvious induction, this leads to the result:

\begin{proposition}
\label{reg_Prop_DD}
Let $\kappa = (\kappa_j)_{j=1} ^d$ be given. Then $\forall 1 \leq j \leq d$, $\Delta_j, \Delta, D \in C^{\infty}(A_{\kappa})$. Moreover:

\begin{align}
\label{MV:k333Ultra}
[\Delta, \i A_{\kappa}]_{\circ} &= \sum_{j=1} ^d [\Delta_j , \i A_j ]_{\circ}  = \sum_{j=1} ^d (1-\Delta_j^2) U_{\kappa_j-1} (\Delta_j), \\
\label{COMMUTATOR_D}
[D, \i A_{\kappa}]_{\circ} &= \sum_{j=1} ^d \frac{D}{\Delta_j} [\Delta_j , \i A_j ]_{\circ} = \sum_{j=1} ^d \frac{D}{\Delta_j} (1-\Delta_j^2) U_{\kappa_j-1} (\Delta_j). 
\end{align}

\end{proposition}
\begin{remark}
Note that $\displaystyle \frac{D}{\Delta_j} = \prod_{i=1, \ldots, d ; i\neq j} \Delta_i$ is a bounded operator.
\end{remark}

\subsection{Regularity criteria for the potential: Application.}
Let $\kappa = (\kappa_j) \in (\N^*)^d$ be given. The basic criterion for the $C^1(A_{\kappa})$ class is:
\begin{Lemma}
\label{LEM0}
Suppose $n_j (V-\tau_j ^{\kappa_j}V)(n) = O(1)$, as $|n|\to\infty$, $j=1,...,d$. Then $V \in C^1(A_{\kappa})$. 
\end{Lemma}
\begin{proof} 
On compactly supported sequences one computes
$$[V,\i A_{\kappa}] = \frac{1}{4} \sum_{j=1}^d (2N_j-\kappa_j)(V-\tau_j^{\kappa_j}V) S_j^{\kappa_j} - (2N_j+\kappa_j) (V-\tau_j^{-\kappa_j}V)S_j^{-\kappa_j}.$$
Then one extends the validity of the commutator to all sequences in $\ell^2(\Z^d)$ by density.
\qed
\end{proof}

In Lemma \ref{LEM0} if the assumption $O(1)$ is replaced with $o(1)$, and if also $V$ is compact, then $V \in C^1_u(A_{\kappa})$, see \cite[Proposition 2.1]{GM}. According to \cite{ABG}, \cite{BSa}, the main criteria for the $C^{1,1}(A_{\kappa})$ class is:
\begin{Lemma} 
\label{criteriaLS}
Short and long range components satisfy $V_s, V_l \in C^{1,1}(A_{\kappa})$ whenever respectively
\begin{equation}
\label{Firstcriterion}
\int_{1} ^{\infty} \sup _{r < | n| < 2r} |V_{\text{s}}(n)| dr < \infty, \ and 
\end{equation}
\begin{equation} 
\label{Secondcriterion}
V_{\text{l}}(n) = o(1) \textrm{ as } |n|\to\infty, \quad \text{and} \quad \int_{1} ^{\infty} \sup _{r < | n| < 2r} \big | (V_{\text{l}} - \tau_j^{\kappa_j} V_{\text{l}})(n) \big | dr < \infty, \quad j=1,...,d.
\end{equation}
\end{Lemma}
\begin{remark}
The criteria in Lemma \ref{criteriaLS} are not thorough to check the $C^{1,1}(A_{\kappa})$ condition. For example, let $V(n) = \sum_{i=1}^d \ln (2 + |n_i|) \cdot \ln ^{-3} (1 + \langle n \rangle)$. Then for $d \geq 2$, $V$ satisfies neither \eqref{Firstcriterion} nor \eqref{Secondcriterion}, regardless of the choice of $\kappa$, but $V$ clearly satisfies the criterion for the $C^2(A_{\kappa=1})$ class (see Lemma \ref{LemC2A} below), and so $V \in C^{1,1}(A_{\kappa=1})$.
\end{remark}

\begin{remark}
Recall the notation \eqref{weightsW205}. If there are $m \in \N$ and $1 = r < q$ such that $V_s(n)$ and $(V_l - \tau_j ^{\kappa_j} V_l)(n) = O(|n|^{-1} \cdot (w_m^{q,r}(n))^{-1})$ as $|n|\to\infty$, $1 \leq j \leq d$, then \eqref{Firstcriterion} and \eqref{Secondcriterion} are true respectively.
\end{remark}

The appendix of \cite{BSa} has a detailed proof of Lemma \ref{criteriaLS} for the case $A_{\kappa=1}$. There is no point in reproducing the argument here as it readily applies to the more general $A_{\kappa}$, thanks to Lemma \ref{A<N} (in \cite[Theorem 6.1]{BSa} one uses $\Lambda = \langle N \rangle = (1+N_1^2+...+N_d^2)^{1/2}$ and $A = A_{\kappa}$). Finally the basic criterion for the $C^2(A_{\kappa})$ class is:
\begin{Lemma}
\label{LemC2A} If $n_j (V-\tau_j ^{\kappa_j}V)(n) = O(1)$ and $n_i n_j \big[ (V-\tau_i ^{\kappa_i} V) - \tau_j ^{\kappa_j}(V-\tau_i ^{\kappa_i} V) \big](n) = O(1)$, as $|n|\to\infty$, $\forall 1 \leq i,j \leq d$, then $V \in C^2(A_{\kappa})$. 
\end{Lemma}
\begin{proof}
It is enough to compute $\left[ [V, \i A_{\kappa} ]_{\circ}, \sum_i N_i S_i ^{\kappa_i} \right]$ on the compactly supported sequences.
\begin{align*}
\displaystyle \left[ [V, \i A_{\kappa} ], \sum_i N_i S_i ^{\kappa_i} \right] &= -\frac{1}{4} \sum_i \kappa_i (2N_i - \kappa_i) (V- \tau_i ^{\kappa_i} V) S_i ^{2\kappa_i} + \kappa_i (2N_i + \kappa_i) (V- \tau_i ^{-\kappa_i} V) \\
& \quad  +\frac{2}{4} \sum_i \kappa_i N_i S_i ^{\kappa_i} (V- \tau_i ^{\kappa_i} V) S_i ^{\kappa_i}  -\kappa_i  N_i S_i ^{\kappa_i} (V- \tau_i ^{-\kappa_i} V) S_i ^{-\kappa_i} \\
& \quad + \frac{1}{4} \sum_{i,j} (2N_i - \kappa_i) N_j  \left( (V- \tau_i ^{\kappa_i} V) - \tau_j ^{\kappa_j}(V- \tau_i ^{\kappa_i} V) \right) S_j ^{\kappa_j} S_i ^{\kappa_i} \\
& \quad - \frac{1}{4} \sum_{i,j}  (2N_i + \kappa_i) N_j \left( (V- \tau_i ^{-\kappa_i} V) - \tau_j ^{\kappa_j}(V- \tau_i ^{-\kappa_i} V) \right) S_j ^{\kappa_j} S_i ^{-\kappa_i}.
\end{align*}
The validity of the commutator extends to all sequences in $\ell^2(\Z^d)$ by density.
\qed
\end{proof}

\section{Strict Mourre estimate for the Standard Laplacian $\Delta$, $d \geq 1$}
\label{stdLaplacianMourre}

In the previous Section the commutator $[\Delta, \i A_{\kappa}]_{\circ}$ was computed to justify operator regularity. Now it is used to determine the sets $\boldsymbol{\mu}_{A_{\kappa}} (\Delta)$ to the best of our capability. In a nutshell we prove

\begin{theorem}
The results mentioned in Table \ref{table:202} are true.
\end{theorem}
\begin{proof}
See Lemmas \ref{Lemma48}, \ref{Lemma49}, \ref{Lemma410}, \ref{Lemma411}, \ref{Lemma412} of this Section.
\qed
\end{proof}

Again, we emphasize that the sets $\boldsymbol{\mu}_{A_{\kappa}}(\Delta)$ that appear in Table \ref{table:202} are symmetric about zero, i.e.\  $\boldsymbol{\mu}_{A_{\kappa}}(\Delta) = -  \boldsymbol{\mu}_{A_{\kappa}}(\Delta)$, see Lemma \ref{Lemma_symmetryDelta}. To paint a more complete picture, we also derive in this Section other properties about the sets $\boldsymbol{\mu}_{A_{\kappa}} (\Delta)$. Numerical evidence is given in Tables \ref{tab:table1011} and \ref{tab:table1013127conj}. To $[\Delta, \i A_{\kappa}]_{\circ}$ given by \eqref{MV:k333Ultra} one associates the polynomial $g_E: [-1,1]^d \mapsto \R$,
\begin{equation}
\label{def:gE}
g_E(E_1,...,E_d):=  \sum_{j=1} ^d (1-E_j^2) U_{\kappa - 1} (E_j).
\end{equation}
Consider the constant energy surface for the Standard Laplacian:
\begin{equation}
\label{constE_MV}
S_E:= \left \{ (E_1, ..., E_d) \in [-1,1]^d: E = \sum_{j=1}^d E_j \right \}.
\end{equation} 
By functional calculus and continuity of the function $g_E$, we have that $E \in \boldsymbol{\mu}^{\pm}_{A_{\kappa}} (\Delta)$ if and only if $\pm \left.g_E\right|_{S_E}$ is strictly positive. We cover the $1-$dimensional case first. 

\begin{Lemma}
Let $d=1$. Then $\forall \kappa \in \N^*$, $\boldsymbol{\mu}_{A_{\kappa}}(\Delta) = [-1,1] \setminus \{\pm \cos(\pi j / \kappa), j=0,...,\floor*{\kappa/2}  \}$. In particular, $\Delta$ has $\kappa+1$ thresholds with respect to $A_{\kappa}$, where the thresholds are defined as 
$\sigma_{\rm ess}(\Delta) \setminus \boldsymbol{\mu}_{A_{\kappa}}(\Delta)$. 
\end{Lemma}
\begin{remark}
When $d=1$, $0 \in \boldsymbol{\mu}_{A_{\kappa}}(\Delta)$ if and only if $\kappa$ is even.
\end{remark}
\begin{proof}
In Fourier space we have $\sin(\xi) \sin(\kappa \xi) = 0 \Leftrightarrow \xi \in \big \{ \pm \pi j / \kappa: j = 0,...,\kappa \big \}$.  One takes the image of this set by the cosine function to get the corresponding threshold energies. Note that $\{\cos(\pi j / \kappa), j=0,...,\kappa  \} = \{\pm \cos(\pi j / \kappa), j=0,...,\floor*{\kappa/2}  \}$. Let us be specific about the choice of conjugate operator in the Mourre estimate \eqref{mourreEstimate123}. Since $d=1$, it is straightforward to decide this, especially if one has a graph of the function $\xi \mapsto \sin(\xi) \sin(\kappa \xi)$ at hand, see Figure \ref{fig2}. For $\kappa=1$, one takes $\mathbb{A} = A_{\kappa}$ for any $E \in (-1,1)$. For $\kappa=2$, one takes $\mathbb{A} = A_{\kappa}$ if $E \in (0,1)$ and $\mathbb{A} = -A_{\kappa}$ if $E \in (-1,0)$. For general $\kappa \geqslant 1$, one takes 
\begin{equation}
\label{formula_A}
\mathbb{A}:=
\begin{cases}
A_{\kappa}, & \text{if \ } E \in \bigcup _{j=0} ^{\floor*{\frac{\kappa-1}{2}}} \left( \cos(\frac{(2j+1)\pi}{\kappa}), \cos(\frac{2j\pi}{\kappa}) \right)\\[1em]
-A_{\kappa}, & \text{if \ } E \in \bigcup _{j=1} ^{\floor*{\frac{\kappa}{2}}} \left( \cos(\frac{2j\pi}{\kappa}), \cos(\frac{(2j-1)\pi}{\kappa}) \right).
\end{cases}
\end{equation}
\qed
\end{proof}
A convenience of the $1-$dimensionnal case is that the best constant in the strict Mourre estimate at energy $E$ is simply equal to $|g_E(E)| = (1-E^2) | U_{\kappa-1}(E) |$.

\begin{figure}[h]
\begin{tikzpicture}[domain=0:10]
\begin{axis}
[
clip = true, 
clip mode=individual, 
axis x line = middle, 
axis y line = middle, 
xlabel={$\xi$}, 
xticklabels={-$\pi$,-$\frac{\pi}{2}$,0, $\frac{\pi}{2}$, $\pi$},
xtick={-3.1415,-1.5707,0,1.5707,3.1415},
yticklabels={-1,0,1},
ytick={-1,0,1},
ylabel={$E$}, 
ylabel style={at=(current axis.above origin), anchor=west}, 
y=1cm,
x=1cm,
enlarge y limits={rel=0.03}, 
enlarge x limits={rel=0.03}, 
ymin = -1.2, 
ymax = 1.2, 
after end axis/.code={\path (axis cs:0,0) ;}]

\addplot[color=gray,line width = 1.0 pt, samples=1000, domain=-3.15:3.15] ({x},{sin(deg(x))*sin(2*deg(x))});
\addplot[color=black,line width = 1.0 pt, samples=1000, domain=-3.15:3.15,dotted] ({x},{cos(deg(x))});
\draw [color=blue, very thick] (0,0) -- (0,1);
\draw [color=yellow, very thick] (0,-1)-- (0,0);
\node[label={},circle,fill,inner sep=0.75pt, red] at (axis cs:0,1) {};
\node[label={},circle,fill,inner sep=0.75pt, red] at (axis cs:3.1415,-1) {};
\node[label={},circle,fill,inner sep=0.75pt, red] at (axis cs:1.5707,0) {};
\end{axis}
\end{tikzpicture}
\quad 
\begin{tikzpicture}[domain=0:10]
\begin{axis}
[
clip = true, 
clip mode=individual, 
axis x line = middle, 
axis y line = middle, 
xlabel={$\xi$}, 
xticklabels={-$\pi$,$-\frac{2\pi}{3}$, $-\frac{\pi}{3}$,0, $\frac{\pi}{3}$, $\frac{2\pi}{3}$, $\pi$},
xtick={-3.1415,-2.0943,-1.0471,0,1.0471,2.0943,3.1415},
yticklabels={$-1$,$\frac{1}{2}$, 0,$\frac{1}{2}$, $1$},
ytick={-1,-0.5,0,0.5,1},
ylabel={$E$}, 
ylabel style={at=(current axis.above origin), anchor=west}, 
y=1cm,
x=1cm,
enlarge y limits={rel=0.03}, 
enlarge x limits={rel=0.03}, 
ymin = -1.2, 
ymax = 1.2, 
after end axis/.code={\path (axis cs:0,0) ;}]

\addplot[color=gray,line width = 1.0 pt, samples=1000, domain=-3.15:3.15] ({x},{sin(deg(x))*sin(3*deg(x))});
\addplot[color=black,line width = 1.0 pt, samples=1000, domain=-3.15:3.15,dotted] ({x},{cos(deg(x))});
\node[label={},circle,fill,inner sep=0.75pt, red] at (axis cs:0,1) {};
\node[label={},circle,fill,inner sep=0.75pt, red] at (axis cs:3.1415,-1) {};
\node[label={},circle,fill,inner sep=0.75pt, red] at (axis cs:1.0471,0.5) {};
\node[label={},circle,fill,inner sep=0.75pt, red] at (axis cs:2.0943,-0.5) {};
\draw [color=blue, very thick] (0,0.5) -- (0,1);
\draw [color=blue, very thick] (0,-1)-- (0,-0.5);
\draw [color=yellow, very thick] (0,-0.5)-- (0,0.5);
\end{axis}
\end{tikzpicture}
\caption{Plot of $\xi \mapsto \sin(\xi) \sin(\kappa \xi)$, $\xi \mapsto \cos(\xi) =$ dotted. $\kappa=2$ (left), $\kappa=3$ (right).}
\label{fig2}
\end{figure}
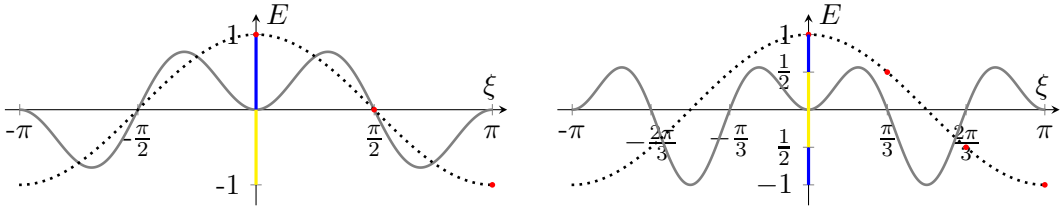

We now proceed with some results in higher dimensions. To stress the dependence on the dimension, write $\Delta_d$ to mean $\Delta$ on $\ell^2(\Z^d)$.

\begin{Lemma} For any $d \geq 1$, any $\kappa = (\kappa_j)$, $\pm d \not \in \boldsymbol{\mu}_{A_{\kappa}} (\Delta_d)$. 
\end{Lemma}
\begin{proof}
$E=d$ iff $E_j \equiv 1 \Rightarrow \left.g_E\right|_{S_E} = 0$. This implies the statement. Similarly for $E=-d$. 
\qed 
\end{proof}

\begin{Lemma} For any $d$ even, for any $\kappa = (\kappa_j)$, $0 \not \in \boldsymbol{\mu}_{A_{\kappa}}(\Delta)$. For any $d$ odd, any $\kappa = (\kappa_j)$ with all $\kappa_j$ even, $0 \not \in \boldsymbol{\mu}_{A_{\kappa}}(\Delta)$.
\end{Lemma}
\begin{proof}
For the case where $d$ is even, take half of the $E_j$'s equal to $1$, the other half equal to $-1$. Then $E= \sum E_j = 0$ and $g_E(E_1,...,E_d)=0$. For the case where we assume $d$ odd and all the $\kappa_j$ even, let $E_j \equiv 0$. Then $E = \sum E_j = 0$ and $g_E(E_1,...,E_d) = 0$ since $U_{\kappa_j-1}(0)=0$. \qed
\end{proof}

\begin{Lemma}
\label{lemSUMcos} 
Let $d \geq 2$. For all $\kappa = (\kappa_j)$,
\[\left\{\sum_{q=1} ^d \cos( j_q \pi / \kappa_q): (j_1,...,j_d) \in \prod_{q=1} ^d \{0,...,\kappa_q\} \right\} \subset [-d,d] \setminus \boldsymbol{\mu}_{A_{\kappa}} (\Delta).\] 
\end{Lemma}
\begin{remark} This lemma supports the conjectures in Tables \ref{tab:table1011} and \ref{tab:table1013127conj}.
\end{remark}
\begin{proof}
Recalling \eqref{e:zerotcheby}, the roots of $x \mapsto (1-x^2)U_{n-1}(x)$ are $\{ \cos(\pi j / n)\}_{j=0}^{n}$. Let $E_q = \cos( j_q \pi / \kappa_q)$, $q=1,...,d$. Then $E= \sum_{q=1}^d E_q = \sum_{q=1} ^d \cos( j_q \pi / \kappa_q)$ and $g_E(E_1, ..., E_d) = 0$.
\qed
\end{proof}
For the next Lemma we require more notation to avoid confusion. Let $g_E ^{[d]}$ denote the function $g_E$ from \eqref{def:gE} to specify it is a function of the $d$ variables $(E_1,...,E_d)$. For a multi-index $\kappa$, denote $\kappa [d]$ the restriction of $\kappa$ to its first $d$ components. 
\begin{Lemma} For any $d \geq 2$, for any $\kappa = (\kappa_j)_{j=1}^{d}$, one has
\[\boldsymbol{\mu}_{A_{\kappa [d]}}(\Delta_{d})\pm 1\subset\boldsymbol{\mu}_{A_{\kappa [d-1]}}(\Delta_{d}).\]
\end{Lemma}
\begin{proof}
From \eqref{def:gE}, $g_E ^{[d-1]}(E_1,...,E_{d-1}) = g_E^{[d]}(E_1,...,E_{d-1},\pm1)$. Therefore, 
\[\inf_{(E_1, \ldots, E_{d-1})\in S_E^{[d-1]}}g_E ^{[d-1]}(E_1,...,E_{d-1}) \geq \inf_{(E_1, \ldots, E_{d})\in S_{E \pm 1}^{[d]}} g_{E}^{[d]}(E_1,...,E_{d})\]
gives $\boldsymbol{\mu}^+_{A_{\kappa [d]}}(\Delta_{d})\pm 1\subset\boldsymbol{\mu}^+_{A_{\kappa [d-1]}}(\Delta_{d})$ and 
\[\sup_{(E_1, \ldots, E_{d-1})\in S_E^{[d-1]}}g_E ^{[d-1]}(E_1,...,E_{d-1}) \leq \sup_{(E_1, \ldots, E_{d})\in S_{E \pm 1}^{[d]}} g_{E}^{[d]}(E_1,...,E_{d})\]
ensures $\boldsymbol{\mu}^-_{A_{\kappa [d]}}(\Delta_{d})\pm 1\subset\boldsymbol{\mu}^-_{A_{\kappa [d-1]}}(\Delta_{d})$. This implies the statement. 
\qed
\end{proof}

\begin{Lemma}
\label{Lemma_symmetryDelta}
For any $d\geq1$, $\forall \kappa = (\kappa_j)$ with all $\kappa_j$ even or all $\kappa_j$ odd, $\boldsymbol{\mu}_{A_{\kappa}} (\Delta) = - \boldsymbol{\mu}_{A_{\kappa}} (\Delta)$. 
\end{Lemma}
\begin{proof}
This follows from \eqref{def:gE} and the fact that the $U_n(\cdot)$ are even when $n$ is even, odd when $n$ is odd. Then use the fact that $S_{-E} = -S_E$.
\qed
\end{proof}
For this reason, we focus only on positive energies whenever all the $\kappa_j$'s have the same parity.

\begin{Lemma}
\label{edgySPdd}
For any $d \geq 1$, for any $\kappa = (\kappa_j)$, set  $\kappa^*:= \max_{1 \leq j \leq d} \kappa_j$. We have:
\[ (d-1+\cos \left(\pi / \kappa^* \right) ,d)\subset \boldsymbol{\mu}_{A_{\kappa}}(\Delta). \]
Moreover if all $\kappa_j$ have the same parity, we also get
\[ -(d-1+\cos \left(\pi / \kappa^* \right) ,d)\subset \boldsymbol{\mu}_{A_{\kappa}}(\Delta). \]
\end{Lemma}
\begin{proof}
Consider the $d$ one variable polynomials, $h_j(x):= (1-x^2)U_{\kappa_j-1}(x)$, $j=1,...,d$. Clearly $g_E(E_1,...,E_d) = \sum_{j=1}^d h_j(E_j)$. Recalling \eqref{e:zerotcheby} and that $U_n(1)=n+1$ for all $n\geq 1$, we infer that
\[h_j(x)>0, \quad \forall x\in (\cos(\pi/\kappa_j), 1).\]
Furthermore, $E \in (d-1+\cos(\pi / \kappa^*),d)$ implies $S_E \subset (\cos(\pi / \kappa^*), 1)^d\subset \boldsymbol{\mu}^+_{A_{\kappa}}(\Delta)$. For the negative part, use Lemma \ref{Lemma_symmetryDelta}. \qed
\end{proof}
\begin{remark} The interval $(d-1+\cos \left(\pi / \kappa^* \right) ,d)$ is maximal in ${\mu}_{A_{\kappa}}(\Delta$) due to Lemma \ref{lemSUMcos} by taking $j_{\kappa^*}=1$ and the other $j_i=0$. 
\end{remark}

\noindent \underline{\textbf{Assumption:}} For the rest of this Section we suppose all $\kappa_j$'s are equal. Thus we always assume without loss of generality that $E \in (0,d)$. 

For some of the following results we use a \textit{ratio test}, to rule out some energies where a strict Mourre estimate can hold. It goes as follows: one considers two points belonging to $S_E$. For example one may choose (1) $(E_1,E_2,...,E_d) = (E-d+1,1,...,1)$, and (2) $(E_1,E_2,...,E_d) = (E/d,E/d,...,E/d)$. Then one looks at the sign of the ratio 
\begin{align*}
r(E):= d \frac{ g_E(E-d+1,1,...,1)}{g_E(E/d,E/d,...,E/d)} &= \frac{(1-(E-d+1)^2) U_{\kappa-1}(E-d+1)}{(1-(E/d)^2) U_{\kappa-1}(E/d)}.
\end{align*}
Assuming $-1 \leq E-d+1 \leq 1$ and $-1 \leq E/d \leq 1$, the sign of $r(E)$ is the same as that of $R(E) = U_{\kappa-1}(E-d+1) [U_{\kappa-1}(E/d)]^{-1}$.  If $R(E) <0$, a Mourre estimate cannot hold at $E$, i.e.\ $E \not \in \boldsymbol{\mu}_{A_{\kappa}}(\Delta)$. If $R(E) >0$, the test is inconclusive, i.e.\ a Mourre estimate may or may not hold at $E$. For warm up Table \ref{table:1019090u} applies the ratio test for some values of $\kappa$ in $2d$.

\begin{table}[H]
    \begin{tabular}{c|c|c} 
      $\kappa$ & $R(E)$ & $\{ E \in (0,2): R(E) <0 \}$     \\ [0.4em]
      \hline
      $1$  & $1$ &  (test inconclusive) \\ [0.4em]
      $2$  & $2(E-1)E^{-1}$ & $E \in (0,1)$ \\ [0.4em]
      $3$  & $[4(E-1)^2-1][E^2-1]^{-1}$ & $E \in (0,1/2) \cup (1,3/2)$ \\ [0.4em]
      $4$  &$[8(E-1)^3-4(E-1)][8(E/2)^3-4(E/2)]^{-1}$ & $E \in (1/2,1) \cup (\sqrt{2}, 3/2)$ 
     \end{tabular}
  \caption{Ratio test. $d=2$, $\kappa =1,2,3,4$. Applied to points $(E-1,1)$ and $(E/2,E/2)$. }
     \label{table:1019090u}
\end{table}
\begin{Lemma} 
\label{Lemma48}
Let $d\geq 1$. Suppose $\kappa=2$. Then $\boldsymbol{\mu}_{A_{\kappa}}(\Delta) = \pm (d-1,d)$. 
\end{Lemma}
\begin{proof}
The inclusion $\supset$ follows from Lemma \ref{edgySPdd}. For the reverse inclusion, we assume $E>0$ and $d\geq2$ (the statement is true for $d=1$). We apply the ratio test to the points $(E_1,x_2,x_3,...,x_d)$ and $(E/d,E/d,...,E/d)$ where the $x_i$ are equal to $0$ or $1$. Say that there are $j$ $x_i$'s that are equal to $1$, then $E_1 = E-j$. $j$ takes any of the values from $1$ to $d-1$. We must assume $-1 \leq E-j \leq 1$ and $-1 \leq E/d \leq 1$. Then $R(E) = d(E-j)/E$. This is negative for $j-1 < E < j$. The result now follows from the fact that the set of points where a Mourre estimate holds is an open set.
\qed
\end{proof}

\begin{Lemma} 
\label{Lemma49}
Let $d=2$, $\kappa = 3$. Then $\boldsymbol{\mu}_{A_{\kappa}}(\Delta) = \pm \left(\frac{1}{2}\sqrt{\frac{1}{2}(5-\sqrt{7})}, 1\right) \cup \pm \left(\frac{3}{2},2 \right)$.
\end{Lemma}
\begin{proof}
Fix $0 < E < 2$. For $\kappa = 3$, it is still human to do an analysis of the function 
\[h(x):= g_E(E-x, x) = (1-x^2)(4x^2-1) + (1-(E-x)^2)(4(E-x)^2-1),\]
 defined for $x \in [E-1,1]$. We give a brief sketch. The roots of $h'$ are $x=E/2$ and also $x = E/2 \pm \sqrt{5/2-3E^2}/2$ provided $E \in (0,\sqrt{5/6})$. If $E \in (0,\sqrt{5/6})$, $h$ reaches its maximum at $x = E/2 \pm \sqrt{5/2-3E^2}/2$ ; if $E \in (\sqrt{5/6},2)$ $h$ reaches its maximum at $x=E/2$. Plug these values into $h$ and find the positive roots of $h$. They are $\frac{1}{2}\sqrt{\frac{1}{2}(5-\sqrt{7})}$ and $1$ respectively. \qed 
\end{proof}

\begin{Lemma} 
\label{Lemma410}
Let $d=2$, $\kappa = 4$.  Then $\boldsymbol{\mu}_{A_{\kappa}}(\Delta) = \pm \left(\sqrt{\frac{3}{2} - \frac{1}{\sqrt{5}}}, \sqrt{2} \right) \cup \pm \left(1+\frac{\sqrt{2}}{2},2 \right)$.
\end{Lemma}
\begin{proof}
Fix $0 < E < 2$. Again we merely sketch a proof. We propose to analyze the function $h(x) := g_E(E-x, x)=(1-x^2)(8x^3-4x) + (1-(E-x)^2)(8(E-x)^3-4(E-x))$ defined for $x \in [E-1,1]$. It reaches its maximum at $x = E/2 \pm \sqrt{9/5-E^2}/2$ or $x=E/2$ depending on $E$. Plug these values into $h$ and find the positive root of $h$. They are $\sqrt{\frac{3}{2} - \frac{1}{\sqrt{5}}}$ and $\sqrt{2}$ respectively.  \qed 
\end{proof}
In the two previous Lemmata, note that the three numbers that are on the right, appear also in Lemma \ref{lemSUMcos}.
\begin{Lemma} 
\label{Lemma411}
Let $d=3$, $\kappa = 3$.  Then $\boldsymbol{\mu}_{A_{\kappa}}(\Delta) = \pm \left(\frac{5}{2}, 3\right)$.
\end{Lemma}
\begin{proof}
Inclusion $\supset$ follows from Lemma \ref{edgySPdd}. Now let $q(E,x):= [4(E-x)^2-1][4(E/3)^2-1]^{-1}$. For the reverse inclusion we use the ratio test. First we apply it to the points $(E-2,1,1)$ and $(E/3,E/3,E/3)$. We have $R(E) = q(E,2)$. This ratio is valid for $-1 \leq E-2 \leq 1$, and negative for $E \in (1,5/2)$. Then we apply it to the points $(E-3/2,1/2,1)$ and $(E/3,E/3,E/3)$. Use $U_2(1/2)=0$. We have $R(E) = q(E,3/2)$. This ratio is valid for $-1 \leq E-3/2 \leq 1$, and negative for $E \in (1/2,1) \cup (3/2, 2)$. Finally we apply it to the points $(E-1,1/2,1/2)$ and $(E/3,E/3,E/3)$. We have $R(E) = q(E,1)$. This ratio is valid for $-1 \leq E-1 \leq 1$, and negative for $E \in (0,1/2)$. Taking closures rules out a Mourre estimate at all energies between $0$ and $5/2$. 
\qed
\end{proof}

\begin{Lemma} 
\label{Lemma412}
Let $d=3$, $\kappa = 4$. Then $[0,2] \cup [3/ \sqrt{2}, 2+1/ \sqrt{2}]\subset [0,3] \setminus \boldsymbol{\mu}_{A_{\kappa}}(\Delta)$.
\end{Lemma}
\begin{proof}
Let $q(E,x):= [8(E-x)^3-4(E-x)][8(E/3)^3-4(E/3)]^{-1}$. First we apply the ratio test to the points $(E-2,1,1)$ and $(E/3,E/3,E/3)$. We have $R(E) = q(E,2)$. This ratio is valid for $-1 \leq E-2 \leq 1$, and negative for $E \in (2-1/ \sqrt{2},2) \cup (3/\sqrt{2}, 2+1/ \sqrt{2})$. Next we apply it to the points $(E-3/2,1/2,1)$ and $(E/3,E/3,E/3)$. We have $R(E) =q(E,3/2)$. This ratio is valid for $-1 \leq E-3/2 \leq 1$, and negative for $E \in (3/2-1/ \sqrt{2},3/2) \cup (3/ \sqrt{2}, 3/2+1/ \sqrt{2})$. Next we apply it to the points $(E-1,1/2,1/2)$ and $(E/3,E/3,E/3)$. We have $R(E) = q(E,1)$. This ratio is valid for $-1 \leq E-1 \leq 1$, and negative for $E \in (1-1/ \sqrt{2},1) \cup (1+ 1/ \sqrt{2}, 2)$. Finally we apply it to the points $(E-1/2,1/2,0)$ and $(E/3,E/3,E/3)$. We have $R(E) = q(E,1/2)$. This ratio is valid for $-1 \leq E-1/2 \leq 1$, and negative for $E \in (0,1/2)$. This implies the statement.
\qed
\end{proof}

\section{Strict Mourre estimate for the Molchanov-Vainberg Laplacian $D$, $d \geq 2$}
\label{SME_MV}

In Section \ref{full_Regularity} the commutator $[D, \i A_{\kappa}]_{\circ}$ was computed to justify operator regularity. Now it is used to determine the sets $\boldsymbol{\mu}_{A_{\kappa}} (D)$ to the best of our capability. In a nutshell we prove

\begin{theorem}
The results mentioned in Table \ref{table:101} are true.
\end{theorem}
\begin{proof}
See Corollary \ref{CorollaryK=2}, Lemma \ref{lemK=4}, and subsections \ref{subk4}, \ref{subk6}, \ref{sub_central}, \ref{sub_second_band} of this Section.
\qed
\end{proof}

Again, we emphasize that the sets $\boldsymbol{\mu}_{A_{\kappa}}(D)$ that appear in Table \ref{table:101} are symmetric about zero, i.e.\  $\boldsymbol{\mu}_{A_{\kappa}}(D) = -  \boldsymbol{\mu}_{A_{\kappa}}(D)$, see Lemma \ref{symmetry_propertyD}. To paint a more complete picture, we also derive in this section other properties about the sets $\boldsymbol{\mu}_{A_{\kappa}} (D)$. Numerical evidence is given in Tables \ref{tab:table1012} and \ref{tab:table10144}. To $[D, \i A_{\kappa}]_{\circ}$ given by \eqref{COMMUTATOR_D} one associates the function $g_E: [-1,1]^d \mapsto \R$,
\begin{equation}
\label{defGE}
g_E(E_1,...,E_d) := \sum_{j=1}^d \frac{E}{E_j} (1-E_j^2) U_{\kappa_j-1} (E_j).
\end{equation}
Consider the constant energy surface for the Molchanov-Vainberg Laplacian:
\begin{equation}
\label{constE_MV22}
S_E:= \left \{ (E_1, ..., E_d) \in [-1,1]^d: E = \prod_{j=1}^d E_j \right \}.
\end{equation} 
Note that on $g_E$ restricted to $S_E$ becomes a polynomial. By functional calculus and continuity of the function $g_E$, we have that $E \in \boldsymbol{\mu}^{\pm}_{A_{\kappa}} (D)$ if and only if $\pm \left.g_E\right|_{S_E}$ is strictly positive. 

\begin{Lemma}\label{l:D0} For any $d \geq 1$, any $\kappa = (\kappa_j)$, $0, \pm 1 \not \in \boldsymbol{\mu}_{A_{\kappa}} (D)$. 
\end{Lemma}
\begin{proof}
If $E \in \{ 0, \pm 1\}$ then $\left.g_E\right|_{S_E} = 0$. This implies the statement.
\qed 
\end{proof}

For the next Lemma we require more notation to avoid confusion. Write $D_d$ to mean $D$ on $\ell^2(\Z^d)$. Let $g_E ^{[d]}$ denote the function $g_E$ from \eqref{defGE} to specify it is a function of the $d$ variables $(E_1,...,E_d)$. For a multi-index $\kappa$, denote $\kappa [d]$ the restriction of $\kappa$ to its first $d$ components. 

\begin{Lemma} For any $d\geq 2$, any $\kappa = (\kappa_j)_{j=1}^{d}$, $\boldsymbol{\mu}_{A_{\kappa [d]}} (D_d) \subset \boldsymbol{\mu}_{A_{\kappa [d-1]}} (D_{d-1})$.
\end{Lemma}
\begin{proof}
From \eqref{defGE}, $g_E ^{[d-1]}(E_1,...,E_{d-1}) = g_E ^{[d]}(E_1,...,E_{d-1},1)$. This implies the statement.
\qed
\end{proof}

\begin{Lemma}
Let $d \geq 2$. Let $\kappa = (\kappa_j)$ and suppose at least two parameters $\kappa_{j_1}$ and $\kappa_{j_2}$ are odd. Then $\boldsymbol{\mu}_{A_{\kappa}} (D) = \emptyset$.
\end{Lemma}
\begin{proof}
We start with the two-dimensional case. In this case we have 
\begin{equation*}
\label{MV:k3}
g_E(E_1,E_2) = E_2 (1-E_1^2)U_{\kappa_1-1}(E_1) + E_1 (1-E_2^2)U_{\kappa_2-1}(E_2).
\end{equation*}
Fix $E = E_1 E_2 \in \pm (0,1)$. $\kappa_1$ and $\kappa_2$ odd implies $U_{\kappa_1-1}(\cdot)$ and $U_{\kappa_2-1}(\cdot)$ are even functions. Note that if $E>0$, then $(-E_1)\cdot(-E_2)= E_1E_2$ and $g_E(E_1,E_2) = -g_E(-E_1,-E_2)$, whereras if $E<0$, $(-E_1)\cdot E_2 = E_1\cdot(-E_2)$ and  $g_E(-E_1,E_2)=-g_E(E_1,-E_2)$. So $g_E$ cannot be strictly positive or strictly negative.

We now extend this observation to the case $ d \geq 3$. Assume, without lost of generality, that $\kappa_1$ and $\kappa_2$ are odd. Fix $E \in \pm (0,1)$. If $E>0$,  $g_E(E_1,E_2,+1,...,+1) = - g_E(-E_1,-E_2,+1,...,+1)$, whereas if $E<0$ we have $g_E(-E_1,E_2,+1,...,+1) = - g_E(E_1,-E_2,+1,...,+1)$. 
\qed
\end{proof}
\noindent \underline{\textbf{Assumption:}} For the rest of this Section we suppose all $\kappa_j$'s are even. Applying \eqref{cheby} to $\xi=\pi/2-\xi_j$, we infer 
\[\sin(\kappa_j \xi_j) = (-1)^{\kappa_j /2-1} \cos(\xi_j) U_{\kappa_j-1}(\sin(\xi_j)).\] 
Thus, $[\Delta_j , \i A_{\kappa}] = \mathcal{F}^{-1} \left[  \sin(\xi_j) \sin(\kappa_j \xi_j) \right] \mathcal{F} $ and 
\begin{equation}
\begin{aligned}
\label{dAk}
[D, \i A_{\kappa}]_{\circ} &=  D \mathcal{F} ^{-1} \sum_{j=1} ^d \frac{\sin(\xi_j) \sin(\kappa_j \xi_j)}{\cos(\xi_j)} \mathcal{F} = D \mathcal{F} ^{-1} \sum_{j=1} ^d \frac{(1- \cos^2(\xi_j)) \sin(\kappa_j \xi_j)}{\cos(\xi_j)\sin(\xi_j)} \mathcal{F} \\
&=D \sum_{j=1} ^d (-1)^{\kappa_j /2-1} (1-\Delta_j^2) \mathcal{F} ^{-1} \left[ \frac{U_{\kappa_j-1} (\sin(\xi_j))}{\sin(\xi_j)} \right] \mathcal{F} \\
&= D \sum_{j=1} ^d (-1)^{\kappa_j /2-1} 2^{\kappa_j-1}  \prod_{l=1} ^{\kappa_j /2} (\sin^2(l \pi / \kappa_j)-\Delta_j^2),
\end{aligned}
\end{equation}
where we used \eqref{e:polyeven} for the last line. Thus we also have 
\begin{equation}
\label{MV:k general222}
g_E(E_1,...,E_d) = E \sum_{j=1} ^d (-1)^{\kappa_j /2-1} 2^{\kappa_j-1}  \prod_{l=1} ^{\kappa_j /2} (\sin^2(l \pi / \kappa_j)-E_j^2).
\end{equation}

\begin{Lemma}
\label{symmetry_propertyD}
For any $d\geq1$, for any $\kappa = (\kappa_j)$ with all $\kappa_j$ even, $\boldsymbol{\mu}_{A_{\kappa}} (D) = - \boldsymbol{\mu}_{A_{\kappa}} (D)$.
\end{Lemma}
\begin{proof}
This follows from \eqref{MV:k general222} and the fact that 
$$S_{-E} = \bigcup_{j=1}^d \{ \lambda_j(E_1,...,E_d): (E_1,...,E_d) \in S_E\},$$ 
where $\lambda_j (E_1,...,E_d) := (F_1,...,F_d)$, with $F_i = E_i$ if $i \neq j$ and $F_j = -E_j$. 
\qed
\end{proof}
For this reason, we focus only on positive energies for the rest of this section.
\begin{Lemma} 
\label{lemPRODsin}
Let $d \geq 2$. For any $\kappa = (\kappa_j)$ with all $\kappa_j$ even, 
\[\left\{\prod_{q=1} ^d \sin( j_q \pi / \kappa_q): (j_1,...,j_d) \in \prod _{q=1} ^d \{0,1,...,\kappa_q/2\} \right\} \subset [0,1] \setminus \boldsymbol{\mu}_{A_{\kappa}} (D).\] 
This supports the conjectures in Tables \ref{tab:table1012} and \ref{tab:table10144}.
\end{Lemma}
\begin{proof} 
If some $j_q=0$, then $\prod_{q=1} ^d \sin( j_q \pi / \kappa_q)=0$ and the result is given by Lemma \ref{l:D0}. Let $(j_1,...,j_d) \in \prod _{q=1} ^d \{1,...,\kappa_q/2\}$. Set
$E_q := \sin(j_q \pi / \kappa_q)$, $q=1,...,d$. Then, we get $E= \prod_{q=1} ^d \sin( j_q \pi / \kappa_q)$ and $g_E(E_1, ..., E_d) = 0$ by \eqref{MV:k general222}.
\qed
\end{proof}

\begin{Lemma}
\label{edgySP}
For any $d \geq 1$, for any $\kappa = (\kappa_j)$ with all $\kappa_j$ even and  $\kappa_j/2$ with the same parity, 
\[\boldsymbol{\mu}_{A_{\kappa}}(D) \supset  \pm (\cos \left(\pi / \kappa^* \right) ,1).\]
\end{Lemma}
\begin{proof}
Fix $E \in (\cos(\pi/\kappa ^*),1)= (\sin((\kappa^*/2-1)\pi/\kappa^*), \sin((\kappa^*/2)\pi/\kappa^*))$, $\kappa^*:= \max \kappa_j$. Note that we have $E_j \in (\cos(\pi/\kappa_j), 1)$ whenever $E = \prod E_j$. Moreover, given $j=1, \ldots, d$, we set $h_j(x) := (-1)^{\kappa_j /2-1} 2^{\kappa_j-1}  \prod_{l=1} ^{\kappa_j /2} (\sin^2(l \pi / \kappa_j)-x^2)$. Note that $(-1)^{\kappa_j/2 -1}h_j(x)>0$, for $x\in (\cos(\pi/\kappa_j), 1)$.  Because all the $\kappa_j/2$'s have the same parity and since $g_E(E_1,...,E_d) = E \sum_{j=1}^d h_j(E_j)$, we derive the statement.
\qed
\end{proof}

\begin{remark}
Note that the interval given in Lemma \ref{edgySP} are maximal. Let $q$  be such that $\kappa_q=\kappa^*$. Taking $j_q=\kappa^*/2 -1$ and  $j_i= \kappa_i/2$, for $i\neq q$, in Lemma \ref{lemPRODsin} we see that $\sin(\pi/\kappa^*)\notin \boldsymbol{\mu}_{A_{\kappa}}(D)$. 
\end{remark}
\begin{corollary}
\label{CorollaryK=2}
Let $d\geq 1$. Suppose $\kappa = 2$. Then $\boldsymbol{\mu}_{A_{\kappa}} (D) = \pm (0,1)$. 
\end{corollary}

\begin{Lemma} 
\label{lemK=4}
Let $d=2$. Suppose $\kappa = 4$. Then $\boldsymbol{\mu}_{A_{\kappa}} (D) = \pm (0,1/2) \cup \pm (1/\sqrt{2}, 1)$. 
\end{Lemma}
\begin{proof}
Let $E = E_1 E_2 >0$. Then from \eqref{MV:k general222} we have
\begin{equation}
\left.g_E\right|_{S_E}  = g_E(E_1,E/E_1) = 4E \left[ (1-E_1^2)(2E_1^2-1) + (1-(E/E_1)^2)(2(E/E_1)^2-1) \right].
\end{equation}
defined for $E \in (0,1)$ and $E_1 \in \pm [E,1]$. We seek the values of $E$ such that $\left.g_E\right|_{S_E}$ is a strictly positive/negative function of $E_1$. We focus on $E_1 \in [E,1]$ only as $g_E$ is an even function of $E_1$.

Write $g_E(E_1, E/E_1) = 4E h(E_1)$. Solving $g'_E(E_1,E/E_1)=0$ is equivalent to solving $h' (E_1) = 0$. One computes $$h'(E_1) = -2E_1(2E_1^2-1) + 4E_1 (1-E_1^2) + 2\frac{E^2}{E_1^3} \left(2 \frac{E^2}{E_1^2} -1\right) - 4\frac{E^2}{E_1^3} \left(1-\frac{E^2}{E_1^2}\right).$$
Simplifying and factorizing leads to solving 
$$0 = -4E_1^8 + 3E_1^6 -3 E^2 E_1^2 + 4E^2 = -(E_1^2-E)(E_1^2+E)(4E_1^4 - 3E_1^2 +4E^2).$$
Now the last term in brackets is
$$4E_1^4 - 3E_1^2 +4E^2 = 
\begin{cases}
>0, & \text{ if } E^2 > 9/64 \\
(3-8E_1^2)^2/16, & \text{ if } E^2 = 9/64 \\
4(E_1^2-3/8)^2 + 4E^2 - 4(3/8)^2 , & \text{ if } E^2 < 9/64.
\end{cases}
$$
For the last case of the 3 cases, 
$$4(E_1^2-3/8)^2 + 4E^2 - 4(3/8)^2 = 4\left[ E_1^2-3/8-\sqrt{(3/8)^2-E^2} \right]\left[E_1^2-3/8+\sqrt{(3/8)^2-E^2} \right].$$
In what follows we write simply $g(E_1)$ instead of $g_E(E_1,E/E_1)$.

\noindent $\bullet$ If $E>3/8$, the roots of $h'$ are $E_1 = \pm \sqrt{E}$. 
One computes that $h''(\pm \sqrt{E}) = 24-64E$ which is strictly negative if and only if $E > 3/8$. Thus $g$ achieves global maxima at $E_1 = \pm \sqrt{E}$, and $g(\pm \sqrt{E}) = 8E (1-E) (2E-1)$. $g$ is increasing on $[E,\sqrt{E}]$ and decreasing on $[\sqrt{E},1]$. Also $g(E) = g(1) = 4E(1-E^2)(2E^2-1)$. We want $g$ to be either strictly positive or strictly negative on $[E,1]$. $g$ is strictly negative if and only if $E<1/2$ ; $g$ is strictly positive if and only if $E>1/\sqrt{2}$. This gives $(3/8,1/2) \cup (1/\sqrt{2},1) = \boldsymbol{\mu}_{A_{\kappa}} (D) \cap (3/8, 1)$. 

\noindent $\bullet$ If $E = 3/8$, the roots of $h'$ are $E_1 = \{ \pm \sqrt{E} \}$. 
 $h''(\pm \sqrt{E}) = 24-64E = 0$. Since $g(3/8) = g(1) < 0$, it must be that $g$ achieves global maxima at $E_1 = \pm \sqrt{E}$. One computes $g(\pm \sqrt{3/8}) = -15/32 <0$. Thus $g$ is strictly negative on the interval $E_1 \in [E,1]$. This gives  $3/8 \in \boldsymbol{\mu}_{A_{\kappa}} (D)$. 
 
\noindent $\bullet$ Finally if $0 < E < 3/8$, the roots of $h'$ are 
$$E_1 \in  \left \{ \pm \sqrt{E},\pm \sqrt{3/8+\sqrt{(3/8)^2-E^2}},  \pm \sqrt{3/8-\sqrt{(3/8)^2-E^2}} \right \}.$$
The objective is to show that $g$ is strictly negative on the interval $E_1 \in [E,1]$. $g(E) = g(1) = 4E(1-E^2)(2E^2-1) < 0$. Also, $g(\sqrt{E}) = 8E(1-E)(2E-1) <0$. One may check that  
\begin{align*}
g \left(\sqrt{3/8+\sqrt{(3/8)^2-E^2}} \right) = g \left(\sqrt{3/8-\sqrt{(3/8)^2-E^2}} \right) < 0
\end{align*}
whenever $E \in (0,3/8)$. Thus $(0,3/8) = \boldsymbol{\mu}_{A_{\kappa}} (D) \cap (0,3/8)$.
\qed
\end{proof}

\noindent \underline{\textbf{Assumption:}} For the rest of this Section we suppose all the $\kappa_j$'s are equal. Thus we always assume without loss of generality that $E \in (0,1)$.

We now devise elementary tests that allow to glean further information about the sets $\boldsymbol{\mu}_{A_{\kappa}} (D)$. The tests are crude but still yield partial results. We develop:
\begin{itemize}
\item A \textit{ratio test}: it allows to determine energies that belong to $[0,1] \setminus \boldsymbol{\mu}_{A_{\kappa}} (D)$. 
\item \textit{Central \& second band tests}: they allow to determine energies that belong to $\boldsymbol{\mu}_{A_{\kappa}} (D)$. The band tests exploit oscillations of apparent decreasing intensity of the function \eqref{functionH}, see Figure \ref{fig102} for an illustration. The central band test probes for energies within $(0, \sin(\pi /\kappa))$ ; the second band test examines energies within $(\sin(\pi /\kappa), \sin(2\pi /\kappa))$. Naturally we could continue with a third band test to examine energies within $(\sin(2\pi /\kappa), \sin(3\pi /\kappa))$, etc.. The central band test is valid for $\kappa \geq 4$ ; the second band test is valid for $\kappa \geq 6$.
\end{itemize}

\subsection{The ratio test}
As usual, we consider only $E \in (0,1)$. One considers 2 points belonging to $S_E$. Unless otherwise specified we will choose (1) $(E_1,E_2,...,E_d) = (E,1,...,1)$, and (2) $(E_1,E_2,...,E_d) = (E^{1/d},E^{1/d},...,E^{1/d})$. Then one looks at the sign of the ratio
\begin{align*}
R(E):= \frac{ d g_E(E,1,...,1)}{g_E(E^{1/d},E^{1/d},...,E^{1/d})} &= \frac{\prod_{l=1} ^{\kappa /2} (\sin^2(l \pi / \kappa)-E^2)}{\prod_{l=1} ^{\kappa /2} (\sin^2(l \pi / \kappa)-E^{2/d})}
\end{align*}

One does a sign chart of the $R(E)$ function. If the overall sign is negative, a Mourre estimate cannot hold at energy $E$, i.e.\ $E \not \in \boldsymbol{\mu}_{A_{\kappa}}(D)$. In other notation, $\{ E \in (0,1): R(E) < 0 \} \subset [0,1] \setminus \boldsymbol{\mu}_{A_{\kappa}}(D)$. If the sign is overall positive, the test is inconclusive, meaning that a Mourre estimate may or may not hold at energy $E$. 

Let us apply this test for different values of $\kappa$. For $d=1$, any $\kappa$ even, the test is trivial, with $R_{d=1}(E) = 1$, which is inconclusive for all energies $\not\in \pm \{\sin(l \pi / \kappa): l=1,...,\kappa/2\}$. This is consistent with the one-dimensional analysis. For any $d \geq 2$, $\kappa = 2$,  $R(E) = (1-E^2)/(1-E^{2/d}) >0$, which is inconclusive for all energies. This is consistent with Corollary \ref{CorollaryK=2}.

\subsection{Ratio test for $\kappa =4$}
\label{subk4}
We focus on dimensions $2$ and $3$. We have
\begin{align*}
R_{d=2}(E) &= \frac{((1/\sqrt{2})^2 -E^2)(1-E^2)}{((1/\sqrt{2})^2 -E)(1-E)}, \quad R_{d=3}(E) = \frac{((1/\sqrt{2})^2 -E^2)(1-E^2)}{((1/\sqrt{2})^2 -E^{2/3})(1-E^{2/3})}.
\end{align*}
\begin{figure}[H]
\begin{tikzpicture}
   \tkzTabInit[lgt = 3.7, espcl = 2, deltacl = 0.9]{$E$ / 1.1 , $R_{d=2}(E)$ / 0.7, $R_{d=3}(E)$ / 0.7}{$0$, $\dfrac{1}{2\sqrt{2}}$, $\dfrac{1}{2}$, $\dfrac{1}{\sqrt{2}}$, $1$}
   \tkzTabLine{d, +, +, +, d, -, d, +,d}
   \tkzTabLine{d, +, d, -, d, -, d, +,d}
\end{tikzpicture}
\label{signChart423}
\caption{Ratio test for $\kappa =4$, $d=2,3$. }
\end{figure}
The results for the case $\kappa=4$, $d=2$ is consistent with Lemma \ref{lemK=4}.

\subsection{Ratio test for $\kappa =6$}
\label{subk6}
We focus on dimensions $2$ and $3$. We have
\begin{align*}
R_{d=2}(E) &= \frac{ ((1/2)^2 -E^2)((\sqrt{3}/2)^2 -E^2)(1-E^2)}{ ((1/2)^2 -E)((\sqrt{3}/2)^2 -E)(1-E)}. 
\end{align*}
\begin{align*}
R_{d=3}(E) &= \frac{ ((1/2)^2 -E^2)((\sqrt{3}/2)^2 -E^2)(1-E^2)}{ ((1/2)^2 -E^{2/3})((\sqrt{3}/2)^2 -E^{2/3})(1-E^{2/3})}.
\end{align*}
\begin{figure}[H]
\begin{tikzpicture}
   \tkzTabInit[lgt = 3.3, espcl = 1.8, deltacl = 0.7]{$E$ / 1.1, $R_{d=2}(E)$ / 0.7}{$0$, $\dfrac{1}{4}$, $\dfrac{1}{2}$, $\dfrac{3}{4}$, $\dfrac{\sqrt{3}}{2}$, $1$}
   \tkzTabLine{d, +, d, -, d, +, d, -, d, +, d}
\end{tikzpicture}
\label{signChart62}
\caption{Ratio test for $\kappa =6$, $d=2$.}
\end{figure}
\begin{figure}[H]
\begin{tikzpicture}
   \tkzTabInit[lgt = 3.3, espcl = 1.8, deltacl = 0.7]{$E$ / 1.1 , $R_{d=3}(E)$ / 0.7}{$0$, $\dfrac{1}{8}$, $\dfrac{1}{2}$, $\dfrac{3\sqrt{3}}{8}$, $\dfrac{\sqrt{3}}{2}$, $1$}
   \tkzTabLine{d, +, d, -, d, +, d, -,d , +, d}
\end{tikzpicture}
\label{signChart63}
\caption{Ratio test for $\kappa =6$, $d=3$.}
\end{figure}

\subsection{Ratio test for $\kappa =8$}
We focus on dimensions $2$ and $3$. We have
\begin{align*}
R_{d=2}(E) &= \frac{ (\frac{2-\sqrt{2}}{4}-E^2)(\frac{1}{2} - E^2) (\frac{2+\sqrt{2}}{4} - E^2) (1-E^2) }{ (\frac{2-\sqrt{2}}{4}-E)(\frac{1}{2} - E) (\frac{2+\sqrt{2}}{4} - E) (1-E) }.
\end{align*}
\begin{align*}
R_{d=3}(E) &= \frac{ (\frac{2-\sqrt{2}}{4}-E^2)(\frac{1}{2} - E^2) (\frac{2+\sqrt{2}}{4} - E^2) (1-E^2) }{ (\frac{2-\sqrt{2}}{4}-E^{2/3})(\frac{1}{2} - E^{2/3}) (\frac{2+\sqrt{2}}{4} - E^{2/3}) (1-E^{2/3}) }.
\end{align*}
\begin{figure}[H]
\begin{tikzpicture}
   \tkzTabInit[lgt = 3.3, espcl = 1.8, deltacl = 0.7]{$E$ / 1.1, $R_{d=2}(E)$ / 0.7}{$0$, $\frac{2-\sqrt{2}}{4}$, $\frac{\sqrt{2-\sqrt{2}}}{2}$, $\frac{1}{2}$, $\frac{1}{\sqrt{2}}$, $\frac{2+\sqrt{2}}{4}$, $\frac{\sqrt{2+\sqrt{2}}}{2}$, $1$}
   \tkzTabLine{ d, +, d, -,d , +, d, -, d, +, d,  -, d, +, d}
\end{tikzpicture}
\label{signChart82}
\caption{Ratio test for $\kappa =8$, $d=2$.}
\end{figure}
\begin{figure}[H]
\begin{tikzpicture}
   \tkzTabInit[lgt = 3.3, espcl = 1.8, deltacl = 0.7]{$E$ / 1.1,  $R_{d=3}(E)$ / 0.7}{$0$, $\left(\frac{\sqrt{2-\sqrt{2}}}{2}\right)^3$,  $\frac{1}{2\sqrt{2}}$, $\frac{\sqrt{2-\sqrt{2}}}{2}$, $\frac{1}{\sqrt{2}}$, $\left(\frac{\sqrt{2+\sqrt{2}}}{2}\right)^3$, $\frac{\sqrt{2+\sqrt{2}}}{2}$, $1$}
   \tkzTabLine{d, +, d, -, d, +, d, -, d, +, d,  -, d, +, d}
\end{tikzpicture}
\label{signChart83}
\caption{Ratio test for $\kappa =8$, $d=3$.}
\end{figure}

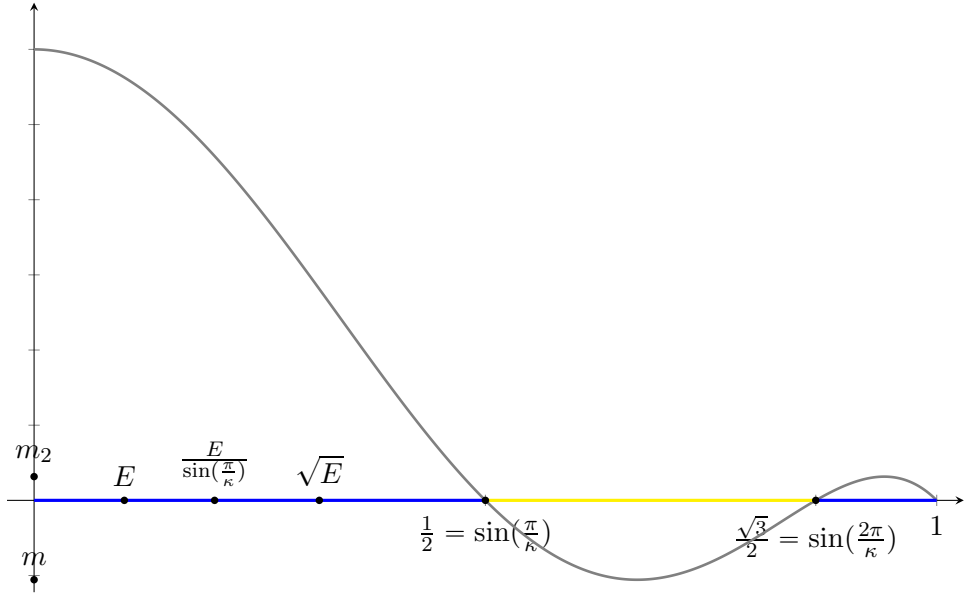
\begin{figure}[H]
\begin{tikzpicture}[domain=0:10]
\begin{axis}
[
clip = true, 
clip mode=individual, 
axis x line = middle, 
axis y line = middle, 
xlabel={}, 
xticklabels={0, $\frac{1}{2} = \sin(\frac{\pi}{\kappa})$, $\frac{\sqrt{3}}{2} = \sin(\frac{2\pi}{\kappa})$, $1$},
xtick={0, 0.5, 0.866, 1},
yticklabels={},
ytick={},
ylabel={}, 
ylabel style={at=(current axis.above origin), anchor=west}, 
y=2cm,
x=12cm,
enlarge y limits={rel=0.03}, 
enlarge x limits={rel=0.03}, 
ymin = -0.5, 
ymax = 3.2, 
after end axis/.code={\path (axis cs:0,0) ;}]

\addplot[color=gray,line width = 1.0 pt, samples=1000, domain=0:1] ({x},{(1-x^2)*(1-4*x^2)*(3-4*x^2)});
\draw [color=blue, very thick] (0,0) -- (0.5,0);
\draw [color=yellow, very thick] (0.5,0) -- (0.866,0);
\draw [color=blue, very thick] (0.866,0)--(1,0);
\node[label={$E$},circle,fill,inner sep=1pt, black] at (axis cs:0.1,0) {};
\node[label={$\frac{E}{\sin(\frac{\pi}{\kappa})}$},circle,fill,inner sep=1pt, black] at (axis cs:0.2,0) {};
\node[label={$\sqrt{E}$},circle,fill,inner sep=1pt, black] at (axis cs:0.316,0) {};
\node[label={},circle,fill,inner sep=1pt, black] at (axis cs:0.866,0) {};
\node[label={},circle,fill,inner sep=1pt, black] at (axis cs:0.5,0) {};
\node[label={$m$},circle,fill,inner sep=1pt, black] at (axis cs:0,-0.528) {};\node[label={$m_2$},circle,fill,inner sep=1pt, black] at (axis cs:0,0.158) {};
\end{axis}
\end{tikzpicture}
\caption{Plot of $h(x) = 2(1-x^2)(1-4x^2)(3-4x^2)$ corresponding to $\kappa=6$.}
\label{fig102}
\end{figure}

\subsection{Central band test}
\label{sub_central}

The test implies a strict Mourre estimate on a subset of $(0, \sin(\pi/ \kappa))$, at least for some values of $\kappa$ and $d$. We consider only $E >0$. Consider the function of $1$ variable
\begin{equation}
\label{functionH}
h: [0,1] \mapsto \R, \quad h:x \mapsto (-1)^{\kappa/2-1} 2^{\kappa-1} \prod_{l=1} ^{\kappa /2} (\sin^2(l \pi / \kappa)-x^2).
\end{equation}
$h$ is graphed in Figure \ref{fig102} for $\kappa=6$. The positive roots of $h$ are $\{ \sin(j\pi/ \kappa): j = 1, ..., \kappa/2 \}$. $h(0) = (-1)^{\kappa/2-1} \kappa$ which is positive if $\kappa =4n+2$ and negative if $\kappa =4n+4$, $n\in \N$. $h'$ is strictly negative (resp.\ positive) on $(0, \sin(\pi/ \kappa))$, depending on whether $\kappa = 4n+2$ (resp. $4n+4$). Thus, if $\kappa = 4n+2$, $h$ is strictly decreasing from $\kappa$ to $0$ on $(0, \sin(\pi/ \kappa)$, whereas if $\kappa = 4n+4$, $h$ is strictly increasing from $-\kappa$ to $0$ on $(0, \sin(\pi/ \kappa))$.

\noindent{\textbf{Case $\kappa = 4n+2$:} } Denote $m = \min_{x \in [\sin(\pi / \kappa), 1]} h(x)$. By counting the number of roots of $h$, $\kappa \geq 4$ implies $m < 0$. When $E_2 = ... = E_d = \sin(\pi/ \kappa)$, $E_1 = E/ \sin^{(d-1)}(\pi / \kappa)$. One solves the inequality $h(E/ \sin^{(d-1)}(\pi / \kappa)) > (d-1)|m|$. Say $[0,E_m)$ solves this inequality then $(0,E_m) \subset \boldsymbol{\mu}_{A_{\kappa}} ^{+} (D)$.

\noindent{\textbf{Case $\kappa = 4n+4$:} } Denote $m = \max_{x \in [\sin(\pi / \kappa), 1]} h(x)$. By counting the number of roots of $h$, $\kappa \geq 4$ implies $m > 0$. When $E_2 = ... = E_d = \sin (\pi/ \kappa)$, $E_1 = E/ \sin^{(d-1)}(\pi / \kappa)$. One solves the inequality $h(E/ \sin^{(d-1)}(\pi / \kappa)) < -(d-1)m$. Say $[0,E_m)$ solves this inequality then $(0,E_m) \subset \boldsymbol{\mu}_{A_{\kappa}} ^{-} (D)$.

Results of the central band test are summarized in Table \ref{table:101909} for $\kappa = 4,6,8$, $d=2,3$. Results for the case $\kappa=4$, $d=2$ are consistent with Lemma \ref{lemK=4}. Figure \ref{fig102} may be of use in understanding the underlying idea behind this test. By way of example, fix $\kappa = 6$, $d=2$. Fix an energy $E = E_1 E_2 \in (0, \sin(\pi / \kappa))$. We start with the point $(E_1,E_2) = (E,1) \in S_E$. As $E_2$ decreases below $1$, $E_1$ increases above $E$. Eventually we get to the point $(E_1,E_2) = (E/ \sin(\pi / \kappa),\sin(\pi / \kappa) )$. Further decreasing $E_2$ and increasing $E_1$, the 2 coordinates cross at the point $(E_1,E_2) = (\sqrt{E},\sqrt{E})$. After crossing $E_1$ and $E_2$ interchange roles.

\begin{table}[H]
    \begin{tabular}{c|c|c|c} 
      $\kappa$ & $m$ & $d=2, E_m$ & $d=3, E_m$    \\ [0.4em]
      \hline
      $4$  & $ 0.5$ & $ \frac{1}{2} \sqrt{\frac{3}{2} - \frac{1}{\sqrt{2}}} \simeq 0.44 $ & $ \frac{\sqrt{3-\sqrt{3}}}{4} \simeq 0.28$   \\ [0.4em]
      $6$  & $ -\frac{10+7\sqrt{7}}{27} \simeq -1.056$ & $ \simeq 0.21 $ & $ \simeq 0.089$   \\ [0.4em]
      $8$  & $\simeq 1.56$ & $ \simeq 0.12 $ & $ \simeq 0.038$   \\ [0.4em]
     \end{tabular}
  \caption{Central band test for $\kappa =4,6,8$, $d=2,3$. $(0, E_m) \subset \boldsymbol{\mu}_{A_{\kappa}}(D)$. }
     \label{table:101909}
\end{table}

\subsection{Second band test} 
\label{sub_second_band}

The procedure of the previous test is adjusted to probe for energies $E \in (\sin(\pi/ \kappa), \sin(2\pi / \kappa))$.

\noindent{\textbf{Case $\kappa = 4n+2$:} } Let $m_2 = \max_{x \in [\sin(2\pi / \kappa), 1]} h(x)$. $\kappa \geq 6 \Rightarrow m_2 > 0$. Solve $h(x) < -(d-1)m_2$. Say $(A_2,B_2)$ is the solution to this inequality that belongs to $[\sin(\pi/ \kappa), \sin(2\pi/ \kappa)]$. 
Then solve $h(E/ \sin^{(d-1)}(2\pi / \kappa)) < -(d-1)m_2$. Say $(D_{m_2},E_{m_2})$ solves this inequality then $(A_2,B_2) \cap (D_{m_2},E_{m_2}) \subset \boldsymbol{\mu}_{A_{\kappa}} ^{-} (D)$.

\noindent{\textbf{Case $\kappa = 4n+4$:} } Let $m_2 = \min_{x \in [\sin(2\pi / \kappa), 1]} h(x)$. $\kappa \geq 6 \Rightarrow m_2 < 0$. Solve $h(x) > (d-1)|m_2|$. Say $(A_2,B_2)$ is the solution to this inequality that belongs to $[\sin(\pi/ \kappa), \sin(2\pi/ \kappa)]$. 
Then solve $h(E/ \sin^{(d-1)}(2\pi / \kappa)) > (d-1)|m_2|$. Say $(D_{m_2},E_{m_2})$ solves this inequality then $(A_2,B_2) \cap (D_{m_2},E_{m_2}) \subset \boldsymbol{\mu}_{A_{\kappa}} ^{+} (D)$.

Results of the second band test are summarized in Table \ref{table:101909098} for $\kappa = 6,8$, $d=2,3$. For $\kappa = 8$, $d=3$, the test does not yield anything, but this is actually in line with the numerical results of Table \ref{tab:table10144}, in the sense that there is no band within $(\sin(\pi/ 8), \sin(2\pi / 8)) \simeq (0.383, 0.707)$. 
    
\begin{table}[H]
    \begin{tabular}{c|c|c|c} 
      $\kappa$ & $m_2$ & $d=2, (A_2,B_2) \cap (D_{m_2},E_{m_2})$ & $d=3, (A_2,B_2) \cap (D_{m_2},E_{m_2})$    \\ [0.4em]
      \hline
      $6$  & $ \frac{7\sqrt{7}-10}{27} \simeq 0.316$ & $ \simeq (0.5283, 0.8236)\cap (0.4575, 0.7133)$ & $ \simeq (0.563,0.780) \cap (0.422,0.585)$ \\ [0.4em]
      $8$  & $\simeq -0.692$ &  $\simeq (0.421,0.649) \cap (0.297,0.458)$ & $ \simeq (0.480, 0.577) \cap (0.240, 0.288) = \emptyset$   \\ [0.4em]
     \end{tabular}
  \caption{Second band test. $\kappa =6,8$, $d=2,3$. $(A_2,B_2) \cap (D_{m_2},E_{m_2}) \subset \boldsymbol{\mu}_{A_{\kappa}}(D)$. }
     \label{table:101909098}
\end{table}
Of course one may perform a third band test for $\kappa = 8$. Without going into all the details, we would find that in $2d$, $(A_3,B_3) \cap (D_{m_3} , E_{m_3}) \simeq (0.729, 0.898) \cap (0.674,0.830) \subset \boldsymbol{\mu}_{A_{\kappa}} (D_{d=2})$, which is consistent with the numerical evidence of Table \ref{tab:table1012} ; whereas in $3d$ we would find that  $(A_3,B_3) \cap (D_{m_3} , E_{m_3}) \simeq (0.756, 0.870) \cap (0.645,0.743) = \emptyset$. This time however the numerical evidence of Table \ref{tab:table10144} suggests that there is a band. This just means that the test is too coarse. 

\section{Isomorphism between Molchanov-Vainberg and Standard Laplacians in dimension $2$.}
\label{Iso2D}

\subsection{Definition of the isomorphism between $\Delta$ and $D$ in dimension $2$.}
Let $\mathcal{G}_B:=(\mathcal{E}_B, \mathcal{V}_B)$ and $\mathcal{G}_R:=(\mathcal{E}_R, \mathcal{V}_R)$ ($B$ for blue, $R$ for red) denote the graphs whose vertices are 
\[\mathcal{V}_B:=\{ (n_1,n_2) \in \Z^2: n_1+n_2 \ \mathrm{ is \ even}\} \quad \mathrm{ \, and \, } \quad \mathcal{V}_R:=\{ (n_1,n_2) \in \Z^2: n_1+n_2 \ \mathrm{ is \ odd}\}\]
We connect the vertices of $\mathcal{G}_B$ and $\mathcal{G}_R$ by the diagonals. Namely, for $X\in \{B,G\}$, we set
 \[\mathcal{E}_X((n_1,n_2), (n_1',n_2')) =1, \mathrm{\, if \,} |n_1-n_1'| = |n_2-n_2'|=1\]
 and $0$ otherwise. The two graphs are illustrated in Figure \ref{figure:lattice}.  As $\Z^2$ is the disjoint union of $\mathcal{V}_B$ and $\mathcal{V}_R$, we have $\ell^2(\Z^2) = \ell^2(\mathcal{V}_B) \oplus \ell^2(\mathcal{V}_R)$. We have that $D$ leaves invariant $\ell^2(\mathcal{V}_B)$ and also $\ell^2(\mathcal{V}_R)$.
Denote $D_B:= \left.D \right|_{\ell^2(\mathcal{V}_B)}$ and $D_R:= \left.D \right|_{\ell^2(\mathcal{V}_R)}$. On the other hand, note that 
\[ \ell^2(\mathcal{V}_R)\simeq \ell^2(\mathcal{V}_B)\simeq \ell^2(\Z^2).\]
For the first one, a simple translation sends $\ell^2(\mathcal{V}_R)$ unitarily onto $\ell^2(\mathcal{V}_B)$. Let us clarify the second isomorphism. Let $\pi_0: \mathcal{V}_B \mapsto \Z^2$, $\pi_0: (n_1,n_2) \mapsto \frac{1}{2}(n_1+n_2,n_2-n_1)$. $\pi_0$ is simply a clockwise rotation of $45^{\circ}$ followed by a $1/ \sqrt{2}$ scaling, i.e.\ $\pi_0 = e^{-\i \pi/4} / \sqrt{2}$. The inverse is $\pi_0^{-1}: \Z^2 \mapsto \mathcal{V}_B$, $\pi_0 ^{-1}: (n_1,n_2) \mapsto (n_1-n_2,n_1+n_2)$. $\pi_0$ induces an isometry $\pi: \ell^2(\mathcal{V}_B)\mapsto \ell^2(\Z^2)$, $\pi: \psi \mapsto \pi(\psi)$, whereby $\pi (\psi) (n) = \psi(\pi_0^{-1}n)$, $n=(n_1,n_2)$. The induced inverse is $\pi^{-1}: \ell^2(\Z^2) \mapsto \ell^2(\mathcal{V}_B)$, whereby $\pi^{-1} (\psi) (n) = \psi(\pi_0 n)$.
With these definitions one secures the identity
\begin{equation}
\label{key_iso_D}
2 D_B  = \pi^{-1} \Delta  \pi \quad \Leftrightarrow   \quad \pi  D_B \pi^{-1}   = \Delta /2 .
\end{equation}
Let $n=(n_1,n_2)$. If $V_B(\cdot)$ is the operator of multiplication by $V_B:\mathcal{V}_B \to \C$ on $\ell^2(\mathcal{V}_B)$ then $\pi V_B (\cdot) \pi^{-1}$ is the operator of multiplication  $V_B(\pi_0^{-1}\cdot)$ on  $\ell^2(\Z^2)$. If $V_B$ is \emph{radial}, i.e.\ $V_B(n_1,n_2) = V_B(\langle n \rangle)$, then $\pi V_B \pi^{-1}$ is radial and $\pi V_B \pi^{-1}(n_1,n_2) = V_B  (\sqrt{2}\langle n \rangle)$. Conversely if $V(\cdot)$   is the operator of multiplication by $V: \Z^2 \to \C$ on $\ell^2(\Z^2)$, then $\pi ^{-1} V(\cdot) \pi$ is the operator of multiplication $V(\pi_0 \cdot)$ on $\ell^2(\mathcal{V}_B)$. If $V(n_1,n_2) = V(\langle n \rangle)$ then $\pi ^{-1} V \pi (n_1, n_2) = V (\langle n \rangle / \sqrt{2})$.

To sum up, for a potential $V$ defined on $\Z^2$, set $V_B:= \left.V \right|_{\mathcal{V}_B}$ and $V_R:= \left.V \right|_{\mathcal{V}_R}$. One has 
\[2D + V = (2D_B + V_B(\cdot)) \oplus (2D_R + V_R(\cdot)) \cong (\Delta + V_B(\cdot)) \oplus (\Delta + V_R(\cdot)),\]
as operators on $\ell^2(\Z^2) = \ell^2(\mathcal{V}_B) \oplus \ell^2(\mathcal{V}_R) \cong \ell^2(\Z^2) \oplus \ell^2(\Z^2)$.

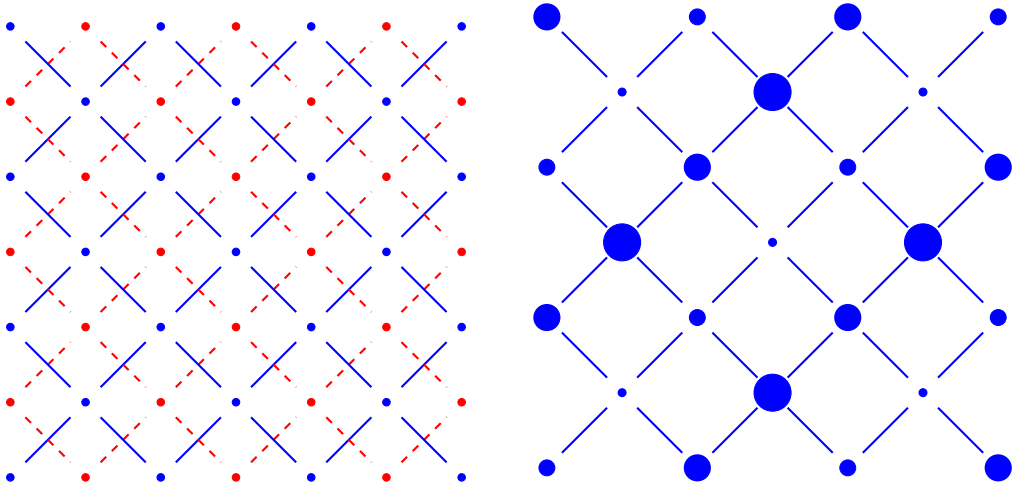
\begin{figure}[ht]

\begin{tikzpicture}
    [
        dot/.style={circle,draw=black, fill,inner sep=1pt},
    ]

\foreach \x in {-3,-1,1,3}{
 \node[dot,blue] at (\x,3){ };
    \node[dot,red] at (\x,2){ };
    \node[dot,blue] at (\x,1){ };
    \node[dot,red] at (\x,0){ };
    \node[dot,blue] at (\x,-1){ };
    \node[dot,red] at (\x,-2){ };
    \node[dot,blue] at (\x,-3){ };
}

\foreach \x in {-2,0,2}{
   \node[dot,red] at (\x,-3){ };
    \node[dot,blue] at (\x,-2){ };
    \node[dot,red] at (\x,1){ };
    \node[dot,blue] at (\x,0){ };
    \node[dot,red] at (\x,-1){ };
    \node[dot,blue] at (\x,2){ };
     \node[dot,red] at (\x,3){ };
}

\foreach \x in {-2.8,...,1.2}
    \draw[thick,red,dashed] (\x,\x+1) -- (\x+.6,\x+.6+1);
    \foreach \x in {-2.8,...,2.2}
    \draw[thick,blue] (\x,\x) -- (\x+.6,\x+.6);    
        \foreach \x in {-1.8,...,2.2}
    \draw[thick,red,dashed] (\x,\x-1) -- (\x+.6,\x+.6-1);    
            \foreach \x in {-2.8,...,0.2}
    \draw[thick,blue] (\x,\x+2) -- (\x+.6,\x+.6+2);    
                \foreach \x in {-0.8,...,2.2}
    \draw[thick,blue] (\x,\x-2) -- (\x+.6,\x+.6-2);    
                    \foreach \x in {1.2,...,2.2}
    \draw[thick,blue] (\x,\x-4) -- (\x+.6,\x+.6-4);    
                        \foreach \x in {-2.8,...,-1.8}
    \draw[thick,blue] (\x,\x+4) -- (\x+.6,\x+.6+4);    
\foreach \x in {-2.8,...,-0.8}
    \draw[thick,red,dashed] (\x,\x+3) -- (\x+.6,\x+.6+3);   
                \foreach \x in {0.2,...,2.2}
    \draw[thick,red,dashed] (\x,\x-3) -- (\x+.6,\x+.6-3);    
            \foreach \x in {-2.8}
    \draw[thick,red,dashed] (\x,\x+5) -- (\x+.6,\x+.6+5);    
                \foreach \x in {2.2}
    \draw[thick,red,dashed] (\x,\x-5) -- (\x+.6,\x+.6-5);
        \foreach \x in {-2.8,...,2.2}
    \draw[thick,blue] (\x,-\x) -- (\x+.6,-\x-.6);    
            \foreach \x in {-2.8,...,1.2}
    \draw[thick,red,dashed] (\x,-\x-1) -- (\x+.6,-\x-.6-1);    
                \foreach \x in {-2.8,...,0.2}
    \draw[thick,blue] (\x,-\x-2) -- (\x+.6,-\x-.6-2);    
                    \foreach \x in {-2.8,...,-0.8}
    \draw[thick,red,dashed] (\x,-\x-3) -- (\x+.6,-\x-.6-3);    
                    \foreach \x in {-2.8,...,-1.8}
    \draw[thick,blue] (\x,-\x-4) -- (\x+.6,-\x-.6-4);    
                    \foreach \x in {-2.8}
    \draw[thick,red,dashed] (\x,-\x-5) -- (\x+.6,-\x-.6-5);    
            \foreach \x in {-1.8,...,2.2}
    \draw[thick,red,dashed] (\x,-\x+1) -- (\x+.6,-\x-.6+1);    
                        \foreach \x in {-0.8,...,2.2}
    \draw[thick,blue] (\x,-\x+2) -- (\x+.6,-\x-.6+2);    
                \foreach \x in {0.2,...,2.2}
    \draw[thick,red,dashed] (\x,-\x+3) -- (\x+.6,-\x-.6+3);    
                            \foreach \x in {1.2,...,2.2}
    \draw[thick,blue] (\x,-\x+4) -- (\x+.6,-\x-.6+4);
                    \foreach \x in {2.2}
    \draw[thick,red,dashed] (\x,-\x+5) -- (\x+.6,-\x-.6+5);
    
\end{tikzpicture}
\quad \quad 
\begin{tikzpicture}
    [
        dot/.style={circle,draw=black, fill,inner sep=0pt},
    ]

\foreach \x in {-3,-1,1,3}{
     \node[dot,blue] at (\x,3){ };
    \node[dot,blue] at (\x,1){ };
    \node[dot,blue] at (\x,-1){ };
    \node[dot,blue] at (\x,-3){ };
}

\foreach \x in {-2,0,2}{
    \node[dot,blue] at (\x,-2){ };
    \node[dot,blue] at (\x,0){ };
    \node[dot,blue] at (\x,2){ };
}

\draw node[fill,circle,minimum size=1pt, blue] at (-3,-1) {};
\draw node[fill,circle,minimum size=1pt, blue] at (-1,1) {};
\draw node[fill,circle,minimum size=1.2pt, blue] at (1,-1) {};
\draw node[fill,circle,minimum size=1.2pt, blue] at (3,1) {};
\draw node[fill,circle,minimum size=1.2pt, blue] at (1,3) {};
\draw node[fill,circle,minimum size=1.2pt, blue] at (-1,-3) {};
\draw node[fill,circle,minimum size=1.2pt, blue] at (3,-3) {};
\draw node[fill,circle,minimum size=1.2pt, blue] at (-3,3) {};

\draw[fill, blue] (0,0) circle[radius=1.5pt];%
\draw[fill, blue] (-2,2) circle[radius=1.5pt];%
\draw[fill, blue] (2,2) circle[radius=1.5pt];%
\draw[fill, blue] (2,-2) circle[radius=1.5pt];%
\draw[fill, blue] (-2,-2) circle[radius=1.5pt];%

\draw[fill, blue] (0,2) circle[radius=7pt];%
\draw[fill, blue] (2,0) circle[radius=7pt];%
\draw[fill, blue] (0,-2) circle[radius=7pt];%
\draw[fill, blue] (-2,0) circle[radius=7pt];%
 
\draw[fill, blue] (-1,-1) circle[radius=3pt];%
\draw[fill, blue] (-3,1) circle[radius=3pt];%
\draw[fill, blue]  (-1,3) circle[radius=3pt];%
\draw[fill, blue]  (1,1) circle[radius=3pt];%
\draw[fill, blue] (3,-1) circle[radius=3pt];%
\draw[fill, blue] (1,-3) circle[radius=3pt];%
\draw[fill, blue] (-3,-3) circle[radius=3pt];%
\draw[fill, blue] (3,3) circle[radius=3pt];%

    \foreach \x in {-2.8,...,2.2}
    \draw[thick, blue] (\x,\x) -- (\x+.6,\x+.6);    
            \foreach \x in {-2.8,...,0.2}
    \draw[thick,blue] (\x,\x+2) -- (\x+.6,\x+.6+2);    
                \foreach \x in {-0.8,...,2.2}
    \draw[thick,blue] (\x,\x-2) -- (\x+.6,\x+.6-2);    
                    \foreach \x in {1.2,...,2.2}
    \draw[thick,blue] (\x,\x-4) -- (\x+.6,\x+.6-4);    
                        \foreach \x in {-2.8,...,-1.8}
    \draw[thick,blue] (\x,\x+4) -- (\x+.6,\x+.6+4);    
        \foreach \x in {-2.8,...,2.2}
    \draw[thick,blue] (\x,-\x) -- (\x+.6,-\x-.6);    
                \foreach \x in {-2.8,...,0.2}
    \draw[thick,blue] (\x,-\x-2) -- (\x+.6,-\x-.6-2);    
                    \foreach \x in {-2.8,...,-1.8}
    \draw[thick,blue] (\x,-\x-4) -- (\x+.6,-\x-.6-4);    
                        \foreach \x in {-0.8,...,2.2}
    \draw[thick,blue] (\x,-\x+2) -- (\x+.6,-\x-.6+2);    
                            \foreach \x in {1.2,...,2.2}
    \draw[thick,blue] (\x,-\x+4) -- (\x+.6,-\x-.6+4);
    
\end{tikzpicture}
  \caption{Left: $\Z^2= \mathcal{G}_B \oplus \mathcal{G}_R$. Right: rotation by $45^{\circ}$ of a ``periodic'' pattern of $2\times2$ squares containing vertices (S, M, L, XL), see example \ref{example11}}
\label{figure:lattice}
\end{figure}

\begin{remark} It is also interesting to depict the isomorphism in Fourier space. Let 
\[\sigma: L^2([-\pi,\pi]^2,d\xi) \mapsto L^2([-\pi,\pi]^2,d\xi), \quad (\sigma f)(\xi_1,\xi_2) = f(\xi_1+\xi_2, \xi_1-\xi_2)\] and $(\sigma^{-1} f)(\xi_1,\xi_2) = f\left((\xi_1+\xi_2)/2, (\xi_1-\xi_2)/2 \right)$. Then using the relationship $2\cos(\xi_1) \cos(\xi_2) = \cos(\xi_1 + \xi_2) + \cos(\xi_1 -\xi_2)$ one gets
$$\mathcal{F}  (2D) \mathcal{F} ^{-1} = \sigma \mathcal{F} \Delta \mathcal{F}^{-1} \sigma^{-1} .$$
\end{remark}

\subsection{Definition of the conjugate operator $\pi A_{2\kappa}\pi^{-1}$ and regularity. }
\label{sub_def_a232} First we must signal a small issue of well-posedness. Let $\kappa = (\kappa_1,\kappa_2) \in (\N^*)^2$. The shifts $\{S_1^{\kappa_1}, S_2^{\kappa_2 }\}$ are invariant on $\ell^2(\mathcal{V}_B)$ and $\ell^2(\mathcal{V}_R)$ if and only if $\kappa_1,\kappa_2$ are even. The conjugate operator $A_{\kappa}$ is invariant on $\ell_0(\mathcal{V}_B)$ and $\ell_0(\mathcal{V}_R)$ if and only if $\kappa_1, \kappa_2$ are both even. Let $\alpha \kappa = (\alpha \kappa_1, \alpha \kappa_2)$, $\alpha \in \R$.

To determine the expression of $\pi A_{2\kappa} \pi^{-1}$, one first computes, for $n \in \N^*$: 
$$\pi  S_1 ^{2n} \pi^{-1} = S_1^{n} S_2 ^{-n}, \quad \pi  S_1^{-2n} \pi^{-1} = S_1^{-n} S_2 ^{n}, \quad \pi S_2 ^{2n} \pi^{-1} = S_1^{n} S_2 ^{n}, \quad \pi S_2^{-2n} \pi^{-1} = S_1^{-n} S_2 ^{-n},$$
$$\pi N_1 \pi^{-1} = N_1 - N_2, \quad \pi N_2 \pi^{-1} = N_1 + N_2.$$
One computes $\pi A_{2\kappa} \pi^{-1}$ on  $\ell_0(\Z^d)$: 
\begin{equation}
\begin{aligned}
\label{PIinversePI}
\pi A_{2\kappa} \pi^{-1} &= \frac{1}{4\i} (N_1 - N_2) \left( S_1^{\kappa_1} S_2 ^{-\kappa_1} - S_1^{-\kappa_1} S_2 ^{\kappa_1} \right) + \frac{1}{4\i}\left( S_1^{\kappa_1} S_2 ^{-\kappa_1} - S_1^{-\kappa_1} S_2 ^{\kappa_1} \right) (N_1 - N_2) \\
& \quad + \frac{1}{4\i} (N_1+N_2) \left( S_1^{\kappa_2} S_2 ^{\kappa_2} - S_1^{-\kappa_2} S_2 ^{-\kappa_2} \right) +\frac{1}{4\i}  \left( S_1^{\kappa_2} S_2 ^{\kappa_2} - S_1^{-\kappa_2} S_2 ^{-\kappa_2} \right) (N_1+N_2). 
\end{aligned}
\end{equation}
Of course \eqref{PIinversePI} may be expressed in Fourier space. By computing on smooth functions that are $2\pi-$periodic, one has $\mathcal{F} S_j ^{\kappa_j} \mathcal{F}^{-1} = e^{\i \kappa_j \xi_j}$ and $\mathcal{F} \i N_j \mathcal{F}^{-1} = \partial / \partial \xi_j$. When $\kappa_1 = \kappa_2$, 
\begin{align}
\nonumber 
& \hspace*{-1cm}\mathcal{F} \pi A_{2\kappa} \pi^{-1} \mathcal{F}^{-1}
\\
\label{ConjFourier}
&=-\i \cos(\kappa \xi_2) \left [ \sin(\kappa \xi_1) \frac{\partial}{\partial \xi_1}+\frac{\partial}{\partial \xi_1} \sin( \kappa \xi_1)  \right] - \i \cos( \kappa \xi_1) \left [ \sin( \kappa \xi_2) \frac{\partial}{\partial \xi_2}+\frac{\partial}{\partial \xi_2} \sin(\kappa \xi_2)  \right].
\end{align}
\eqref{PIinversePI} and \eqref{ConjFourier} are to be compared with \eqref{generatorDilations_011} and \eqref{A_fourier}. Regularity with respect to $\Delta$ is excellent:

\begin{proposition}  $\Delta \in C^{\infty}(\pi A_{2\kappa}\pi^{-1})$ for all $\kappa \in (\N^*)^2$.
\end{proposition}

\begin{proof}
Thanks to \eqref{dAk}, one computes on $\ell_0(\Z^d)$ that $\left [  \Delta , \i \pi A_{2\kappa} \pi^{-1} \right ]$ equals
\begin{equation*}
\frac{\Delta}{2} \left [ (-1)^{\kappa_1-1} 2^{2\kappa_1}  \prod_{l = 1} ^{\kappa_1} \left( \sin^2 \left(\frac{l \pi}{2\kappa_1} \right) - \pi \Delta_1^2 \pi^{-1} \right) + (-1)^{\kappa_2-1} 2^{2\kappa_2} \prod_{l = 1} ^{\kappa_2} \left( \sin^2 \left(\frac{l \pi}{2\kappa_2} \right) - \pi \Delta_2^2 \pi^{-1} \right) \right],
\end{equation*}
where 
$$ \pi \Delta_1^2 \pi^{-1} = 1/2 + (S_1 S_2^{-1} + S_1 ^{-1} S_2) / 4, \quad \pi \Delta_2^2 \pi^{-1} = 1/2 + (S_1 S_2 + S_1 ^{-1} S_2 ^{-1})/4. $$ 
Extending by density to all vectors in $\ell^2(\Z^d)$ gives $\Delta \in C^1(\pi A_{2\kappa} \pi^{-1})$. An obvious induction argument gives the $C^{\infty} (\pi A_{2\kappa} \pi^{-1})$ regularity. 
\qed
\end{proof}

The next Lemma states the $C^1(\pi A_{2\kappa} \pi^{-1})$ and $C^{1,1}(\pi A_{2\kappa} \pi^{-1})$ regularity for the potential.
\begin{Lemma}
Let $\kappa \in (\N^*)^2$ be given. Suppose
\begin{equation}
\label{CriterionA4441}  
(n_1-n_2) (V - \tau_1^{\kappa_1} \tau_2^{-\kappa_1} V)(n) \quad and \quad (n_1+n_2) (V - \tau_1^{\kappa_2} \tau_2^{\kappa_2} V)(n) = O (1), \textrm{as } |n|\to\infty.
\end{equation} 
Then $V(\cdot) \in C^1(\pi A_{2\kappa} \pi^{-1})$. Now let $V_s$ satisfy \eqref{Firstcriterion}. Then $V_s(\cdot) \in C^{1,1} (\pi A_{2\kappa} \pi^{-1})$. Let $V_l (n) = o(1)$ as $|n|\to\infty$. Then $V_l(\cdot) \in C^{1,1} (\pi A_{2\kappa} \pi^{-1})$ whenever
\begin{equation} 
\label{Secondcriterion22}
\int_{1} ^{\infty} \sup _{r < | n| < 2r} \big | (V_{\text{l}} - \tau_1^{\kappa_1} \tau_2 ^{-\kappa_1} V_{\text{l}})(n) \big | dr < \infty \quad \text{and} \quad \int_{1} ^{\infty} \sup _{r < | n| < 2r} \big | (V_{\text{l}} - \tau_1^{\kappa_2} \tau_2 ^{\kappa_2} V_{\text{l}})(n) \big | dr < \infty.
\end{equation}
\end{Lemma}

\subsection{Transposing the results for $D_B$ to $\Delta$.}
\eqref{key_iso_D} implies 
\begin{equation}
\label{implication1}
\left [  \Delta , \i \pi A_{2\kappa} \pi^{-1} \right ]_{\circ} = \pi \left [ 2 D_B , \i  A_{2\kappa} \right ]_{\circ} \pi^{-1}.
\end{equation}
Let $\theta \in C_c ^{\infty}(\R)$ be compactly supported. By the Helffer-Sj\"ostrand formula, $\theta \left( 2 D_B \right) = \pi^{-1} \theta \left( \Delta \right) \pi$ and $\theta \left( 2D \right)  = \theta \left( 2D_B \right)  \oplus \theta \left( 2D_R \right).$ In terms of the Mourre estimate this means that 
$$\theta \left( 2 D_B \right) [ 2 D_B,  \pm \i  A_{2\kappa} ]_{\circ} \theta \left( 2 D_B \right) \geq \gamma \theta \left( 2 D_B \right) \Longleftrightarrow \theta \left( \Delta \right) \left[  \Delta, \pm \i \pi A_{2\kappa} \pi^{-1} \right]_{\circ} \theta \left( \Delta \right) \geq \gamma \theta \left( \Delta \right).$$
In other words, $\boldsymbol{\mu}_{\pi A_{2\kappa} \pi^{-1}} ^{\pm}(\Delta) = \boldsymbol{\mu}_{A_{2\kappa}} ^{\pm} (2D_B)$. This formula says that whenever a spectral interval exhibits operator positivity for $2D_B$ with respect to $A_{2\kappa}$ it can be transferred into operator positivity for $\Delta$ with respect to $\pi A_{2\kappa} \pi^{-1}$ and vice versa.

\noindent \underline{\textit{Example:}}  We treat the case $\kappa = (\kappa_1,\kappa_2)=(1,1)$.  A direct calculation using \eqref{ConjFourier} gives:
\begin{equation}
\label{calcul_independent}
\mathcal{F} \left [  \Delta , \i \pi A_{2\kappa} \pi^{-1} \right ]_{\circ} \mathcal{F}^{-1} = [\cos(\xi_1) + \cos(\xi_2), \i \eqref{ConjFourier} ]_{\circ} = 2\sin^2(\xi_1) \cos(\xi_2) + 2\sin^2(\xi_2) \cos(\xi_1). 
\end{equation}
On the other hand, using \eqref{implication1} and \eqref{COMMUTATOR_D} gives
\begin{equation}
\label{commutateurDelta2}
\begin{aligned}
[\Delta, \i \pi A_{2\kappa} \pi^{-1} ]_{\circ} &= 
4 \pi D_B \left( (1-\Delta_1^2) +(1-\Delta_2^2) \right) \pi^{-1} = 2 \Delta \left( 1 - \Delta_1 \Delta_2 \right) \\
&=  \mathcal{F}^{-1} 2\left(\cos(\xi_1) + \cos(\xi_2) \right) \left( 1- \cos(\xi_1) \cos(\xi_2) \right) \mathcal{F}.
\end{aligned}
\end{equation} 
We see that \eqref{commutateurDelta2} and \eqref{calcul_independent} agree. This confirms \eqref{implication1}.

\subsection{Definition of the conjugate operator $\pi^{-1} A_{\kappa}\pi$ and regularity. } This time around $\pi^{-1} A_{\kappa}\pi$ is well defined on $\ell_0( \mathcal{V}_B)$, for all $\kappa \in (\N^*)^2$. One has for $n \in \N^*$,
$$\pi^{-1} S_1 ^{n} \pi = S_1^{n} S_2 ^{n}, \quad \pi^{-1} S_1^{-n}\pi = S_1^{-n} S_2 ^{-n}, \quad \pi^{-1} S_2 ^{n} \pi = S_1^{-n} S_2 ^{n}, \quad \pi^{-1} S_2^{-n} \pi = S_1^{n} S_2 ^{-n},$$
$$\pi^{-1} N_1 \pi = (N_1 + N_2)/2, \quad \pi^{-1} N_2 \pi = (N_2 - N_1)/2.$$
One computes $\pi^{-1} A_{\kappa}\pi$ on $\ell_0(\mathcal{V}_B)$: it is equal to 
\begin{equation}
\begin{aligned}
\label{A_2d}
 &(8\i)^{-1} (N_1 + N_2) \left( S_1^{\kappa_1} S_2 ^{\kappa_1} - S_1^{-\kappa_1} S_2 ^{-\kappa_1} \right) +(8\i)^{-1}  \left( S_1^{\kappa_1} S_2 ^{\kappa_1} - S_1^{-\kappa_1} S_2 ^{-\kappa_1} \right) (N_1 +N_2) \\
& \quad + (8\i)^{-1}  (N_2-N_1) \left( S_1^{-\kappa_2} S_2 ^{\kappa_2} - S_1^{\kappa_2} S_2 ^{-\kappa_2} \right) + (8\i)^{-1} \left( S_1^{-\kappa_2} S_2 ^{\kappa_2} - S_1^{\kappa_2} S_2 ^{-\kappa_2} \right) (N_2-N_1). 
\end{aligned}
\end{equation}
The proof of the following 2 regularity results are left to the reader. 
\begin{proposition}  $D \in C^{\infty}(\pi^{-1} A_{\kappa}\pi)$ for all $\kappa \in (\N^*)^2$.
\end{proposition}

\begin{Lemma} Let $\kappa \in (\N^*)^2$ be given. Suppose
\begin{equation}
\label{CriterionA5551}  
(n_1+n_2) (V - \tau_1^{\kappa_1} \tau_2^{\kappa_1} V)(n) \quad and \quad (n_2-n_1) (V - \tau_1^{\kappa_2} \tau_2^{-\kappa_2} V)(n) = O (1) \textrm{ as } |n|\to\infty.
\end{equation} 
Then $V(\cdot) \in C^1(\pi^{-1} A_{\kappa} \pi)$. Now let $V_s$ satisfy \eqref{Firstcriterion}. Then $V_s(\cdot) \in C^{1,1}(\pi^{-1} A_{\kappa} \pi)$. Let $V_l (n) = o(1)$ as $|n|\to\infty$. Then $V_l(\cdot) \in C^{1,1}(\pi^{-1} A_{\kappa} \pi)$ whenever 
\begin{equation*} 
\label{Secondcriterion223}
\int_{1} ^{\infty} \sup _{r < | n| < 2r} \big | (V_{\text{l}} - \tau_1^{\kappa_1} \tau_2 ^{\kappa_1} V_{\text{l}})(n) \big | dr < \infty \quad \text{and} \quad \int_{1} ^{\infty} \sup _{r < | n| < 2r} \big | (V_{\text{l}} - \tau_1^{\kappa_2} \tau_2 ^{-\kappa_2} V_{\text{l}})(n) \big | dr < \infty.
\end{equation*}
\end{Lemma}

\subsection{Transposing the results for $\Delta$ to $D_B$.} We briefly mention the key formulas. Let $\kappa = (\kappa_1, \kappa_2) \in (\N^*)^2$. One has
\begin{equation}
\label{implication22}
\left [  2D_B , \i \pi^{-1} A_{\kappa} \pi \right ]_{\circ} = \pi^{-1} \left [ \Delta , \i  A_{\kappa} \right ]_{\circ} \pi.
\end{equation}
Thus one has $\boldsymbol{\mu}_{\pi^{-1} A_{\kappa} \pi} ^{\pm} (2D_B) = \boldsymbol{\mu}_{ A_{\kappa} } ^{\pm}(\Delta)$. This formula says that whenever a spectral interval exhibits operator positivity for $\Delta$ with respect to  $A_{\kappa}$ it can be transferred into operator positivity for $2D_B$ with respect to \ $\pi^{-1} A_{\kappa} \pi$ and vice versa.

\subsection{Regularity conditions: a comparison.} As discussed above, the regularity requirements imposed on the potential $V$ by the conjugate operators $A_{\kappa}$, $\pi A_{2\kappa} \pi^{-1}$ and $\pi^{-1} A_{\kappa} \pi$ are of different nature. In this subsection we illustrate some further considerations in that regard. Let $\epsilon >0$. In what follows we further suppose $\kappa = (\kappa_1, \kappa_2)$ with $\kappa_1 = \kappa_2$. Consider the statements:
\begin{align}
\label{form11}
& (V - \tau_1^{\kappa} V)(n) \quad and \quad (V - \tau_2^{\kappa} V)(n) = O(|n|^{-1-\epsilon}), \\
\label{form22}
& (V - \tau_1^{\kappa} \tau_2^{-\kappa} V)(n) \quad and \quad (V - \tau_1^{\kappa} \tau_2 ^{\kappa}  V)(n) = O(|n|^{-1-\epsilon}), \\
\label{form33}
& (V - \tau_1^{2\kappa} V)(n) \quad and \quad (V - \tau_2^{2\kappa} V)(n) = O(|n|^{-1-\epsilon}),
\end{align} 
and
\begin{align}
\label{form111}
& n_1(V - \tau_1^{\kappa} V)(n) \quad and \quad n_2(V - \tau_2^{\kappa} V)(n) = O(|n|^{-\epsilon}), \\
\label{form222}
& (n_1-n_2)(V - \tau_1^{\kappa} \tau_2^{-\kappa} V)(n) \quad and \quad (n_1+n_2)(V - \tau_1^{\kappa} \tau_2 ^{\kappa}  V)(n) = O(|n|^{-\epsilon}), \\
\label{form333}
& n_1(V - \tau_1^{2\kappa} V)(n) \quad and \quad n_2(V - \tau_2^{2\kappa} V)(n) = O(|n|^{-\epsilon}),
\end{align} 
as $|n| \to\infty$. Table \ref{table:456798} gives relationships between these statements.

\begin{table}[H]
    \begin{tabular}{c|l|l|l|l} 
      \footnotesize Assumption & \eqref{form11} holds & \eqref{form22} holds & \eqref{form111} holds & \eqref{form222} holds \\ [0.4em]
      \hline
     \footnotesize  $Consequences$ &   $ \bullet$ \eqref{form22} holds &  $ \bullet$ \eqref{form33} holds & $\bullet V$ satisfies \eqref{A3} & $\bullet V$ satisfies \eqref{A3}    \\ [0.4em]
        &  $ \bullet$ \eqref{form111} holds &  $ \bullet$ \eqref{form222} holds & with $\mathbb{A} = A_{\kappa}$ &with $\mathbb{A} = \pi A_{2\kappa} \pi^{-1}$  \\ [0.4em]
        &  $ \bullet V(\cdot) \in C^{1,1}(A_{\kappa})$ &  $\bullet V(\cdot) \in C^{1,1}(\pi A_{2\kappa} \pi^{-1})$ & $\bullet$ \eqref{form333} holds &  \\ [0.4em]
     \end{tabular}
  \caption{Implications involving \eqref{form11}--\eqref{form333} in dimension 2.}
     \label{table:456798}
\end{table}

\section{An isomorphism for the Standard Laplacian in dimension $3$}
\label{Iso3D}

One may leverage the isomorphism of the previous section to dig out other results for $\Delta$ on $\ell^2(\Z^3)$. We have $\mathcal{V}_B \times \Z \cong \Z^3$, where $\mathcal{V}_B$ is given in Section \ref{Iso2D}. Define a Laplacian $\mathscr{D} $ that corresponds to the Molchanov-Vainberg Laplacian on the plane, and to the Standard Laplacian on the vertical coordinate:
$$\mathscr{D}:= 2 D_B + \Delta_3 \cong \Delta_1+\Delta_2 + \Delta_3, \quad \text{ \ on \ } \quad \ell^2(\mathcal{V}_B \times \Z)  \cong \ell^2(\Z^3).$$
The exact specification of the isomorphism is as follows. Denote 
\[\pi_0: \mathcal{V}_B \times \Z \mapsto \Z^3, \quad \pi_0: (n_1,n_2,n_3) \mapsto \left( (n_1+n_2)/2, (n_2-n_1)/2, n_3 \right)\] 
Note that  $\pi_0^{-1}: \Z^3 \mapsto \mathcal{V}_B \times \Z$, $\pi_0^{-1}: (n_1,n_2,n_3) \mapsto \left( n_1-n_2, n_1+n_2, n_3 \right)$. $\pi_0$ and $\pi_0^{-1}$ induce natural isometries  $\pi: \ell^2(\mathcal{V}_B \times \Z )\mapsto \ell^2(\Z^3)$,
and $\pi^{-1}: \ell^2(\Z^3) \mapsto \ell^2(\mathcal{V}_B \times \Z )$, given by  $\pi^{\pm 1} f(x):= f(\pi_0^{\mp 1} x)$. We now perform a brief spectral analysis of $\mathscr{D}$ with the help of an appropriate conjugate oprerator. Start by noting that $\sigma(\mathscr{D}) = [-3,3]$. Let $\kappa = (\kappa_1, \kappa_2, \kappa_3) \in (\N^*)^3$. Consider conjugate operators of the form
\begin{equation}
\label{A_kappa_rho}
A_{\kappa,\rho}:= A_{\kappa_1,\kappa_2} + \rho A_{\kappa_3}, \quad \mathrm{on} \quad \ell^2(\mathcal{V}_B \times \Z).
\end{equation}
As mentioned in the Introduction, $A_{\kappa,\rho}$ is self-adjoint and essentially self-adjoint on $\ell_0(\mathcal{V}_B \times \Z)$.
We have found that it is useful to include a parameter $\rho \in \R$, for some unknown reason. $A_{\kappa,\rho}$ is invariant on $\ell_0(\mathcal{V}_B \times \Z)$ if and only if both $\kappa_1$ and $\kappa_2$ are even. Thanks to \eqref{MV:k333Ultra} and \eqref{COMMUTATOR_D} one computes
$$[ \mathscr{D}, \i A_{\kappa,\rho} ]_{\circ} = 2 \Delta_2 (1-\Delta_1^2) U_{\kappa_1-1} (\Delta_1) + 2 \Delta_1 (1-\Delta_2^2) U_{\kappa_2-1} (\Delta_2) + \rho (1-\Delta_3^2) U_{\kappa_3-1} (\Delta_3).$$
In fact one can show by induction that $\mathscr{D} \in C^{\infty}(A_{\kappa,\rho})$. To this commutator one associates a polynomial $g_E: [-1,1]^3\mapsto \R$, 
$$g_E(E_1,E_2,E_3):= 2 E_2 (1-E_1^2) U_{\kappa_1-1} (E_1) + 2 E_1 (1-E_2^2) U_{\kappa_2-1} (E_2) + \rho (1-E_3^2) U_{\kappa_3-1} (E_3).$$
The constant energy surface for $\mathscr{D}$ is $S_E:= \{ (E_1, E_2, E_3) \in [-1,1]^3: E = 2E_1 E_2 + E_3 \}$, $E \in \sigma(\mathscr{D})$. One has that $E \in \boldsymbol{\mu}^{\pm}_{A_{\kappa,\rho}} (\mathscr{D})$ if and only if $\pm \left.g_E\right|_{S_E}$ is strictly positive.

Table \ref{table:4567} shows results obtained with the computer together along with conjectures on the exact values. Choosing $\rho =0.5$ gives a band of a.c.\ spectrum that is already covered by the Standard Laplacian in Section \ref{stdLaplacianMourre} (although conditions on $V$ are different). But if we take $\rho =-0.5$ we get another band, adjacent to the other one, which is not covered by the Standard Laplacian, see Table \ref{tab:table1013127conj}. The choice of $\rho$ is based on observation, trial and error.

\begin{table}[H]
    \begin{tabular}{c|c|c|c|c} 
      $\kappa$ & $\rho = 0.5$, num. & $\rho = 0.5$, conjecture & $\rho=-0.5$, num. & $\rho = -0.5$, conjecture \\ [0.4em]
      \hline
      $2$  & $(2,3)$ & $(2+\cos(\pi/2),3)$  & $(1,2)$  &  $(1+2\cos(\pi/2), 2+\cos(\pi/2))$ \\ [0.4em]
      $4$  & $(2.707,3)$ & $(2+\cos(\pi/4),3)$ & $(2.414,2.707)$ & $(1+2\cos(\pi/4), 2+\cos(\pi/4))$   \\ [0.4em]
      $6$  & $(2.866,3)$ & $(2+\cos(\pi/6),3)$ & $(2.733, 2.866)$ & $(1+2\cos(\pi/6), 2+\cos(\pi/6))$   \\ [0.4em]
      $8$  & $(2.924,3)$ & $(2+\cos(\pi/8),3)$ & $(2.848, 2.924)$ & $(1+2\cos(\pi/8), 2+\cos(\pi/8))$   \\ [0.4em]
     \end{tabular}
  \caption{$\boldsymbol{\mu} ^{\text{NUM}}_{A_{\kappa,\rho}}(\mathscr{D}) \cap [0,3]$ in dimension $3$. $\kappa_1 = \kappa_2 = \kappa_3$, denoted $\kappa$ in short}
     \label{table:4567}
\end{table}

\begin{table}[H]
    \begin{tabular}{c|c|c} 
      $\kappa$ & $\rho = 0.5$ & $\rho=-0.5$  \\ [0.4em]
      \hline
       $(2,2,1)$  & $(1,3)$ & $\emptyset$    \\ [0.4em]
      $(2,2,2)$  & $(2,3)$   & $(1,2)$    \\ [0.4em]
      $(2,2,4)$  & $(1.2929,2)\cup (2.707,3)$   & $(1,1.2929) \cup (2,2.707)$    \\ [0.4em]
       $(2,2,8)$  &  $(1.0762,1.2918) \cup (1.6174,2)$ & $(1,1.074)\cup (1.2929,1.6164)$  \\ [0.4em]
        ...  & $\quad \quad ...\cup (2.3826,2.707) \cup (2.924,3)$ & $\quad \quad ...\cup (2,2.3825) \cup (2.707,2.924)$    \\ [0.4em]      
      $(4,4,1)$  & $(2.414,3)$  & $\emptyset$     \\ [0.4em]
      $(4,4,2)$  & $(2.414,3)$  & $\emptyset$    \\ [0.4em]
      $(4,4,4)$  & $(2.707,3)$  & $(2.414,2.707)$    \\ [0.4em]
      $(4,4,6)$  & $(2.414,2.5) \cup (2.866,3)$  & $(2.5,2.866)$    \\ [0.4em]
      $(4,4,8)$  & $(2.414,2.707) \cup (2.924,3)$  & $(2.707,2.924)$    \\ [0.4em]
     \end{tabular}
  \caption{More results for $\boldsymbol{\mu} ^{\text{NUM}}_{A_{\kappa,\rho}}(\mathscr{D}) \cap [0,3]$ for the operator $\mathscr{D}$ in dimension $3$}
     \label{table:4567new}
\end{table}

To obtain the results in Tables \ref{table:4567} and \ref{table:4567new} we plotted the polynomial $g_E$ given above and then used the small algorithm:
\begin{itemize}
\item For all $E \in [0,3]$:
\item For all $E_3 \in [\max(E-2,-1),1]$:
\begin{itemize}
\item let $E_2 = (E-E_3)/(2E_1)$,
\item check if the function $E_1 \mapsto g_E(E_1, E_2 , E_3)$ has same sign on $E_1 \in \pm [|E-E_3|/2,1]$.
\end{itemize}
\end{itemize}

\begin{Lemma}
For $\kappa_1, \kappa_2,\kappa_3$ all even, any $\rho$, $\boldsymbol{\mu}_{A_{\kappa,\rho}} (\mathscr{D}) = - \boldsymbol{\mu}_{A_{\kappa,\rho}} (\mathscr{D})$.
\end{Lemma}
\begin{proof}
The $U_{\kappa_j-1}(\cdot)$ are odd functions. Also, $S_{-E} = \{ (-E_1, E_2,-E_3) , (E_1, -E_2, -E_3): (E_1,E_2,E_3) \in S_E \}$. Thus $g_{-E}(-E_1,E_2,-E_3) = g_{-E}(E_1,-E_2,-E_3) = -g_{E}(E_1, E_2, E_3)$. This implies the statement.
\qed
\end{proof}

\begin{Lemma} 
\label{lem3dSUM} Let $\rho \in \R$. 
$\{ 2\cos( j_1 \pi / \kappa_1) \cos( j_2 \pi / \kappa_2) +\cos( j_3 \pi / \kappa_3): j_i = 0, ...,\kappa_i, i=1,2,3 \} \subset [-3,3] \setminus \boldsymbol{\mu}_{A_{\kappa,\rho}} (\mathscr{D})$ whenever $\kappa_1, \kappa_2,\kappa_3$ are all even. This supports the conjectures in Table \ref{table:4567}.
\end{Lemma}
\begin{proof}
Let $E_i = \cos( j_i \pi / \kappa_i)$, $j_{i} =0,...,\kappa_i$. Then $g_E(E_1,E_2,E_3) = 0$ (see Lemma \ref{lemSUMcos}).
\qed
\end{proof}

\subsection{Transposing the results for $\mathscr{D}$ to $\Delta$} First we need to compute the action of the transformed conjugate operator $\pi A_{\kappa,\rho} \pi^{-1}$. However to ensure it is well defined on $\ell_0(\Z^3)$ we choose to adjust the notation (we face the same obstacle as in Subsection \ref{sub_def_a232}). Instead of using $\kappa = (\kappa_1, \kappa_2,\kappa_3) \in (\N^*)^3$, we use $\tilde{\kappa}:= (2\kappa_1, 2\kappa_2,\kappa_3)$, $\kappa \in (\N^*)^3$. In this way one finds 
\begin{equation}
\label{CONJ_3d}
\pi A_{\tilde{\kappa},\rho} \pi^{-1} = \eqref{PIinversePI} + \rho (4\i)^{-1} \left[ (S_3 ^{\kappa_3} - S_3^{-\kappa_3})N_3 +  N_3 (S_3^{\kappa_3}-S_3^{-\kappa_3}) \right]. 
\end{equation}
It is well defined on $\ell_0(\Z^3)$, $\forall \kappa \in (\N^*)^3$. When $\kappa_1 = \kappa_2$, setting $\kappa$ to $\kappa_1 = \kappa_2$ in \eqref{ConjFourier} gives:
$$\mathcal{F} \pi A_{\tilde{\kappa},\rho} \pi^{-1} \mathcal{F}^{-1} = \eqref{ConjFourier} + \rho(2\i)^{-1} \left[ \sin(\kappa_3 \xi_3) \frac{\partial}{\partial \xi_3} +  \frac{\partial}{\partial \xi_3} \sin(\kappa_3 \xi_3) \right].$$
We mention without proof the obvious result:

\begin{proposition}
\label{iso3D_prop} Fix $d=3$ and $\kappa = (\kappa_j) \in (\N^*)^3$. Then $\Delta \in C^{\infty}(\pi A_{\tilde{\kappa},\rho} \pi^{-1})$ and $\left [  \Delta , \i \pi A_{\tilde{\kappa},\rho} \pi^{-1} \right ]_{\circ} = \pi \left [ \mathscr{D} , \i  A_{\tilde{\kappa},\rho} \right ]_{\circ} \pi^{-1}$. In particular $\boldsymbol{\mu}_{\pi A_{\tilde{\kappa},\rho} \pi^{-1}} (\Delta)=\boldsymbol{\mu}_{A_{\tilde{\kappa},\rho}} (\mathscr{D})$. The latter set is numerically estimated (see Tables \ref{table:4567} and \ref{table:4567new}).
\end{proposition}

Finally we outline the change in regularity requirements for the potential.

\begin{Lemma} 
\label{criteriaEXotic3D}
Fix $\kappa = (\kappa_j) \in (\N^*)^3$. Suppose $V$ satisfies \eqref{CriterionA4441} and $n_3(V -\tau_3 ^{\kappa_3}V)(n) = O (1)$ as $|n|\to\infty$. Then $V(\cdot) \in C^1(\pi A_{\tilde{\kappa},\rho} \pi^{-1})$. If $V_s$ satisfies \eqref{Firstcriterion} then $V_s(\cdot)  \in C^{1,1}(\pi A_{\tilde{\kappa},\rho} \pi^{-1})$. Finally $V_l(\cdot) \in C^{1,1}(\pi A_{\tilde{\kappa},\rho} \pi^{-1})$ whenever $V_l (n) = o(1)$ as $|n|\to\infty$, $V_l(\cdot)$ satisfies \eqref{Secondcriterion22}, and 
\begin{equation*} 
\label{Secondcriterion2290}
\int_{1} ^{\infty} \sup _{r < | n| < 2r} \big | (V_{\text{l}} - \tau_3^{\kappa_3} V_{\text{l}})(n) \big | dr < \infty.
\end{equation*}

\end{Lemma}





\section{The LAP based on Mourre's original paper}
\label{proofMourre}

Notation is fixed. Let $\ell^2(L^{\infty}(\R))$ be the space of real-valued measurable functions $g(t)$ with $\|g\|_{\ell^2(L^{\infty})}:= \{ \sum_{n=0}^{\infty} s_n(g) ^2\}^{1/2} < \infty$, where $s_n(g):= \mathrm{ess  \ sup \ } \{ |g(x)|: n \leq |x| \leq n+1\}$. Let $T$ be a self-adjoint operator in $\mathscr{H}$, $E_\Sigma (T)$ its spectral projection onto a set $\Sigma$. Let 
\[\Sigma_j := \{x \in \R: 2^{j-1} \leqslant |x| \leqslant 2^j \},  \mbox{\, for\,} j \geqslant 1\] 
and $\Sigma_0:= \{ x\in \R: |x| \leqslant 1\}$. Define the Banach spaces with the obvious norms: 
\begin{equation*}
B (T):= \Big \{ \psi \in \mathscr{H} : \| \psi \|_{B(T)}:= \sum_{j=0} ^{\infty} \sqrt{2^j} \| E_{\Sigma_j} (T) \psi \|_{\mathscr{H}}  < \infty \Big \}.
\end{equation*}
The dual of $B(T)$ is the Banach space obtained by completing $\mathscr{H}$ in the norm 
$$ \| \psi \|_{B^*(T)} = \sup_{j \in \N} \sqrt{2^{-j}} \| E_{\Sigma_j} (T) \psi \|_{\mathscr{H}}.$$
We refer to \cite{JP} and the references therein for these definitions. The following result holds for $(\mathfrak{D}, \mathfrak{H}) = (\Delta, \Delta+V)$ or $(D,D+V)$, and any $\mathbb{A}$ as in Table \ref{table:2020987} (for example: $\mathbb{A} = A_{\kappa}$).

\begin{theorem} 
\label{THM_MOURRE}
Suppose that $V \in C^2(\mathbb{A})$ and $V(n)=o(1)$ as $|n| \to +\infty$. Then for any closed interval $I \subset \boldsymbol{\mu}_{\mathbb{A}}(\mathfrak{D}) \setminus \sigma_\mathrm{p}(\mathfrak{H})$, any $f_1, f_2 \in \ell^2(L^{\infty}(\R))$ there is $c>0$ such that 
$$\sup_{z \in I_{\pm}} \| f_1(\mathbb{A}) (\mathfrak{H} - z)^{-1} f_2(\mathbb{A}) \| \leq c \| f_1 \|_{\ell^2(L^{\infty})} \| f_2 \|_{\ell^2(L^{\infty})}.$$
In particular the map $I_{\pm} \ni z \mapsto (\mathfrak{H}-z)^{-1} \in \mathscr{B}(\mathcal{K} ,\mathcal{K}^*)$ extends to a weak-$^*$ continuous map on $I$, with $\mathcal{K} = B(\mathbb{A})$ and $\sigma_{\mathrm{sc}} (\mathfrak{H}) \cap I = \emptyset$. By Lemma \ref{A<N} the statement also holds for $\mathcal{K} = B(\mathbf{N})$.
\end{theorem}
Theorem \ref{THM_MOURRE} is an application of \cite[Theorem I.2]{Mo2}. Since $\mathfrak{H}$ is a bounded operator, the technical assumptions of \cite[Theorem I.2]{Mo2} are trivially fulfilled. Under our hypothesis, the Mourre estimate holds, see Section \ref{full_Mourre}. The statement about the LAP holding in the $B(\mathbb{A}) - B(\mathbb{A})^*$ spaces is proved exactly as in \cite[Proposition 2.1]{JP}.

\section{The LAP in the Besov spaces}
\label{proofBesov}

Notation is fixed. For each real $s$, denote by $\mathscr{H}_s (\mathbb{A})$ the Sobolev space associated to $\mathbb{A}$. For $s\geq 0$ it is the domain of $\langle \mathbb{A} \rangle^s$ and for $s<0$, set $\mathscr{H}_s (\mathbb{A}):= (\mathscr{H}_{-s} (\mathbb{A}))^*$. By identifiying $\mathcal{H}$ with its space of anti-linear forms, we can choose on $\mathscr{H}_s (\mathbb{A})$, the norm $\|f \|_s:= \| \langle \mathbb{A} \rangle^s f \|$, $f \in \mathscr{H}_s (\mathbb{A})$,  for $s\in \R$. For real numbers $t \leq s$, one has the continuous dense embedding $\mathscr{H}_s (\mathbb{A}) \subset \mathscr{H}_t (\mathbb{A})$. By interpolation the Besov spaces $\mathscr{H}_{s,p} (\mathbb{A})$ associated to $\mathbb{A}$ are obtained, namely
\begin{align*}
& \mathscr{H}_{s,p} (\mathbb{A}) = (\mathscr{H}_{s_1} \left(\mathbb{A}), \mathscr{H}_{s_2} (\mathbb{A}) \right) _{\theta,p}, \\
& \mathrm{for \ } s_1 < s_2,\, 0 < \theta < 1, \,s= \theta s_1 + (1-\theta) s_2, \,1 \leq p \leq \infty.
\end{align*}
We refer to \cite{ABG} and \cite{BSa} for more facts on the Besov spaces. The following result holds for $(\mathfrak{D}, \mathfrak{H}) = (\Delta, \Delta+V)$ or $(D,D+V)$, and any $\mathbb{A}$ as in Table \ref{table:2020987} (for example: $\mathbb{A} = A_{\kappa}$). 
\begin{theorem}
\label{BesovTHM}
Suppose that $V \in C^{1,1}(\mathbb{A})$ and $V(n)=o(1)$. Then for any closed interval $I \subset \boldsymbol{\mu}_{\mathbb{A}}(\mathfrak{D}) \setminus \sigma_\mathrm{p}(\mathfrak{H})$, the map $I_{\pm} \ni z \mapsto (\mathfrak{H}-z)^{-1} \in \mathscr{B}\left(\mathcal{K},\mathcal{K}^* \right)$ extends to a weak-$^*$ continuous map on $I$, with $\mathcal{K} = \mathscr{H}_{\frac{1}{2},1} (\mathbb{A})$. In particular $\sigma_{\mathrm{sc}} (\mathfrak{H}) \cap I = \emptyset$. By Lemma \ref{A<N} the statement also holds for $\mathcal{K} = \mathscr{H}_{\frac{1}{2},1} (\mathbf{N})$.
\end{theorem}
Theorem \ref{BesovTHM} is a straightforward application of \cite[Theorem 7.3.1]{ABG}.

\section{The LAP based on energy estimates}
\label{proofEnergy}

Notation is fixed. Let $\log_0 (x):= 1$, $\log_1(x):= \log(1+ x)$, and for integer $k \geqslant 2$, $\log_k(x):= \log \left(1+\log_{k-1}(x)\right)$. Thus $k$ is the number of times the function $\log(1+x)$ is composed with itself. Also denote $\log_k ^q (x):= (\log_k  (x) )^q$, $q \in \R$.  Let $\langle x \rangle:= \sqrt{1+x^2}$. Let
\begin{equation}
\label{weightsW205}
w_{M} ^{\alpha,\beta} (x):= \log_{M+1} ^{\alpha} \left( \langle x \rangle \right) \prod_{k=0}  ^M \log_{k} ^{\beta} \left( \langle x \rangle \right), \quad \alpha, \beta \in \R, M \in \N.
\end{equation}

Let $T$ be a self-adjoint operator. Let $s, p, p' \in \R$, $M \in \N$. Define a family of Banach spaces with the obvious norms:
\begin{equation}
\label{L2sppM}
L^2_{s, p, p', M}(T):= \left \{ \psi \in \mathscr{H}: \|\psi\|_{L^2_{s, p, p', M}(T)}:= \| \langle T \rangle ^{s} w_M ^{p,p'}(T)  \psi \|_{\mathscr{H}} < \infty \right \}.
\end{equation}
The dual with respect to the inner product on $\mathscr{H}$ is $( L^2_{s, p, p', M}(T) )^* = L^2_{-s, -p, -p', M}(T)$. 
Write $L^2_{s}(T):= L^2_{s, 0, 0, 0}(T)$. For any $s, p>1/2$ and $M \in \N$, the following inclusions hold: 
$$L^2_s(T) \subsetneq L^2_{1/2, p,1/2, M}(T) \subsetneq B(T)  \subsetneq L^2_{1/2}(T),$$
and 
$$L^2_{-1/2}(T) \subsetneq B^*(T)  \subsetneq L^2_{-1/2, -p,-1/2, M}(T) \subsetneq L^2_{-s}(T).$$
We need the obvious extension of \cite[Lemma 4.13]{GM2}, see also \cite[Lemma 5.1]{BSa}.
\begin{Lemma} 
\label{A<N}
Fix $d$ and $\kappa =(\kappa_j)_{j=1}^d \in (\N^*)^d$. Then for all conjugate operators $\mathbb{A}$ in Table \ref{table:2020987} there is $c>0$ such that for all $\alpha \in [0,1]$, $(\mathbb{A}^2+1)^{\alpha} \leq c (N_1^2 + ...+ N_d ^2 +1)^{\alpha}$.
\end{Lemma}
Thus, we can apply all the estimates of \cite[Section 4]{GM2} to all such $\mathbb{A}$. Now let $P^{\perp}:= 1-P$, where $P$ is the projection onto the pure point spectral subspace of $\mathfrak{H}$. Let 
\begin{align*}
\mathbb{K}_{\mathbb{A}}(\mathfrak{H}) & := \sigma(\mathfrak{H}) \setminus  \{ E \in \sigma_{\mathrm{p}} (\mathfrak{H}): \{\mathrm{eigenspace \ of \ } \mathfrak{H} \mathrm{ \ associated \ to \ } E\}  \not\subset \mathrm{Dom}[\mathbb{A}] \}.
\end{align*}
Because we include a projector $P^{\perp}$ in the LAP we need an additional local regularity verification, i.e.\ a version of \cite[Lemma 5.2]{GM2}:

\begin{Lemma}
\label{equivLem52}
Suppose $V$ satisfies \eqref{A1}, \eqref{A2}. Let $E \in \Omega$, $\Omega$ as in Theorem \ref{lapy3055}. Then there is a closed interval $I$ of $E$ such that $I \subset \Omega$ and for all $J \subset I$, for all $\theta\in C^{\infty}_c(\R)$ with supp$(\theta) = J$, $PE_{J}(\mathfrak{H}) $ and $P^{\perp}\theta(\mathfrak{H}) \in C^1(\mathbb{A})$.
\end{Lemma}
\begin{proof}
$E \in \boldsymbol{\mu}_{\mathbb{A}}(\mathfrak{D})$ together with the compactness of $V$ implies that a Mourre estimate holds for $\mathfrak{H}$ and $\mathbb{A}$ in a neighborhood $I$ of $E$. In particular $\sigma_{\mathrm{p}}(\mathfrak{H})$ is finite on $I$, see \cite[Corollary 7.2.11]{ABG}. So $PE_{J}(\mathfrak{H})$ is a finite rank operator whenever $J \subset I$. By reducing the size of $I$ it is possible to have $I \subset$ supp$(\eta)$ and $I \subset \mathbb{K}_{\mathbb{A}}(\mathfrak{H})$.  We apply \cite[Proposition 5.1]{GM2} to get $PE_{J}(\mathfrak{H}) \in C^1(\mathbb{A})$. Write $P^{\perp}\theta (\mathfrak{H}) = \theta(\mathfrak{H}) - P E_{J}(\mathfrak{H})\theta(\mathfrak{H})$ to see it too belongs to $C^1(\mathbb{A})$.
\qed
\end{proof}
\begin{remark}
This remark concerns the Standard Laplacian and $\mathbb{A} = A_{\kappa}$ when $\kappa=1$, (which is the case in \cite{GM2} for example). The assumptions \eqref{A1} and \eqref{A2}, coupled with \cite[Theorem 1.5]{Ma2}, imply that the eigenfunctions of $\Delta +V$, if any, belong to $\mathrm{Dom}[A_{\kappa=1}]$. For more general $\kappa = (\kappa_j)$ we do not know of a similar result, but there is the abstract result \cite{FMS} which roughly says that if $V \in C^{1+n}(A_{\kappa})$ then the eigenfunctions of $\Delta+V$ belong to $\mathrm{Dom}[A_{\kappa}^{n}]$, $n \in \N^*$. 
\end{remark}

The following result holds for $(\mathfrak{D}, \mathfrak{H}) = (\Delta, \Delta+V)$ or $(D,D+V)$, and any $\mathbb{A}$ as in Table \ref{table:2020987} (for example: $\mathbb{A} = A_{\kappa}$). 

\begin{theorem}
\label{lapy3055} Let $\varphi_{p,m}$ be the function in \eqref{FunctionVARPHI}. Let $\mathscr{W}_{M} ^{p} (x):= \langle x \rangle ^{\frac{1}{2}} w_{M} ^{p,\frac{1}{2}} (x)$. Suppose 
\begin{equation}
\label{A1}
V \in C^{1}(\mathbb{A}),
\end{equation}
\begin{equation}
there \ are \  m\in \N \ \mathrm{ and } \  2 = r < q \ such \ that \ V(n) = O((w_m ^{q,r}(n))^{-1}), \ and
\label{A2}
\end{equation}
there are $\eta \in C^{\infty}_c(\R)$, $m \in \N$, $p > 1/2$, and bounded and compact operators on $\mathscr{H}$, $B$ and $K$ respectively, whose norms are uniformly bounded with respect to $t$, such that for $t \in \R^+$ large enough,
\begin{equation}
\mathscr{W}_m ^{p} (\mathbb{A}/t) \eta(\mathfrak{D}) [V, \varphi_{p,m}(\mathbb{A}/t)]_{\circ} \eta(\mathfrak{D}) \mathscr{W}_m ^{p} (\mathbb{A}/t) = t^{-2}B + t^{-1}K.
\label{A3}
\end{equation}
Denote $\Omega:= \mathrm{supp} (\eta) \cap \boldsymbol{\mu}_{\mathbb{A}}(\mathfrak{D}) \cap \mathbb{K}_{\mathbb{A}}(\mathfrak{H})$. Let $E \in \Omega$. Then there is a closed interval $I$ of $E$ such that for any integer $M \geqslant m$, and any $p>1/2$, the map $I_{\pm} \ni z \mapsto (\mathfrak{H}-z)^{-1} P^{\perp} \in \mathscr{B}\left( \mathcal{K}, \mathcal{K}^* \right)$ extends to a uniformly bounded map on $I$ with $\mathcal{K} = L^2_{\frac{1}{2},p,\frac{1}{2},M}(\mathbb{A})$. In particular $\sigma_{\mathrm{sc}} (\mathfrak{H}) \cap I = \emptyset$. By Lemma \ref{A<N} the statement also holds for $\mathcal{K} = L^2_{\frac{1}{2},p,\frac{1}{2},M}(\mathbf{N})$.
\end{theorem}

\begin{proof} The proof is that of \cite[Theorem 1.3]{GM2}, the only difference being the formulation of the assumptions. A sketch of proof is outlined for convenience. For $m \in \N$, $p > 1/2$, let 
\begin{equation}
\label{FunctionVARPHI}
\varphi_{p,m}: \R \mapsto \R,  \quad \varphi_{p,m}(t):= \int_{-\infty} ^t \langle x \rangle ^{-1} w_m ^{-2p,-1}(x) dx.
\end{equation} 
$\varphi_{p,m}$ is a bounded function and its derivative yields the weights that appear in the LAP, i.e.\ $\mathscr{W}_m ^{p} (x) = \left(\frac{d}{dx}\varphi_{p,m}(x)\right)^{-1/2}$. One considers the operator $F:= \Pp \theta(\mathfrak{H}) [\mathfrak{H}, \i \varphi_{p,m}(\mathbb{A}/t)]_{\circ} \theta(\mathfrak{H}) \Pp$, where $t \in \R^+$ is a parameter that will be chosen sufficiently large later. One writes 
\begin{equation}
\label{DecompF}
F = \Pp \tH [\mathfrak{D}, \i \varphi_{p,m}(\mathbb{A}/t)]_{\circ} \tH \Pp +  \Pp \tH [V, \i \varphi_{p,m}(\mathbb{A}/t)]_{\circ} \tH \Pp.
\end{equation}
Using assumptions \eqref{A1} and \eqref{A2} one may show that the first term on the rhs.\ of \eqref{DecompF} is $\geq$
\begin{align}
\label{1stLINE}
& \gamma t^{-1} \Pp \tH \left( \mathscr{W}_{M} ^{p} (\mathbb{A}/t) \right) ^{-2}   \tH \Pp  \\
\label{2ndLINE}
&\quad +  \Pp \tH \left( \mathscr{W}_{M} ^{p} (\mathbb{A}/t) \right) ^{-1} \left( t^{-2}B + t^{-1} K \right)  \left( \mathscr{W}_{M} ^{p} (\mathbb{A}/t) \right) ^{-1}  \tH \Pp.
\end{align}
$\gamma >0$ comes from applying the strict Mourre estimate to $\mathfrak{D}$ with respect to $\mathbb{A}$ ; $B$ and $K$ denote respectively bounded and compact operators whose norms do not grow with $t$, and the norm of $K$ goes to zero as the support of $\theta$ shrinks. As for the second term on the rhs.\ of \eqref{DecompF}, again using assumptions \eqref{A1} and \eqref{A2} one shows it is equal to 
\begin{equation}
\label{F1}
\Pp \tH \eta(\mathfrak{D}) [V, \i \varphi_{p,m}(\mathbb{A}/t)]_{\circ} \eta(\mathfrak{D}) \tH \Pp
\end{equation}
plus another term of the form \eqref{2ndLINE}. Then applying \eqref{A3} one shows that \eqref{F1} is also of the form \eqref{2ndLINE}. Taking $t$ large enough implies $F \geq$ \eqref{1stLINE} with $\gamma'$ instead of $\gamma$, $\gamma' \in (0, \gamma)$. It is explained in \cite[Section 2]{GM2} how the weighted estimate $F \geq$ \eqref{1stLINE} implies the LAP. Alternatively one can also argue as in \cite[proof of Theorem 1]{G} starting from the equation (3.30) of that article.
\qed
\end{proof}

\begin{remark} Theorems \ref{lapy305} and \ref{iso2Dthm} are special cases of Theorem \ref{lapy3055} because assumptions \eqref{A20}, \eqref{CriterionA444} and \eqref{CriterionA555} all imply \eqref{A3} when respectively $\mathbb{A} = A_{\kappa}$, $\mathbb{A} = \pi A_{2\kappa} \pi^{-1}$, and $\mathbb{A} = \pi^{-1} A_{\kappa} \pi$. To prove this, one applies the Helffer-Sj\"ostrand formula to express the commutator in \eqref{A3} as an integral, and then a simple analysis proves that this integral converges in norm to a compact operator whose norm does not depend on $t$. In this case the localization $\eta(\mathfrak{D})$ is both harmless and useless. 
\end{remark}
\begin{remark} 
Unlike in the preceding remark, there are instances where the localization $\eta(\mathfrak{D})$ is resourceful. It is the case for some oscillating potentials, such as Wigner von-Neumann potentials, which decay like $O(|n|^{-1})$, see examples \ref{exampleOSC} and \ref{exWvN}. 
\end{remark}

\section{Examples}
\label{Sec_example}

\subsection{Radial potential modulo $\kappa_j$ in dimension $d$.} 
\label{example11}
Fix $\kappa = (\kappa_j) \in (\N^*)^d$. Consider any functions
$$\sigma: \prod_{j=1}^d \{0,...,\kappa_j-1\} \mapsto \{\pm 1\}, \quad \alpha: \prod_{j=1}^d \{0,...,\kappa_j-1\} \mapsto \R^+.$$ Let $V(\kappa_1 n_1 + i_1, ..., \kappa_d n_d +i_d) = \sigma(i_1,...,i_d) \langle n \rangle ^{-\alpha(i_1,...,i_d)}$, for all $n=(n_1,...,n_d) \in \Z^d$ and $i_j \in \{0,...,\kappa_j-1\}$, $j=1,...,d$. This defines a potential on $\ell^2(\Z^d)$. We view $V$ as the product of a periodic component times a damping factor that decays radially. $V(n) = O(|n|^{-\min \alpha(i_1,...,i_d)})$ and $(V-\tau_j^{\kappa_j}V)(n) = O(|n|^{-1 - \min \alpha(i_1,...,i_d)})$, for all $1 \leq j \leq d$. In particular, $V \in C^{1,1}(A_{\kappa})$.

Now let $d=2$. As per the isometry $\pi$ of Section \ref{Iso2D}, $V$ induces a potential $\pi^{-1} V \pi$ defined on $\ell^2(\mathcal{G}_B)$. What does it look like? To illustrate, consider the case $\kappa_1 = \kappa_2 = \kappa = 2$. The graph on the right in Figure \ref{figure:lattice} illustrates the ``periodic pattern'' rotated by $45^{\circ}$. It satisfies $\pi^{-1} V \pi (n_1,n_2) = V (\pi_0 (n_1,n_2)) = V \left(2\floor*{\frac{n_1+n_2}{4}} + f(\frac{n_1+n_2}{2}), 2\floor*{\frac{n_2-n_1}{4}} + f(\frac{n_2-n_1}{2}) \right)$, where $f(x) = x$ modulo $2$. 

\subsection{A class of oscillating potentials in dimension $1$.} 
\label{ExHarmonic} Let $d=1$. Fix $\alpha \in (0,1)$. We pose the ansatz $V(n)-V(n-1) = (-1)^{n}/n^{\alpha}$, $n \geq 1$. Then $(V-\tau^2V)(n) = O(|n|^{-1-\alpha})$. By means of a telescoping sum one finds $V(2n) - V(0) = -H_{2n,\alpha} + 2^{1-\alpha} H_{n,\alpha} = (2^{1-\alpha}-1) \zeta(\alpha) + 2^{-1-\alpha} n^{-\alpha} + O(|n|^{-\alpha-1})$, where we have used the fact that the generalized harmonic numbers $H_{n,\alpha}$ satisfy $H_{n,\alpha} = \zeta(\alpha) + (1-\alpha)^{-1} n^{1-\alpha} + 2^{-1} n^{-\alpha} + O(|n|^{-\alpha-1})$. Set $V(0):=-(2^{1-\alpha}-1) \zeta(\alpha)$ so that $V(n) = o(1)$ at infinity. Thus $V$ is actually purely alternating (at least for $n$ sufficiently large), i.e.\ $V(n) = (-1)^n \tilde{V}(n)$, with $0 \leq \tilde{V}(n) = O(|n|^{-\alpha})$. One is inclined to choose $\mathbb{A} = A_{\kappa=2}$.

\subsection{An oscillating potential in dimension $d$}
\label{exampleOSC} Let  $\kappa := (\kappa_j) \in (\N^*)^d$. Set  $V(n) := \sigma (n) \tilde{V}(n)$, with $\sigma(n) := (-1)^{n_1+...+n_d}$, and suppose there are $m \in \N$ and $2=r<q$ such that $\tilde{V}$ satisfies 
\begin{itemize}
\item $(H_0)$ $\tilde{V}(n) = O((w_m ^{q,r}(n))^{-1})$,  
\item $(H_1)$ $n_i (\tilde{V} - \tau_j \tilde{V})(n) = O((w_m ^{q,r}(n))^{-1})$ for $1 \leq i,j \leq d$, 
\item $(H_2)$ the $\kappa_j$'s are all even, or
\item $(H_2')$ $n_j \tilde{V}(n) = O(1)$, for $j=1,...,d$.
\end{itemize}
On $\ell_0(\Z^d)$ we see that $[A_{\kappa}, V(\cdot) ]=$
\[\frac{1}{2\i}  \sum_{j=1}^d 2^{-1} \kappa_j [(S_j^{\kappa_j} - S_j^{-\kappa_j}), V(\cdot)] +  [(S_j^{\kappa_j} - S_j^{-\kappa_j}), \sigma ] \tilde{V}(\cdot) N_j + \sigma [(S_j^{\kappa_j} - S_j^{-\kappa_j}), \tilde{V}(\cdot) ]  N_j.\] 
Assuming $(H_0), (H_1)$ and either $(H_2)$ or $(H_2')$, since $\ell_0(\Z^d)$ is a core for $A_{\kappa}$ and since the commutator extends to a element of $\mathscr{B}(\ell^2(\Z^d))$, we infer that $V \in C^1(A_{\kappa})$. For specific $1$-dimensional examples, let us mention for $(H_2)$, $V(n) = (-1)^n\log^{-p}(2+|n|)$, $p>2$, with $A_{\kappa=2}$, and for $(H_2')$, $V(n)= n^{-1}(-1)^n(1+2/ \log(n))$ with $A_{\kappa=1}$. The latter is Remling's example from \cite{R}. 

The point about this class of oscillating potentials is that it is relevant to use localizations to verify the hypothesis \eqref{A3}. First we treat the Standard Laplacian. Thanks to the relation $\sigma \Delta = - \Delta \sigma$, in all dimensions $d$, one has $\eta(\Delta) \sigma \eta (\Delta) = \eta(\Delta) \eta(-\Delta) \sigma$ for any $\eta \in C^{\infty}_c(\R)$. Furthermore by functional calculus $\eta(\Delta) \eta(-\Delta) = 0$ whenever supp$(\eta) \Subset \pm (0,d]$. Thus for $\eta$ localized away from $E=0$,
\begin{align*}
\eta(\Delta) [ V , A_{\kappa} ]_{\circ} \eta(\Delta) &= \eta(\Delta) \sigma [\tilde{V}A_{\kappa},  \eta(\Delta) ]_{\circ} - [ \eta(\Delta) , A_{\kappa} \tilde{V} ]_{\circ} \sigma \eta(\Delta).
\end{align*}
One may check that the assumptions $(H_0)$ and $(H_1)$ mean that \eqref{A3} holds for any $p \in (1/2,q/2)$. Now we briefly discuss the Molchanov-Vainberg Laplacian. One may check that for even dimensions $d$, one has $D \sigma = \sigma D$, which in turn implies $\eta(D) \sigma \eta (D) = \eta ^2 (D) \sigma$, which is non-zero regardless of the localization. So we cannot expect to use a localization argument. For odd dimensions $d$ however, one has $\eta(D) \sigma \eta (D) = \eta (D) \eta(-D) \sigma$ which is equal to $0$ whevener supp$(\eta) \Subset \pm (0,1]$. Thus we may use a localization argument exactly as described above to obey \eqref{A3}.

\subsection{Oscillating potential in dimension $2$.} This example is relevant for Section \ref{Iso2D}. Let $V_B(n_1,n_2) = (\i)^{n_1+n_2} \langle n \rangle ^{-\epsilon}$, $\epsilon >0$, defined on $\ell^2(\mathcal{G}_B)$ (recall $n_1+n_2$ is even). It satisfies \eqref{form22} with $\kappa =2$, but not \eqref{form11} with $\kappa =2$.

\subsection{Oscillating potential in dimension $2$.} This example is relevant for Section \ref{Iso2D}. For $\epsilon >0$, let 
$$V_B(n_1,n_2) = (-1)^{n_1} \log(n_1) \langle n \rangle ^{-\epsilon} + (-1)^{n_2} \log(n_2) \langle n \rangle ^{-\epsilon},$$ 
defined on $\ell^2(\mathcal{G}_B)$. $V_B$ satisfies \eqref{form111} with $\kappa = 2$. By the unitary transformation, 
$$\pi V_B \pi^{-1} (n_1,n_2) = (-1)^{n_1-n_2} \log (n_1 - n_2) \langle n \rangle ^{-\epsilon} + (-1)^{n_1+n_2} \log (n_1 + n_2) \langle n \rangle ^{-\epsilon}$$
defined on $\ell^2(\Z^2)$ satisfies \eqref{form222} with $\kappa =1$.

\subsection{Wigner-von Neumann potential}
\label{exWvN}
We consider the Wigner-von Neumann potential 
\begin{equation}
\label{WvNdef}
W(n):= |n|^{-1} \cdot \sin(k(n_1+...+n_d)), \quad k \in (0,\pi).
\end{equation} 
We stick to $k \in (0,\pi)$ for simplicity. The LAP for $\Delta +W(\cdot)$ was treated in \cite{Ma1}, \cite{GM2}. $W$ satsifies \eqref{A1} for any $\mathbb{A}$ as in Table \ref{table:2020987}, and \eqref{A2}. The point of this potential is that localization in energy $\eta(\Delta)$ is necessary to satisfy criterion \eqref{A3}. 
It is proved in \cite{Ma1} that
\begin{equation}
\begin{aligned}
\boldsymbol{\tilde{\mu}}_{A_{\kappa}=1}(\Delta+W(\cdot)) &= (-1,1)\setminus \{E_{\pm}(k)\} \ \ \text{for} \ \ d=1, \\
\boldsymbol{\tilde{\mu}}_{A_{\kappa}=1}(\Delta+W(\cdot)) & \supset (-d,-d+E(k)) \cup (d-E(k),d) \ \ \text{for} \ \ d\geqslant 2. 
\label{MuH}
\end{aligned}
\end{equation}
with equality in the case $d=2$. Here $E_{\pm}(k):= \pm \cos\left(k/2\right)$ and $E(k) = 2-2|\cos(k/2)|$.
We refer to \cite[Lemma 3.4 and Proposition 4.4]{Ma1} for this result. Given a triplet $(V, \mathfrak{D}, \mathbb{A})$ and $\mathfrak{H} = \mathfrak{D}+V(\cdot)$, define 
$$\aleph(V, \mathfrak{D}, \mathbb{A}):= \{ E \in \sigma(\mathfrak{H}): \exists \eta \in C_c^{\infty}(\R) \ supported  \ on \  I, I \ni E,  \ such \  that \ \eqref{A3}  \ holds \}.$$
Based on the work \cite{Ma1} it follows quite forwardly that:
\begin{theorem} 
For any relevant $\mathbb{A}$ as in Table \ref{table:2020987}, $\aleph(W, \Delta, \mathbb{A}) \supset \boldsymbol{\tilde{\mu}}_{A_{\kappa}=1}(\Delta+W(\cdot))$, given by \eqref{MuH}. If $H_{\mathrm{std}} = \Delta + W(\cdot) + V(\cdot)$, and $V$ satisfies \eqref{A1}, \eqref{A2} and \eqref{A3} for some $\mathbb{A}$ as in Table~\ref{table:2020987}, then for any $E \in \aleph(W, \Delta, \mathbb{A}) \cap \aleph(V, \Delta, \mathbb{A}) \cap \boldsymbol{\mu}_{\mathbb{A}}(\Delta) \cap \mathbb{K}_{\mathbb{A}}(H_{\mathrm{std}})$, the conclusion of Theorem \ref{lapy3055} holds.
\end{theorem}

Now we develop the analogous result for $H_{\mathrm{MV}} = D + W(\cdot)+V(\cdot)$. To this end we parallel the calculation of \cite[Section 3]{Ma1}. The underlying idea is traced back to \cite[Lemma 2.5]{FH}. Let $T_k$ be the operator of multiplication on $\ell^2(\Z^d)$ given by $(T_k u)(n):= e^{\i k(n_1+...+n_d)}u(n)$. Then $(\mathcal{F}T_{k}\mathcal{F}^{-1}f)(\xi)=f(\xi+k)$. Denote by $\widecheck{\1}_{[0,\pi],i}$ the operator on $\ell^2(\Z^d)$ satisfying $(\mathcal{F}\widecheck{\1}_{[0,\pi],i}\mathcal{F}^{-1}f)(\xi) = \1_{[0,\pi]}(\xi_i) f(\xi)$. $\widecheck{\1}_{[0,\pi],i}$ is a bounded self-adjoint operator with spectrum $\sigma(\widecheck{\1}_{[0,\pi],i}) = \{0,1\}$. Let 
\begin{equation}
g_k: [-1,1] \times \{0,1\} \mapsto \R,  \quad g_k(x,y):= x \cos(k) - \sin(k)\sqrt{1-x^2}(2y-1).
\label{okid}
\end{equation}
Then one easily proves the following key relation: 
\begin{equation*}
T_k D = \left(\prod_{i=1}^d g_k(\Delta_i, \widecheck{\1}_{[0,\pi],i}) \right)T_k.
\end{equation*}
By the Helffer-Sj\"ostrand formula this implies:
\begin{equation}
\label{vd=bv54}
T_k \theta(D) = \theta \left(\prod_{i=1}^d g_k(\Delta_i, \widecheck{\1}_{[0,\pi],i}) \right)T_k.
\end{equation}
Since $\{\Delta_i,\widecheck{\1}_{[0,\pi],i}\}_{i=1}^d$ forms a family of self-adjoint commuting operators, we may apply the functional calculus for such operators. Let 
\begin{align*}
\widetilde{E}(k) &:= \max \{ \cos^2(k/2), \sin^2(k/2) \} \\
&= \cos^2(k/2) \ for \ k\in (0,\pi/2) \cup (3\pi/2,2\pi) \quad and \quad \sin^2(k/2) \ for  \ k\in (\pi/2, 3\pi/2).
\end{align*}

\begin{Lemma} 
Let $d=2$. For every $E \in [-1,1] \setminus [-\widetilde{E}(k), \widetilde{E}(k)]$ there is $\epsilon >0$ such that for any $\theta \in C^{\infty}_c(\R)$ supported on $I:= (E-\epsilon,E+\epsilon)$, 
$$\theta(D) \theta \left( \prod_{i=1}^{d=2} g_k(\Delta_i, \widecheck{\1}_{[0,\pi],i}) \right) = \theta(D) \theta \left( \prod_{i=1}^{d=2}  g_{2\pi -k}(\Delta_i, \widecheck{\1}_{[0,\pi],i}) \right) = 0.$$
In particular, writing $W = (2\i)^{-1}(T_k - T_{-k})|N|^{-1}$, it follows by \eqref{vd=bv54} that for any relevant $\mathbb{A}$ 
$$\aleph(W,D, \mathbb{A}) = \pm (\widetilde{E}(k),1), \quad for \ d=2.$$
\end{Lemma}
\begin{remark}
The choice of $\mathbb{A}$ is not important for $W$ because it is the localization in energy that is doing the job.
\end{remark}

\begin{proof}
We assume $k\in (0,\pi)$. The case $k \in (\pi,2\pi)$ is similar. Let us explain the strategy in dimension $d$. We want to find $\epsilon(E) > 0$ such that for the interval $\mathcal{I}:= (E-\epsilon,E+\epsilon)$ we have
\begin{equation}
\mathcal{I} \cap \big \{ \prod_{1 \leqslant i \leqslant d} g_k(x_i,y_i): (x_1,...,x_d) \in S_{\mathcal{I}} \ \ \text{and} \ \  (y_1,...,y_d) \in \{0,1\}^d \big \} = \emptyset,
\label{oksy}
\end{equation} 
where $S_{\mathcal{I}}$ is the region defined by $S_{\mathcal{I}}:= \{(x_1,...,x_d) \in [-1,1]^d: \prod_{i=1}^d x_i \in \mathcal{I} \}$. In this way if supp$(\theta)=\mathcal{I}$, then we will have $\theta(\prod_i x_i)\theta(\prod_i g_k(x_i,y_i)) = 0$ as required. Set 
\begin{equation*}
\mathcal{E}_d (k):= \{ E \in [-1,1]: \exists \  (x_i)_{i=1}^d \in [-1,1]^d \ \text{and} \ (y_i)_{i=1}^d \in \{0,1\}^d \text{ s.t. }  E = \prod_{i=1}^d x_i = \prod_{i=1} ^d g_k(x_i,y_i) \}.   
\end{equation*}
If $E \in \mathcal{E}_d (k)$, then \eqref{oksy} does not hold at $E$. By a continuity argument, the converse is true as well, provided $\epsilon$ is sufficiently small. Note also that $\mathcal{E}_d (k) = \mathcal{E}_d (2\pi-k)$. We only identify the set $\mathcal{E}_2(k)$, as the problem becomes too complex for $d\geq 3$. We solve 
\begin{equation}
\label{espin1}
x_1 x_2 = g_k(x_1,y_1) g_k(x_2,y_2).
\end{equation} 
One makes the change of variable $\cos(\phi_i) = x_i$, $\phi_i \in [0,\pi]$. So it's the same as solving $\cos(\phi_1) \cos(\phi_2) = \cos(\phi_1 \pm k) \cos(\phi_2 \pm k)$. 
Case $y_1 = y_2 = 1$. Thanks to the product to sum cosine formula one shows that \eqref{espin1} is equivalent to $\sin(\phi_1+\phi_2 + k) = 0$, which has solutions $\phi_1+\phi_2+k = \pi$ or $2\pi$, as we assume $k \in (0,\pi)$. Let $f(\phi_1,\phi_2):= \cos(\phi_1) \cos(\phi_2)$. We have 
$$ \{ f(\phi_1, \pi - k - \phi_1 ) \}_{\phi_1 \in [0,\pi] } = \bigg [f \left(\pi-\frac{k}{2}, \frac{k}{2} \right), f \left(\frac{\pi-k}{2}, \frac{\pi-k}{2} \right) \bigg ] = \bigg [-\cos^2 \left(\frac{k}{2}\right),\sin^2 \left(\frac{k}{2}\right) \bigg].$$

$$ \{ f(\phi_1, 2\pi - k - \phi_1 ) \}_{\phi_1 \in [0,\pi] } = \{ - f(\phi_1, \pi - k - \phi_1 ) \}_{\phi_1 \in [0,\pi] }  = \bigg [-\sin^2 \left(\frac{k}{2}\right),\cos^2 \left(\frac{k}{2}\right) \bigg].$$

Case $y_1=y_2=0$. \eqref{espin1} is equivalent to $\sin(\phi_1+\phi_2 - k) = 0$. Thus, since we assume 
$k \in (0,\pi)$, we have $\phi_1+\phi_2-k=0$ or $\pi$.

$$ \{ f(\phi_1, k - \phi_1 ) \}_{\phi_1 \in [0,\pi] } = \bigg [f \left(\frac{\pi+k}{2}, \frac{\pi-k}{2} \right), f \left(\frac{k}{2}, \frac{k}{2} \right) \bigg ] = \bigg [-\sin^2 \left(\frac{k}{2}\right),\cos^2 \left(\frac{k}{2}\right) \bigg].$$

$$ \{ f(\phi_1, \pi + k - \phi_1 ) \}_{\phi_1 \in [0,\pi] } =  \{ -  f(\phi_1, k - \phi_1 ) \}_{\phi_1 \in [0,\pi] } = \bigg [-\cos^2 \left(\frac{k}{2}\right),\sin^2 \left(\frac{k}{2}\right) \bigg].$$

Case $y_1 = 1$, $y_2=0$.  \eqref{espin1} is equivalent to $\sin(\phi_1-\phi_2 + k) = 0$. Thus, since we assume $k \in (0,\pi)$, we have $\phi_1-\phi_2+k=0$ or $\pi$.

$$\{ f(\phi_1, k + \phi_1 ) \}_{\phi_1 \in [0,\pi] }  =  \{ f(\phi_1, 2\pi - k - \phi_1 ) \}_{\phi_1 \in [0,\pi] }.$$
$$\{ f(\phi_1, -\pi + k + \phi_1 ) \}_{\phi_1 \in [0,\pi] }  =  \{ f(\phi_1, \pi - k - \phi_1 ) \}_{\phi_1 \in [0,\pi] }.$$

Case $y_1 = 0$, $y_2=1$. This is the same as the previous case, with $\phi_1$ interchanged with $\phi_2$. So the solutions are the same as in the previous case by symmetry.

The statement of the Lemma follows by taking $\max_{k \in [0,\pi]} \{\sin^2(k/2), \cos^2(k/2) \}$.
\qed
\end{proof}

We now have our LAP for $H_{\mathrm{MV}} := D+ W +V$.
\begin{theorem} 
Let $d=2$. If $H_{\mathrm{MV}} = D + W(\cdot) + V(\cdot)$, and $V$ satisfies \eqref{A1}, \eqref{A2} and \eqref{A3} for some $\mathbb{A}$ as in Table \ref{table:2020987}, then for any $E \in \aleph(W, D, \mathbb{A})\cap \aleph(V, D, \mathbb{A}) \cap \boldsymbol{\mu}_{\mathbb{A}}(D) \cap \mathbb{K}_{\mathbb{A}}(H_{\mathrm{MV}})$, the conclusion of Theorem \ref{lapy3055} holds.
\end{theorem}

Finally let us make an observation. We continue with $d=2$. On the one hand, one has $\pi W \pi ^{-1}(n_1,n_2) = (\sqrt{2} |n|)^{-1} \sin(2kn_1)$, where $W$ is given by \eqref{WvNdef}, and $\pi$ as in Section \ref{Iso2D}. Thanks to the isomorphism we infer $\aleph(\pi W\pi^{-1}, \Delta, \mathbb{A}) = \pm (2\widetilde{E}(k),2)$, for any relevant $\mathbb{A}$ as in Table \ref{table:2020987}. On the other hand, one can confirm this directly working only with $\Delta$ as follows. Let $\tilde{T}_k$ be the operator of multiplication by $e^{\i k n_1}$ on $\ell^2(\Z^2)$. One has $\tilde{T}_k \Delta = \left(g_k(\Delta_1, \widecheck{\1}_{[0,\pi],1})  + \Delta_2 \right)\tilde{T}_k$. Thus $\tilde{T}_k \theta(\Delta) = \theta \left( g_k(\Delta_1, \widecheck{\1}_{[0,\pi],1}) + \Delta_2 \right) \tilde{T}_k$. To solve the equation $\theta(\Delta) \theta \left( g_k(\Delta_1, \widecheck{\1}_{[0,\pi],1}) + \Delta_2 \right) =0$ one is led to solve $x_1 + x_2 = g_k(x_1, y) + x_2$, which has solutions $E = x_1 + x_2 = \pm \cos(k/2) +[-1,1]$. By this approach we conclude that $\aleph(\pi W\pi^{-1}, \Delta, \mathbb{A}) = \pm (e(k),2)$, where $e(k):= \max \{1+\cos(k), 1-\cos(k) \}$ (note we have to plug in $2k$ instead of $k$). The conclusions are in agreement, i.e.\ $e(k) =  2\widetilde{E}(k)$ because of the identities $1+\cos(k) = 2\cos^2(k/2)$ and $1-\cos(k) = 2\sin^2(k/2)$.

\section{Appendix: Convergence of the 2 Laplacians}
\label{convergence_appendix}

Consider the Hilbert space $\mathscr{H}_h := \ell^2(h \Z^d)$, where $h \Z^d$ is the square lattice, $d \geqslant 1$ is the dimension, and $h >0$ is a scaling parameter determining the mesh size. To establish convergence between the 2 Laplacians we shift and rescale them. Instead of \eqref{def:std} and \eqref{def:MV22} we use :

\begin{align}
\label{eqn2.1}
\Delta_h & := h^{-2} \sum _{j=1}^d (2 - S_j - S_j ^*) = 2 h^{-2} \sum_{j=1} ^d (1-\Delta_j), \quad \Delta_j := (S_j + S_j^*)/2, \\
D_h &:= 2h^{-2} - 2^{-d+1} h^{-2} \prod_{j=1}^d (S_j + S_j ^*) = 2h^{-2} \left(1- \prod_{j=1}^d \Delta_j \right).
\end{align}
The spectra of $\Delta_h$ and $D_h$ are $h^{-2}[0,4d]$ and $h^{-2}[0,4]$ respectively. The next result extends those proved in \cite{NT}.
 \begin{proposition} 
 \label{thm1prp} For any fixed $\mu \in \C \setminus \R$ we have
 \begin{equation*}
\| (\Delta_h - \mu ) ^{-1} - (D_h - \mu)^{-1} \| _{\mathscr{B}(\mathscr{H}_h)} = O(h^2), \quad \text{as} \ h \to 0.
\end{equation*}
If $V$ is a real-valued bounded potential on $h\Z^d$, then 
\begin{equation}
 \label{main_result}
\| (\Delta_h +V - \mu ) ^{-1} - (D_h+V - \mu)^{-1} \| _{\mathscr{B}(\mathscr{H}_h)} = O(h^2), \quad \text{as} \ h \to 0.
\end{equation}
\end{proposition}
The first convergence formula is proved by expressing the Laplacians as operators of multiplication by functions in Fourier space and performing a Taylor expansion. Now let $H_h := \Delta_h + V$, $\tilde{H}_h := D_h +V$. The second formula follows directly from the first thanks to the identity :
$$(\tilde{H}_h - \mu)^{-1} - (H_h - \mu)^{-1}  = \left[ 1 -(\tilde{H}_h - \mu)^{-1} V \right ] \left [ (D_h - \mu)^{-1} - (\Delta_h - \mu)^{-1} \right] (\Delta_h - \mu) (H_h - \mu)^{-1}. $$

\section{Appendix: Algorithm details}
\label{appendix_algo}

When $\mathfrak{D} = \Delta$ in dimension $2$: For the results in Tables \ref{tab:table1011} and  \ref{tab:table10112162} we used the simple algorithm:
\begin{itemize}
\item For all $E \in [-2,2]$: 
\begin{itemize}
\item let $E_2 = E-E_1$
\item check if the function $E_1 \mapsto g_E(E_1, E_2)$ has same sign on the interval \\ \mbox{$E_1 \in [\max(E-1,-1),\min(E+1,1)]$}.
\end{itemize}
\end{itemize}

When $\mathfrak{D} = \Delta$ in dimension $3$: For the results in Tables \ref{tab:table1013127conj} and \ref{tab:table1013127NEW} we used the simple algorithm:
\begin{itemize}
\item For all $E \in [0,3]$: 
\item For all $E_3 \in [\max(E-2,-1),\min(E+2,1)]$:
\begin{itemize}
\item let $E_2 = E-E_1-E_3$
\item check if the function $E_1 \mapsto g_E(E_1, E_2 , E_3)$ has same sign on the interval \\ $E_1 \in [\max(E-E_3-1,-1),\min(E-E_3+1,1)]$
\end{itemize}
\end{itemize}

When $\mathfrak{D} = D$ in dimension $2$: For the results in Tables \ref{tab:table1012} and \ref{tab:table101244} we used the simple algorithm:
\begin{itemize}
\item For all $E \in [-1,1]$: 
\begin{itemize}
\item let $E_2 = E/E_1$
\item check if the function $E_1 \mapsto g_E(E_1, E_2)$ has same sign on the interval $E_1 \in \left[ -1, -|E|] \cup [|E|,1 \right]$.
\end{itemize}
\end{itemize}

When $\mathfrak{D} = D$ in dimension $3$: For the results in Tables \ref{tab:table10144} and \ref{tab:table10144NEW} we used the simple algorithm:
\begin{itemize}
\item For all $E \in [-1,1]$: 
\item For all $E_3 \in [-1,-|E|] \cup [|E|,1]$:
\begin{itemize}
\item let $E_2 = E/(E_1 E_3)$
\item check if the function $E_1 \mapsto g_E(E_1, E_2 , E_3)$ has same sign on the interval $E_1 \in \left[ -1, -|E/E_3|] \cup [ |E/E_3|,1 \right]$.
\end{itemize}
\end{itemize}

\newpage
\clearpage

\section{Appendix: Numerical evidence for Standard Laplacian} 
\label{appendix_std}
In the tables below, the decimal numbers were obtained using computer software. We denote those results by $\boldsymbol{\mu} ^{\text{NUM}}_{\mathbb{A}}(\mathfrak{D})$. The \textbf{horizontal single line} towards the top of the table separates the values of $\kappa$ for which we have proved the results rigorously from the values for which we don't have a rigorous proof. Below the \textbf{horizontal double line} towards the bottom of the table we have included conjectures on closed form formulas for the results, which we denote $\boldsymbol{\mu} ^{\text{CONJ}}_{\mathbb{A}}(\mathfrak{D})$.


\begin{table}[H]
  \begin{center}
    \begin{tabular}{c|c} 
      $\kappa$ & energies in $[0,2]$ for which a Mourre estimate holds for $\Delta$ wrt.\ $A_{\kappa}$, $d=2$.     \\ [0.5em]
      \hline
      $1$  & $(0,2)$ \\ [0.5em]
      $2$  & $(1,2)$  \\ [0.5em]
      $3$  &  $(0.542477,1.000000) \cup (1.500000,2)$   \\[0.5em]
      $4$ &  $( 1.026054 , 1.414214 ) \cup ( 1.707107 , 2 ) $  \\[0.5em]
      \hline
      $5$ &  $( 1.326098 , 1.618034 ) \cup ( 1.809017 , 2 )$  \\[0.5em]
      $6$ &  $( 1.511990 , 1.732051 ) \cup ( 1.866026 , 2 )$  \\[0.5em]
      $7$ &  $( 1.222521 , 1.246980 ) \cup ( 1.632351 , 1.801938 ) \cup ( 1.900969 , 2 )$  \\[0.5em]
      $8$ &  $( 1.382684 , 1.414214 ) \cup ( 1.713916 , 1.847760) \cup ( 1.923880 , 2 )$  \\[0.5em]
      $9$ &  $( 1.500001 , 1.532089 ) \cup ( 1.771437 , 1.879386 ) \cup ( 1.939693 , 2 )$  \\[0.5em]
      $10$ &  $( 1.587786 , 1.618034 ) \cup ( 1.813393 , 1.902114 ) \cup ( 1.951057 , 2 )$  \\[0.5em]
      $11$ &  $( 1.654861 , 1.682508 ) \cup ( 1.844874 , 1.918986 ) \cup ( 1.959493 , 2 )$  \\[0.5em]
      $12$ &  $( 1.707107 , 1.732051 ) \cup ( 1.869071 , 1.931852 ) \cup ( 1.965926 , 2 )$  \\[0.5em]
      $13$ &  $( 1.748511 , 1.770913 ) \cup ( 1.888053 , 1.941884 ) \cup ( 1.970942 , 2 )$  \\[0.5em]
      $14$ &  $( 1.781832 , 1.801938 ) \cup ( 1.903209 , 1.949856 ) \cup ( 1.974928 , 2 )$  \\[0.5em]
      $15$ &  $( 1.809017 , 1.827091 ) \cup ( 1.915498 , 1.956296 ) \cup ( 1.978148 , 2 )$  \\[0.5em]
      $16$ &  $( 1.831470 , 1.847760 ) \cup ( 1.925596 , 1.961571 ) \cup ( 1.980786 , 2 )$  \\[0.5em]
      $17$ &  $( 1.850218 , 1.864945 ) \cup ( 1.933994 , 1.965947 ) \cup ( 1.982974 , 2 )$  \\[0.5em]
      $18$ &  $( 1.866026 , 1.879386 ) \cup ( 1.941050 , 1.969616 ) \cup ( 1.984808 , 2 )$  \\[0.5em]
      $19$ & $( 1.879474 , 1.891635 ) \cup ( 1.947036 , 1.972723 ) \cup ( 1.986362 , 2 )$ \\[0.5em]
      $20$ &  $( 1.891007 , 1.902113 ) \cup ( 1.952156 , 1.975377 ) \cup ( 1.987688 , 2 )$  \\[0.5em]
      \hline \hline
      $3-6$ & $\left( \ \_ \_ \ ,2\cos(\pi/ \kappa) \right) \cup \left(1+\cos(\pi/ \kappa) ,2 \right)$   \\[0.5em]
    $7-20$ &  $\left( 1+\cos(3\pi/ \kappa) , 2\cos(2\pi/ \kappa) \right) \cup \left( \ \_ \_ \ , 2\cos(\pi/ \kappa) \right) \cup \left( 1+\cos(\pi/ \kappa) , 2 \right)$  
    \end{tabular}
  \end{center}
     \caption{$\boldsymbol{\mu} ^{\text{NUM}}_{A_{\kappa}}(\Delta) \cap [0,2]$ above the double horizontal line and $\boldsymbol{\mu} ^{\text{CONJ}}_{A_{\kappa}}(\Delta) \cap [0,2]$ below the horizontal double line. Dimension 2, $\kappa_1 = \kappa_2$}
           \label{tab:table1011}
\end{table}

\begin{table}[H]
  \begin{center}
    \begin{tabular}{c|c} 
      $\kappa = (\kappa_1,\kappa_2)$ & energies in $[-2,2]$ for which a Mourre estimate holds for $\Delta$ wrt.\ $A_{\kappa}$, $d=2$.    \\ [0.5em]
      \hline
      $(1,2)$  & $(-0.970992,-0.166382) \cup (1.000000,2)$ \\ [0.5em]
      $(1,3)$  & $\pm (1.500000,2)$  \\[0.5em]
      $(1,4)$  &  $(-1.610194,-1.000000) \cup (0.292894,0.612721) \cup (1.707107,2)$   \\[0.5em]
      $(1,5)$ &  $\pm (0.690983, 0.979613) \cup \pm( 1.809017 , 2 ) $  \\[0.5em]
      $(1,6)$ &  $( -1.776066,-1.500000 ) \cup ( 1.000000,1.210566 ) \cup (1.866026,2)$  \\[0.5em]
      $(2,3)$ &  $(-1.500000,-1.136027) \cup (1.500000,2)$  \\[0.5em]
            $(2,4)$ &  $\pm (1.707108,2)$  \\[0.5em]
            $(2,5)$ &  $(-1.796088, -1.309018) \cup (1.809017,2)$  \\[0.5em]
            $(2,6)$ &  $\pm (1.144385, 1.376569) \cup \pm (1.866025,2)$  \\[0.5em]
 $(3,4)$ &  $(-1.707107,-1.566774) \cup (1.015055,1.192464) \cup (1.707107,2)$  \\[0.5em]
  $(3,5)$ &  $\pm (1.809017,2)$  \\[0.5em]
   $(3,6)$ &  $(-1.865620,-1.558422) \cup (1.866026,2)$  \\[0.5em]
        $(4,5)$ &  $(-1.809017,-1.746259) \cup (1.320973,1.510642) \cup (1.809017,2)$  \\[0.5em]
    $(4,6)$ &  $\pm(1.504482,1.561559) \cup \pm (1.866025,2)$  \\[0.5em]
      $(5,6)$ &  $(-1.866025,-1.834695) \cup (1.509219, 1.672531) \cup (1.866025,2)$  \\[0.5em]
    \end{tabular}
  \end{center}
   \caption{$\boldsymbol{\mu} ^{\text{NUM}}_{A_{\kappa}}(\Delta) \cap [-2,2]$ in dimension 2, $\kappa_1 \neq \kappa_2$. }
       \label{tab:table10112162}
\end{table}


\begin{table}[H]
  \begin{center}
    \begin{tabular}{c|c} 
      $\kappa$ & energies in $[0,3]$ for which a Mourre estimate holds for $\Delta$ wrt.\ $A_{\kappa}$, $d=3$.   \\ [0.5em]
      \hline
      $1$  & $[0,1) \cup (1,3)$ \\ [0.5em]
      $2$  & $(2,3)$  \\ [0.5em]
      $3$  &  $(2.5,3)$   \\[0.5em]
      \hline
      $4$ &  $(2.072,2.121) \cup (2.707,3)$  \\[0.5em]
      $5$ &  $(2.353,2.427) \cup (2.809,3)$  \\[0.5em]
      $6$ &  $(2.529,2.598) \cup (2.866,3)$  \\[0.5em]
      $7$ &  $(2.645,2.703) \cup (2.901,3)$  \\[0.5em]
      $8$ &  $(2.724,2.771) \cup (2.924,3)$  \\[0.5em]
      $9$ &  $(2.780,2.819) \cup (2.940,3)$  \\[0.5em]
      $10$ &  $(2.820,2.853) \cup (2.951,3)$  \\[0.5em]
      $12$ &  $(2.874,2.897) \cup (2.966,3)$  \\[0.5em]
      $16$ &  $(2.928,2.942) \cup (2.981,3)$  \\[0.5em]
      \hline \hline
      $4-10,12,16$ &  $ \left(  \ \_ \_  \ , 3\cos(\pi/ \kappa) \right) \cup \left(2+\cos(\pi/ \kappa),3 \right)$ 
    \end{tabular}
  \end{center}
     \caption{$\boldsymbol{\mu} ^{\text{NUM}}_{A_{\kappa}}(\Delta) \cap [0,3]$ and $\boldsymbol{\mu} ^{\text{CONJ}}_{A_{\kappa}}(\Delta) \cap [0,3]$ in dimension 3, $\kappa_1 = \kappa_2 = \kappa_3$.}
         \label{tab:table1013127conj}
\end{table}

\begin{table}[H]
  \begin{tabular}{c|c|c|c|c|c}
    \multirow{2}{*}{} $\kappa$ &       
      \multicolumn{5}{c}{energies in $[0,3]$ for which a Mourre estimate holds for $\Delta$ wrt.\ $A_{\kappa}$, $d=3$} \\ [0.5em]
      \hline
    
     \footnotesize $(1,1,1)$  & \footnotesize $[0,1) \cup (1,3)$ & \footnotesize $(2,2,5)$  & \footnotesize $(2.8091,3)$ &  \footnotesize $(4,4,3)$  & \footnotesize $(2.7071,3)$ \\ [0.5em]
    \footnotesize  $(2,1,1)$  & \footnotesize $(0,0.8336) \cup (2,3)$ & \footnotesize $(2,2,6)$  & \tiny $(2.1444,2.3765) \cup (2.8660,3)$ & \footnotesize $(4,4,4)$  & \tiny $(2.0719,2.1213) \cup (2.7071,3)$ \\ [0.5em]
     \footnotesize $(3,1,1)$  & \footnotesize $(2.5,3)$ & \footnotesize $(3,3,1)$  & \footnotesize $(2.5,3)$ & \footnotesize $(4,4,5)$  & \footnotesize $(2.8091,3)$ \\ [0.5em]
     \footnotesize $(4,1,1)$  & \tiny $(1.2929,1.6127) \cup (2.7071,3)$ &  \footnotesize $(3,3,2)$  & \footnotesize $(2.5,3)$ & \footnotesize $(4,4,6)$  & \footnotesize $(2.8660,3)$ \\ [0.5em]
     \footnotesize $(5,1,1)$  & \tiny $(1.6910,1.9796) \cup (2.8091,3)$ & \footnotesize $(3,3,3)$  & \footnotesize $(2.5,3)$  & \footnotesize $(5,5,1)$  & \footnotesize $(2.8091,3)$ \\ [0.5em]
   \footnotesize   $(6,1,1)$  & \footnotesize $(2,2.2105) \cup (2.8660,3)$ & \footnotesize $(3,3,4)$  & \footnotesize $(2.7071,3)$  & \footnotesize $(5,5,2)$  & \footnotesize$(2.8091,3)$ \\ [0.5em]
    \footnotesize  $(2,2,1)$  & \footnotesize $(2,3)$ & \footnotesize $(3,3,5)$  & \footnotesize $(2.8091,3)$  & \footnotesize $(5,5,3)$  & \footnotesize $(2.8091,3)$ \\ [0.5em]
    \footnotesize  $(2,2,2)$  & \footnotesize $(2,3)$ & \footnotesize $(3,3,6)$  & \footnotesize $(2.8660,3)$  & \footnotesize $(5,5,4)$  & \footnotesize $(2.8091,3)$ \\ [0.5em]
   \footnotesize   $(2,2,3)$  & \footnotesize $(2.5,3)$ & \footnotesize $(4,4,1)$  & \footnotesize $(2.7071,3)$ &  \footnotesize $(5,5,5)$ &  \tiny $(2.3526,2.4270) \cup (2.8091,3)$  \\[0.5em]
   \footnotesize   $(2,2,4)$  & \footnotesize $(2.7071,3)$ & \footnotesize $(4,4,2)$  & \footnotesize $(2.7071,3)$ & \footnotesize  $(5,5,6)$  & \footnotesize $(2.8660,3)$ \\ [0.5em]

     \end{tabular}
    \caption{$\boldsymbol{\mu} ^{\text{NUM}}_{A_{\kappa}}(\Delta) \cap [0,3]$ in dimension $3$. }
    \label{tab:table1013127NEW}
\end{table}

\section{Appendix: Numerical evidence for Molchanov-Vainberg Laplacian}
\label{appendix_MV}

As in the previous section, the \textbf{horizontal single line} towards the top of the table separates the values of $\kappa$ for which we have proved the results rigorously from the values for which we don't have a rigorous proof. Below the \textbf{horizontal double line} towards the bottom of the table we have included conjectures on closed form formulas for the results, which we denote $\boldsymbol{\mu} ^{\text{CONJ}}_{\mathbb{A}}(\mathfrak{D})$.

\begin{table}[H]
  \begin{center}
    \begin{tabular}{c|c} 
      $\kappa$ & energies in $[0,1]$ for which a Mourre estimate holds for $D$ wrt.\ $A_{\kappa}$, $d=2$.   \\ [0.5em]
      \hline
      $2$  & $(0,1)$  \\ [0.5em]
      $4$ &  $(0,0.5) \cup ( 0.707107 , 1 )$  \\[0.5em]
      \hline
      $6$ &  $(0,0.25) \cup (0.50644,0.75) \cup (0.866026,1)$  \\[0.5em]
      $8$ &  $(0,0.146447) \cup (0.382684,0.5) \cup (0.712105 ,0.853554) \cup (0.92388,1)$  \\[0.5em]
   $10$ &   \footnotesize $(0,0.095492) \cup (0.310325 , 0.345492) \cup ( 0.587786 , 0.654509 ) \cup ( 0.81264 , 0.904509 ) \cup ( 0.951057 , 1 )$  \\[0.5em]
      $12$ &  $( 0 , 0.066988 ) \cup ( 0.707107 , 0.75) \cup ( 0.868705 , 0.933013 ) \cup ( 0.965926 , 1 )$  \\[0.5em]
      $14$ &  $(0,0.049516) \cup (0.781832,0.811745) \cup (0.903011,0.950485) \cup (0.974928,1)$  \\[0.5em]
      $16$ &  $(0,0.038061) \cup ( 0.83147 , 0.853554 ) \cup ( 0.92548 , 0.96194 ) \cup ( 0.980786 , 1 )$  \\[0.5em]
      $18$ &  $( 0 , 0.030154 ) \cup ( 0.866026 , 0.883023 ) \cup ( 0.940977 , 0.969847 ) \cup ( 0.984808 , 1 )$  \\[0.5em]
       \hline \hline
      $4$ &  $\left(0,\sin^2(\frac{\pi}{4}) \right) \cup \left( \sin(\frac{\pi}{4}) , 1 \right)$  \\[0.5em]
      $6$ &  $\left(0,\sin^2(\frac{\pi}{6}) \right) \cup \left( \ \_ (^*) \_ \ , \sin^2(\frac{2\pi}{6}) \right) \cup \left( \sin(\frac{2\pi}{6}),1 \right)$  \\[0.5em]
      $8$ &  $\left(0,\sin^2(\frac{\pi}{8}) \right) \cup \left( \sin(\frac{\pi}{8}),\sin^2(\frac{2\pi}{8}) \right) \cup \left( \ \_\_ \  ,\sin^2(\frac{3\pi}{8}) \right) \cup \left(\sin(\frac{3\pi}{8}),1 \right)$  \\[0.5em]
      $10$ & \small  $\left(0,\sin^2(\frac{\pi}{10}) \right) \cup \left( \ \_\_ \ , \sin^2(\frac{2\pi}{10}) \right) \cup \left(\sin(\frac{2\pi}{10}),\sin^2(\frac{3\pi}{10}) \right) \cup \left( \ \_\_ \ , \sin^2(\frac{4\pi}{10}) \right) \cup \left(\sin(\frac{4\pi}{10}),1 \right)$  \\[0.5em]
      $12-18$ & \small $\left( 0 , \sin^2(\frac{\pi}{\kappa}) \right) \cup \left(\sin (\frac{(\frac{\kappa}{2}-3)\pi}{\kappa}) , \sin^2(\frac{(\frac{\kappa}{2}-2)\pi}{\kappa}) \right) \cup \left( \ \_\_ \ , \sin^2(\frac{(\frac{\kappa}{2}-1)\pi}{\kappa} )\right) \cup \left( \sin(\frac{(\frac{\kappa}{2}-1)\pi}{\kappa}) , 1 \right)$  \\[0.5em]
    \end{tabular}
  \end{center}
    \caption{$\boldsymbol{\mu} ^{\text{NUM}}_{A_{\kappa}}(D) \cap [0,1]$ and $\boldsymbol{\mu} ^{\text{CONJ}}_{A_{\kappa}}(D) \cap [0,1]$ in dimension 2, $\kappa_1 = \kappa_2$.}
        \label{tab:table1012}
\end{table}

In Table \ref{tab:table1012}, $\kappa = 6$ we do not know a closed form solution for the missing value $(^*)$. But it appears that for $E \in [0.4,0.6]$ the function $g_E(E_1, E/E_1)$ has a global maximum attained at 
$$E_1 = h(E):= \frac{1}{2} \sqrt{\frac{1}{6}\left( 8+ \sqrt{144 E^2 +7} +\sqrt{16 \sqrt{144 E^2+7} + 71- 432E^2} \right)}.$$
The value $(^*)$ is therefore the root of $g_E(h(E),E/h(E))$, which can be estimated numerically to high accuracy.

\begin{table}[H]
  \begin{tabular}{c|c|c|c}
    \multirow{2}{*}{} $\kappa$ &       
      \multicolumn{3}{c}{energies in $[0,1]$ for which a Mourre estimate holds for $D$ wrt.\ $A_{\kappa}$, $d=2$} \\ [0.5em]
      \hline
      $(2,4)$  & $(0.707107,1)$ &  $(4,8)$ &  $( 0 , 0.246095) \cup ( 0.923880 , 1 )$  \\[0.5em]
      $(2,6)$ &  $(0,0.445133) \cup ( 0.866025 , 1 )$  & $(4,10)$ &  $(0.309017,0.379323) \cup (0.951057,1)$  \\[0.5em]
      $(2,8)$ & \small $ (0.382683, 0.647959) \cup (0.923880,1)$  & $(6,8)$ & \small  $( 0.710171, 0.797581) \cup ( 0.923880 , 1 )$  \\[0.5em]
      $(2,10)$ & \small  $ (0.587786,0.750000) \cup (0.951056,1)$  & $(6,10)$ & \footnotesize $( 0, 0.145934) \cup ( 0.809728, 0.818608 ) \cup ( 0.951057 , 1 )$  \\[0.5em]
      $(4,6)$ &  \small  $( 0.502625 , 0.605938 ) \cup ( 0.866025 , 1 )$  & $(8,10)$ & \small   $( 0.811615, 0.877563) \cup ( 0.951057 , 1 )$  \\[0.5em]

     \end{tabular}
    \caption{$\boldsymbol{\mu} ^{\text{NUM}}_{A_{\kappa}}(D) \cap [0,1]$ in dimension 2, $\kappa = (\kappa_1, \kappa_2)$, $\kappa_1 \neq \kappa_2$.}
    \label{tab:table101244}
\end{table}

 There appears to be no strict positivity for $D$ in dimension 2 wrt.\ $A_{\kappa}$ for $(\kappa_1, \kappa_2) = (1,2)$, $(1,3)$, $(1,4)$, $(1,5)$, $(2,3)$, $(2,4)$, $(2,5)$, $(3,4)$, $(3,5)$, $(4,5)$, $(5,5)$.
 
\begin{table}[H]
  \begin{center}
    \begin{tabular}{c|c} 
      $\kappa$ & energies in $[0,1]$ for which a Mourre estimate holds for $D$ wrt.\ $A_{\kappa}$, $d=3$. \\ [0.5em]
      \hline
      $2$  & $(0,1)$  \\ [0.5em]
      \hline
      $4$ &  $(0,0.3535) \cup (0.7071,1)$  \\[0.5em]
      $6$ &  $(0,0.125) \cup (0.5148,0.6495) \cup (0.8660,1)$  \\[0.5em]
      $8$ &  $(0, 0.0560) \cup (0.7187,0.78858) \cup (0.9238,1)$  \\[0.5em]
      $10$ &  $(0,0.029508) \cup (0.81751,0.8602) \cup (0.951056,1)$  \\[0.5em]
      $12$ &  $(0,0.01734) \cup (0.87235,0.9012) \cup (0.965925,1)$ \\[0.5em]
      $14$ &  $ (0,0.0110)\cup (0.9058,0.9266)  \cup (0.9749,1)$ \\[0.5em]
      $16$ &  $(0, 0.007425) \cup (0.927666 , 0.943456) \cup (0.980785, 1) $ \\[0.5em]
      $18$ &  $(0,0.0052) \cup (0.9428,0.9551) \cup (0.9848,1)$ \\[0.5em]
      $20$ &  $(0,0.0038) \cup (0.9536,0.9635) \cup (0.9877,1)$ \\[0.5em]
       $24$ &  $(0,0.0022) \cup (0.9677,0.9746) \cup (0.9914,1)$ \\[0.5em]
       \hline \hline
        $4$ &  $\left(0,\sin^3 \left( \frac{\pi}{4} \right) \right) \cup \left(\sin \left(\frac{\pi}{4}\right),1 \right)$  \\[0.5em]
      $6-24$ &  $\left(0, \sin^3 \left(\frac{\pi}{\kappa} \right) \right) \cup \left( \ \_\_ \  , \sin^3 \left(\frac{(\frac{\kappa}{2}-1)\pi}{\kappa}\right)  \right) \cup \left(\sin \left(\frac{(\frac{\kappa}{2}-1)\pi}{\kappa}\right) ,1\right)$  \\[0.5em]
    \end{tabular}
  \end{center}
\caption{$\boldsymbol{\mu} ^{\text{NUM}}_{A_{\kappa}}(D) \cap [0,1]$ and $\boldsymbol{\mu} ^{\text{CONJ}}_{A_{\kappa}}(D) \cap [0,1]$ in dimension 3, $\kappa_1 = \kappa_2 = \kappa_3$.}
    \label{tab:table10144}
\end{table}

\begin{table}[H]
  \begin{tabular}{c|c|c|c}
    \multirow{2}{*}{} $\kappa$ &       
      \multicolumn{3}{c}{energies in $[0,1]$ for which a Mourre estimate holds for $D$ wrt.\ $A_{\kappa}$, $d=3$} \\ [0.5em]
      \hline
      $(2,2,1)$  & $\emptyset$  & $(4,4,6)$ &  $(0.8660,1)$  \\[0.5em]
      $(2,2,2)$  & $(0,1)$  & $(4,4,8)$ &  $(0,0.1737) \cup (0.9238,1)$  \\[0.5em]
      $(2,2,3)$  & $\emptyset$  &  $(6,6,2)$ &  $(0.8660,1)$  \\[0.5em]
      $(2,2,4)$  & $(0.7071,1)$ & $(6,6,4)$ &  $(0.5115,0.5205) \cup (0.8660,1)$  \\[0.5em]
      $(2,2,5)$  & $\emptyset$  & $(6,6,6)$ &  $(0,0.1250) \cup (0.5148,0.6495) \cup (0.8660,1)$  \\[0.5em]
      $(2,2,6)$  & $(0,0.4451) \cup (0.8660,1)$  & $(6,6,8)$ &  $(0.9238,1)$  \\[0.5em]
      $(2,2,7)$  & $\emptyset$ & $(8,8,2)$ &  $(0.9238,1)$  \\[0.5em]
      $(2,2,8)$  & $(0.3827,0.6479) \cup (0.9238,1)$ & $(8,8,4)$ &  $(0,0.0743) \cup (0.9238,1)$  \\[0.5em]
       $(4,4,2)$  & $(0.7071,1)$ & $(8,8,6)$ &  $(0.7170,0.7349) \cup (0.9238,1)$  \\[0.5em]
      $(4,4,4)$ &  $(0,0.3535) \cup (0.7071,1)$ & $(8,8,8)$ &  $(0, 0.0560) \cup (0.7187,0.78858) \cup (0.9238,1)$  \\[0.5em]

     \end{tabular}
    \caption{$\boldsymbol{\mu} ^{\text{NUM}}_{A_{\kappa}}(D) \cap [0,1]$ in dimension 3}
    \label{tab:table10144NEW}
\end{table}

\end{document}